\documentclass[11pt,a4paper]{amsart}

\usepackage[utf8]{inputenc}
\usepackage{mathrsfs}
\usepackage{xypic}

\usepackage{graphicx}
\usepackage{hyperref}
\usepackage{amssymb}
\usepackage{amsmath}
\usepackage{verbatim,enumerate}
\usepackage{tikz}
\usepackage{tikz-cd}
\usepackage{array}
\usepackage{multirow}
\usepackage{hhline}

\numberwithin{equation}{section}
\numberwithin{figure}{section}

\newtheorem{thm}{Theorem}[section]
\newtheorem{lemma}[thm]{Lemma}
\newtheorem{prop}[thm]{Proposition}
\newtheorem{cor}[thm]{Corollary}

\newtheorem{const}[thm]{Construction}

\newtheorem{defi}[thm]{Definition}
\newtheorem{nota}[thm]{Notation}
\newtheorem{hyp}[thm]{Hypothesis}
\theoremstyle{definition}
\newtheorem{exa}[thm]{Example}
\newtheorem{rem}[thm]{Remark}

\newcommand{\C}{\mathbb{C}}
\newcommand{\A}{\mathbb{A}}

\newcommand{\M}{\mathcal{M}}

\renewcommand{\P}{\mathbb{P}}

\renewcommand{\epsilon}{\varepsilon}
\newcommand{\R}{\mathbb{R}}
\newcommand{\Q}{\mathbb{Q}}
\newcommand{\Z}{\mathbb{Z}}

\newcommand{\D}{\mathcal{D}}

\newcommand{\sC}{\mathscr{C}}
\newcommand{\sB}{\mathscr{B}}

\newcommand{\cX}{\mathcal{X}}
\newcommand{\cO}{\mathcal{O}}

\renewcommand{\geq}{\geqslant}
\renewcommand{\ge}{\geqslant}
\renewcommand{\leq}{\leqslant}
\renewcommand{\le}{\leqslant}

\newcommand\angles[1]{\langle #1 \rangle}
\newcommand{\Bl}{\operatorname{Bl}}
\newcommand{\Mbar}{\overline{\mathcal{M}}}
\newcommand{\Spec}{\operatorname{Spec}}
\newcommand{\Ob}{\operatorname{Ob}}
\newcommand{\Stab}{\operatorname{Stab}}
\newcommand{\Def}{\operatorname{Def}}
\newcommand{\oH}{\operatorname{H}}

\newcommand{\Proj}{\operatorname{Proj}}
\newcommand{\Gal}{\operatorname{Gal}}
\newcommand{\GL}{\operatorname{GL}}

\newcommand{\Pic}{\operatorname{Pic}}

\newcommand{\Set}{\operatorname{Set}}

\newcommand{\bir}{\operatorname{bir}}
\newcommand{\mass}{\operatorname{mass}}
\newcommand{\ev}{\operatorname{ev}}

\newcommand{\Int}{\operatorname{Int}}
\newcommand{\Rul}{\operatorname{Rul}}

\newcommand{\Aut}{\operatorname{Aut}}

\newcommand{\Hom}{\operatorname{Hom}}

\newcommand{\Disc}{\operatorname{Disc}}
\newcommand{\Tr}{\operatorname{Tr}}

\newcommand{\Res}{\operatorname{Res}}
\newcommand{\Div}{\operatorname{div}}
\newcommand{\rk}{\operatorname{rk}}
\newcommand{\wt}{\operatorname{weight}}

\newcommand{\GW}{\operatorname{GW}}
\newcommand{\sGW}{\mathcal{GW}}

\newcommand{\lra}[1]{\langle #1 \rangle}
\newcommand{\Fet}{\mathrm{F\acute{e}t}}
\newcommand{\et}{\mathrm{\acute{e}t}}
\newcommand{\odp}{\mathrm{odp}}
\newcommand{\unr}{\mathrm{unr}}

\begin{document}

\title{A quadratic Abramovich-Bertram formula}

\author{Erwan Brugallé}
\address{Erwan Brugallé, Nantes Université, Laboratoire de
 Mathématiques Jean Leray, 2 rue de la Houssinière, F-44322 Nantes Cedex 3,
France}
\email{erwan.brugalle@math.cnrs.fr}

\author{Kirsten Wickelgren}
\address{Kirsten Wickelgren, Department of Mathematics, Duke University, 120 Science Drive
Room 117 Physics, Box 90320, Durham, NC 27708-0320, USA}
\email{kirsten.wickelgren@duke.edu}


\date{July 2025, revised May 2026}

\begin{abstract}
Quadratic Gromov--Witten invariants allow one to obtain an arithmetically meaningful count of curves satisfying constraints over a field $k$ without assuming that $k$ is the field of complex or real numbers. This paper studies the behavior of quadratic genus $0$ Gromov--Witten invariants during an algebraic analogue of surgery on del Pezzo surfaces. For this, we define and study (twisted) binomial coefficients in the Grothendieck--Witt group, building on work of Serre. We obtain a formula expressing the quadratic genus $0$ Gromov--Witten invariants of surfaces obtained as a smoothing of a given nodal surface in terms of those of the one having the largest Picard group. We give applications to quadratic Gromov--Witten invariants of rational del Pezzo surfaces of degree at least 7, some cubic surfaces, 
for point constraints defined over quadratic extensions of $k$, as well as an invariance result under a Dehn twist.

\end{abstract}

\maketitle
\tableofcontents

\section{Introduction}

This paper studies the behavior of quadratic genus $0$ Gromov--Witten invariants during an algebraic analogue of surgery on del Pezzo surfaces. Quadratic Gromov--Witten invariants \cite{degree} allow one to obtain an arithmetically meaningful count of curves satisfying constraints over a field $k$ with some additional hypotheses, but without assuming that $k$ is the field of complex or real numbers. Curves are weighted by elements of the Grothendieck--Witt group $\GW(k)$, recalled in Section~\ref{section:BackgroundGW}, to obtain an invariant sum in $\GW(k)$, partially generalizing Welschinger's beautiful invariants over the real numbers \cite{Welschinger-invtsReal4mflds,Welopendim4}. See Section \ref{section:BackgroundA1Gromov-Witten} for discussion of quadratic genus $0$ Gromov--Witten invariants. See also \cite{Levine-Welschinger} for a version valued in $\sGW(\textrm{Sym}^n S \setminus \Delta) \ncong \GW(k)$. We also refer to \cite{IKS17} for an algebro-geometric proof that Welschinger invariants are well-defined for real del Pezzo surfaces, and to \cite{Bru18} for a strengthening of the original invariance statement by Welschinger.

A del Pezzo surface over $k$ is a non-singular, projective algebraic surface $S$ over $k$ with $-K_S$ ample, where $K_S$ is the canonical divisor of $S$. If $k$ is algebraically closed, such a del Pezzo surface is either $\P^1_k\times \P^1_k$  or $\P^2_k$ blown-up at $r\le 8$ closed points in general position. Without the assumption that $k$ is algebraically closed, interesting twists arise. For example the quadric $Q(d) = \{x^2 - y^2 + z^2 -d w^2 = 0 \}$ in $\P^3_k$ with $d$ in $k^* $ is a del Pezzo surface, which  is isomorphic to the restriction of scalars $\Res_{k[\sqrt{d}]/k} \P^1$. In particular $Q(d_1)$ and $Q(d_2)$ are isomorphic if and only if $d_1$ and $d_2$ have the same class in $k^* \setminus (k^*)^2$, and $Q(1)=\P^1_k\times \P^1_k$. When $k=\R$, one obtains different topological spaces $Q(d)(\R)$
depending on the sign of $d$, which differ by a surgery where  a $[0,1] \times S^1$ replaced by an $S^0 \times B^2$. See Figure \ref{fig:surgery} for a picture in the affine chart $w=1$.
\begin{figure}[h!]
  \begin{center}
  \begin{tabular}{ccc}
  \includegraphics[width=2cm, angle=0]{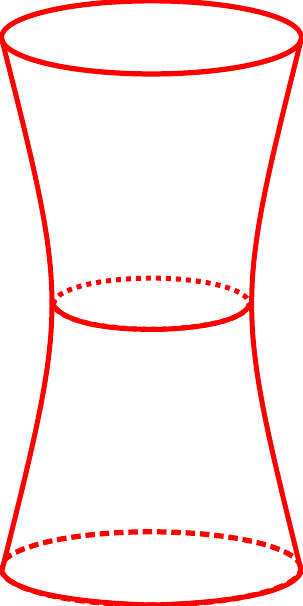}
  & 
  \hspace{8ex} &
  \includegraphics[width=2cm, angle=0]{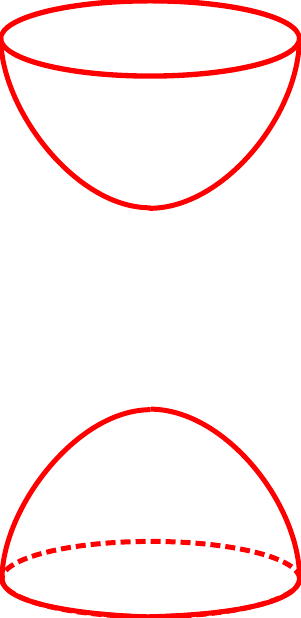}
  \\ \\ $Q(1)(\R)$ && $Q(-1)(\R)$
  \end{tabular}
  \end{center}
  \caption{The real parts of $Q(1)$ and $Q(-1)$ differ by a $1$-surgery}
  \label{fig:surgery}
  \end{figure}
  All these different quadrics can be degenerated to the nodal quadric $Q(0)$, i.e. can be incorporated in the family of quadrics $ \{x^2 - y^2 + z^2 -t w^2 = 0 \}$ over $k[[t]]$. Our strategy, which  goes back to the work  \cite{AbramovichBertram} over $\C$, is then roughly to relate the enumerative geometry of $Q(d)$, with $d\in k^*$, to the enumerative geometry of $Q(0)$, or rather of its desingularization given by the second Hirzebruch surface $\mathbb F_2$. 
  Eliminating terms coming from $\mathbb F_2$, one then obtains surprisingly simple relations among quadratic invariants of the $Q(d)$'s for different $d$'s.

Another nodal degeneration of del Pezzo surfaces arises when a configuration of $r$ points on $\P^2_k$ in general position moves into a special configuration, such as when $3$ distinct points lie on a line or $6$ points lie on a conic. Blowing up the special configuration of $r$ points produces a uninodal del Pezzo surface $\widetilde{S}$, which has a unique rational curve $E$ with self-intersection $-2$ crushed by the anticanonical embedding to produce a unique node. (See Definition \ref{df:uninodal_del_Pezzo}.) We obtain a family where the general fiber is a del Pezzo surface and the special fiber is a uninodal del Pezzo surface. We define a $1$-nodal Lefschetz fibration $\cX \to \Spec k[[t]]$ of del Pezzo surfaces (see Definition \ref{df:1-nodal_Lefschetz_fibration}) to describe this sort of degeneration. Pulling back $\cX \to \Spec k[[t]]$ by $t \mapsto d' t^2$ for $d' \in k^*$(and, respectively, blowing up the node) produces general fibers  (and families) which differ by some analogue of surgery in the following sense. The special fiber of the family has two components $\widetilde S$ and $Q(d)$, see Section \ref{sec:sugery} for a precise relation between $d'$ and $d$. They are glued along $E$ in $\widetilde S$ and a smooth hyperplane section in $Q(d)$. Over $\mathbb{C}$, one can consider an analogous family over a small disk. The homeomorphism type of the general fiber is then the connected sum of $\widetilde{S}(\C)$ and $Q(d)(\C)$, obtained by removing a neighborhood of $E$ and a neighborhood of the hyperplane section and then gluing together along the boundary. Over $\R$ this connected sum can be done equivariantly, and one then obtains the general fiber for different $d'$ by the surgery relating the different $Q(d)(\R)$'s.

In the algebraic setting, we define a $d$-surgery $\cX(d) \to \Spec k[[t]]$ of a $1$-nodal Lefschetz fibration $\cX \to \Spec k[[t]]$ as the aforementioned family, where $d'$ is so that the special fiber $\cX(d)_0$ of $\cX(d)$ has  $\widetilde S$ and $Q(d)$ as irreducible components. We let $\Sigma(d)$ denote the general fiber. See Construction \ref{const:d-surgery}. Our main result expresses the quadratic genus $0$ Gromov--Witten invariants of $\Sigma(d)$ from those of $\Sigma(1)$. 

A  $1$-nodal Lefschetz fibration $\cX \to \Spec k[[t]]$ of del Pezzo surfaces has a vanishing cycle $\gamma\in \Pic(\Sigma(1))$ (see Definition~\ref{df:vanishing_cycle}), and there is a canonical injection $\Pic \Sigma(d) \hookrightarrow \Pic \Sigma(1)$ identifying $\Pic(\Sigma(d))$ with $\gamma^\perp$ when $d\notin k^2$. (See Corollary~\ref{cor:mapPicSigma_to_PicS}.) Let $\cX_0 \to \Spec k$ denote the special fiber. A special case of our main result is as follows. Given an algebraic surface $S$, we denote by $K_S$ its canonical divisor. Given a finite  extension $k\to \sigma$, note that $\sigma\otimes_k k((t))$ is canonically isomorphic to $\sigma((t))$.

\begin{thm}\label{thm:N-surgery-intro}
Let $\cX \to \Spec k[[t]]$ be a $1$-nodal Lefschetz fibration of del Pezzo surfaces of degree at least 4. Suppose $k$ is a characteristic $0$ field, and $\Sigma_1$ is $k((t))$-rational and $\cX_0$ is $k$-rational. Then for all $D$ in $\Pic \Sigma(d)$,  and all finite étale extensions $k\to\sigma$ of degree $-K_{\Sigma(d)}\cdot D-1$ we have
  \begin{align*}
  N_{\Sigma(d),D,\sigma((t)) }
  & = N_{\Sigma(1),D,\sigma((t)) }+
 (\langle 2 \rangle - \langle 2d \rangle) \sum_{j\ge 1} (-1)^j N_{\Sigma(1),D-j\gamma,\sigma((t)) } .
\end{align*}
\end{thm}

Some comments are appropriate before outlining the proof of Theorem \ref{thm:N-surgery-intro}. The hypothesis that $\Sigma_1$ is rational implies that $\Sigma_1$ is $\A^1$-connected. This connectivity is used in \cite{degree} to show that the quadratic Gromov--Witten invariant is well-defined. Implicit in the statement of Theorem~\ref{thm:N-surgery-intro} is that the left hand side $ N_{\Sigma_d,D,\sigma}$ is well-defined in an appropriate sense, even though $\Sigma_d$ is not assumed to be $k$-rational or $\A^1$-connected. See Remark~\ref{rem:independence_tildep_woassuming_connectivity}. 

It is not necessary to assume that $k$ has characteristic $0$. More general results are contained in Theorem~\ref{thm:N-surgery} and Corollary~\ref{cor:N-surgery-5}.

  The term $\langle 2 \rangle - \langle 2d \rangle$ in Theorems \ref{thm:N-surgery-intro} and \ref{thm:N-surgery} is the difference of the $\A^1$-Euler characteristics \cite{Hoyois_lef} \cite{Levine-EC} \cite{AMBOWZ22} of $\Sigma(1)$ and $\Sigma(d)$ up to a factor of $\langle -1 \rangle$, at least in characteristic $0$. See Proposition~\ref{prop:differenceA1ChiLFdP} and Remark~\ref{rem:Denef_Loeser}. 

When $k=\R$, Theorems \ref{thm:N-surgery-intro} and \ref{thm:N-surgery} are a particular case of \cite[Theorem 2.1]{Bru18}. As said above the strategy  used in this paper goes back to the work  \cite{AbramovichBertram} over $\C$, where Abramovich and Bertram relate some enumerative invariants  of $\P^1_\C\times \P^1_\C$ and $\mathbb F_2$ using a family analogous to those described above. It has been realized later that a ``semi-stable degeneration" interpretation of the Abramovich-Bertram method could provide successful computations of Welschinger invariants of real symplectic $4$-manifolds \cite{BruPui15}. The real version of Theorems \ref{thm:N-surgery-intro} and \ref{thm:N-surgery} eventually appeared in \cite{Bru18}, generalizing earlier versions from \cite{IKS15,Bru16}. 

A motto of Theorem \ref{thm:N-surgery} can be summarized: \emph{quadratic genus $0$ Gromov--Witten invariants of a rational surface $S$ reduce to the twists of $S$ for which the action of $\Gal(\overline k/k)$ on the Picard group is trivial}. This suggests a stronger invariance property of quadratic invariants than the one currently known, generalizing \cite[Theorem 1.3 and Remark 1.4]{Bru18}, and which constitutes one of the roots of the forthcoming paper \cite{BruRauWic25}. 

\medskip
The proof of Theorem~\ref{thm:N-surgery} counts curves on the general fiber by showing the count is equal to the count of the specializations. Let us assume for the moment that $k$ is algebraically closed. The special fiber is isomorphic to $\widetilde S \cup_E Q(d)$, with $\widetilde S$ independent of $d$. One shows that, as in the case over $\C$, the specializations are reducible curves with one component $f_{\widetilde S}: C \to \widetilde S$ on $\widetilde S$ glued to components on $Q(d)$ at the points of intersection of $f_{\widetilde S}(C)$ with $E$. These curves on $Q(d)$ are either of bidegree $(0,1)$ or $(1,0)$. Let $I$ denote the number of points where $f_{\widetilde S}(C)$ and $E$ intersect. In order for the reducible curve on $\widetilde S \cup_E Q(d)$ to be the specialization of a curve on the general fiber of the correct class in the Picard group, there will be some fixed $b$, with $0 \leq b \leq I$, such that exactly $b$ of the components on $Q(d)$ have bidegree $(0,1)$ and the remaining $I-b$ components on $Q(d)$ have bidegree $(1,0)$. It follows that for the same  $f_{\widetilde S}$, one obtains exactly ${I \choose b}$ reducible curves on the special fiber of the correct form by making different choices of the $b$ points out of $I$. 

To make the corresponding count for quadratic Gromov--Witten invariants when $k$ is not algebraically closed, we must remember that the intersection of $f_{\widetilde S}(C)$ and $E$ is some $k$-algebra of degree $I$. We need to appropriately count the number of reducible curves with component $f_{\widetilde S}$ by replacing ${I \choose b}$ by the appropriate count of the number of ways to choose $b$ of the intersection points when this intersection is a finite \'etale $k$-algebra of degree $I$.

To accomplish this, this paper introduces new combinatorial techniques in $\GW(k)$. We define (twisted) binomial coefficients in $\GW(k)$ in Section~\ref{Section:twisted_binomial_coefficients_section}. This defines ${A/k \choose b}$ for $A$ a finite \'etale $k$-algebra as well as a twist appearing in our context, generalizing operations defined by Serre \cite[30.12-30.14, Appendix B]{Garibaldi-Serre-Merkurjev}. Certain combinatorial identities have an enriched version, including Pascal's identity (see Proposition~\ref{prop:basic_binomial_identities_scheme}). The analogue of Pascal's triangle over a finite field is computed in \cite{Chen-Wickelgren} by  Chen and the second named author. In Section \ref{sec:binomial identity} we prove an enriched version of the binomial identity used in the proof of \cite[Theorem 2.1]{Bru18}. 
These enriched combinatorial identities allow one to completely eliminate terms from curve counting on $\widetilde S$. This cancelation produces the simple wall-crossing formula of Theorem~\ref{thm:N-surgery-intro} relating quadratic genus $0$ Gromov--Witten invariants of $\Sigma_d$'s for different $d$'s. 

\medskip
In Section~\ref{sec:applications} we give several applications. In Section \ref{sec:ex quad} and \ref{sec:ex bup}, we reduce the computation of genus $0$ quadratic Gromov--Witten invariants of rational del Pezzo surfaces of degree at least 7 to the case of toric del Pezzo surfaces. Note that these toric cases have been covered by Jaramillo Puentes, Markwig, Pauli, and R\"ohrle  \cite{PauliPuentes23,JPMPR25} for multiquadratic extensions $\sigma$. We combine our results with theirs. Blowing up at a colliding pair of points reduces point constraints associated to degree $2$ field extensions to invariants of the blow-up without the degree $2$ extension point condition. We give some applications to cubic surfaces in Section~\ref{section:enumerative_cubic_surfaces}.
We end by showing in Section \ref{sec:dehn} that  $N_{\Sigma_1,D,\sigma}=N_{\Sigma_1,D+(D\cdot \gamma)\gamma,\sigma}$, generalizing the invariance of Gromov-Witten and Welschinger invariants under (real) Dehn twists. A conjectural expression of genus 0 quadratic Gromov--Witten invariants of all rational surfaces is proposed in \cite{BruRauWic25}. We also prove there this conjecture in the case of rational del Pezzo surfaces of degree at least 6, using in particular Theorem \ref{thm:N-surgery-intro} as a key ingredient.

\subsection*{Acknowledgements} We warmly thank Thomas Dedieu, Johannes Rau, Burt Totaro, Ilya Tyomkin, and Ravi Vakil for useful discussions. We thank the organizers of the Mathematisches Forschungsinstitut Oberwolfach workshop {\em Tropical Methods in Geometry} for introducing us and providing a stimulating environment. We thank Andr\'es Jaramillo Puentes for helpful comments on a preliminary draft. EB was partially supported by l’Agence Nationale de la Recherche (ANR), project ANR-22-CE40-0014. KW was partially supported by National Science Foundation Awards DMS-2103838 and DMS-2405191. KW thanks Duke University for supporting research travel with family. For the purpose of open access, the authors apply a CC-BY public copyright licence to any Author Accepted Manuscript (AAM) version arising from this submission.

\section{Background}\label{sec:Background}

\subsection{Notation}
We will denote by $k$ a field of characteristic not 2.
A \emph{del Pezzo surface} is a non-singular algebraic surface $S$ over a field $k$ with $-K_S$ ample, where $K_S$ is the canonical divisor of $S$. The degree of $S$ is defined as $K_S^2$. If $k$ is algebraically closed, such del Pezzo surface is either $\P^1_k\times \P^1_k$  or $\P^2_k$ blown-up at $r\le 8$ points in general position, see for example \cite[Chapter 8]{Dolg12} . In the former case $d_S=8$, while $d_S=9-r$ in the latter case.
We denote by $\Pic(S)$ the group of invertible sheaves on $S$ under tensor product. Note that it is the $k$-points of the Picard scheme of $S$.

Let $k$ be a field and let $k \to L$ be a field extension. Let $X$ be a $k$-scheme, $D\in \Pic(X)$ and let $\sigma$ be an \'etale $k$-algebra. Then $S_L = S \times_{\Spec k} \Spec L$ denotes the basechange, $D_L$ denotes the image of $D$ under the pullback map $\Pic(X) \to \Pic (X_L)$ and $\sigma_L := \sigma \otimes_k L$.

Given a field $k$, we denote by $k^s$ its separable closure, and $\overline k$ its algebraic closure.

\subsection{Grothendieck--Witt group}\label{section:BackgroundGW}

Let $k$ be a field. Let $\GW(k)$ denote the Grothendieck--Witt ring of $k$, defined to be the group completion of the semi-ring of isomorphism classes of symmetric non-degenerate bilinear forms over $k$. See for example \cite{milnor73} for more information. For $a$ in $k^*$, let $\langle a \rangle$ denote the class in $\GW(k)$ of the bilinear form $k \times k \to k$ given by $(x,y) \mapsto a xy$. For a finite separable extension $k \subseteq E$, let $\Tr_{E/k}: \GW(E) \to \GW(k)$ denote the abelian group homomorphism taking the class of a nondegenerate symmetric bilinear form $\beta: V \times V \to E$ to the the class of $V \times V \to E \stackrel{\Tr_{E/k}}{\longrightarrow} k$, where the map $\Tr_{E/k}: E \to k$ is the trace from Galois theory.

\subsection{Genus 0 quadratic Gromov--Witten invariants}\label{section:BackgroundA1Gromov-Witten}

The works \cite{degree}\cite{KLSW-relor} define quadratic Gromov--Witten invariants in $\GW(k)$ for smooth del Pezzo surfaces over a field $k$ under the following hypotheses. Let $S$ be a smooth del Pezzo surface over $k$ and let $D$ be an effective Cartier divisor on $S$. There is a moduli stack $M_{0,n}(S,D)$ parametrizing stable, $n$-pointed, degree $D$ maps from a smooth genus $0$ curve to $S$ with a compactification $\Mbar_{0,n}(S,D)$ parametrizing stable, $n$-pointed, degree $D$ maps from genus $0$ curves allowed to have nodal singularities to $S$ \cite{Abramovich--Oort-mixed_char}. A geometric point of $M_{0,n}(S,D)$ is the class of a map $f: \P^1 \to S$ over an algebraically closed field together with a choice of $n$ distinct rational points $p_1,\ldots, p_n$ of $\P^1$. Let $M^{\bir}_0(S, D) \subset M_{0,n}(S,D)$ denote the open subsheme with geometric points $[(f: \P^1 \to S, p_1,\ldots,p_n)]$ such that $\P^1 \stackrel{f}{\to} f(\P^1)$ is birational. See \cite[Section 2]{KLSW-relor}.

\begin{hyp}\label{hyp:SDk_with_N}
 Assume $k$ is perfect of characteristic not $2$ or~$3$.  Let $S$ be a smooth del Pezzo over $k$ and $D$ a Cartier divisor on $S$. Moreover, assume that $d_S\ge 4$, or $d_S=3$ and $n:= -K_S \cdot D-1\neq 5$, or $d_S = 2$ and $n\ge 6$. If $k$ is positive characteristic and $d_S=2$, assume additionally that for every effective $D' \in Pic(S)$, there is a geometric point $f$ in each irreducible component of $M^{\bir}_0(S, D')$ with $f$ unramified. 
\end{hyp}

Let $(S,D,k)$ be as in Hypothesis~\ref{hyp:SDk_with_N}. Let $\sigma = \prod_{i=1}^r L_i $ be a finite \'etale product of field extensions $L_i$ of $k$ with $\sum_{i=1}^n [L_i: k] = -K_S \cdot D -1$. Assume additionally that $S$ is  $\mathbb{A}^1$-connected.

\begin{defi}
Let $N_{S,D, \sigma}$ in $\GW(k)$ denote the quadratic invariants of \cite[Theorems 1 and 2]{degree}.
\end{defi}

\begin{rem}
In \cite[Theorems 1 and 2]{degree}, the class $D$ is assumed not to be an $m$-fold multiple of a $-1$-curve for $m>1$. However, in this case, it is consistent to define the invariants to be $0$, which we do here.
\end{rem}

The invariants $N_{S,D, \sigma}$ have the following enumerative meaning. Let $u : \mathbb{P}_{k(u)} \to S_{k(u)}$ be a rational curve defined over some field extension $k(u)$ of $k$. For every node $x$ of $u(\mathbb{P}_{k(u)})$, the two tangent directions at $x$ define a (degree at most two) extension $k(x) \subset k(x)[\sqrt{ \delta(x)}]$, with $\delta(x) \in k(x)^*/(k(x)^*)^2$. Then define the {\em mass} of the node $x$ in $\GW(k(u))$ by
\[
 \operatorname{mass}(x) = \langle N_{k(x)/k(u)} \delta(x) \rangle
\] where $N_{k(x)/k(u)}: k(x)^* \to k(u)^*$ denotes the norm from Galois theory. For more information, see \cite[Definition 1.1]{degree}. 

We use the notation $\Res$ to denote Weil restriction or restriction of scalars. See for example \cite[Section 4.5]{Poonen-rational_points_on_varieties}. Consider $\Res_{\sigma/k} S \cong \prod_{i=1}^r \Res_{L_i/k} S$. For any finite, separable extension $F$ of $k$, the tensor product $F_i=L_i \otimes_k F$ is a finite product of field extensions. Then $\sigma_F = \prod_{i=1}^j F_i$ is the product of all these field extensions. By definition of the restriction of scalars, there is a canonical bijection between the $F$-points $p$ of $\Res_{\sigma/k} S$ and points $(p_1,\ldots, p_j)$ of $S$ with residue fields $(F_1,\ldots, F_j)$. Consider an element of $\GW(k)$ as determining an element of $\GW(F)$ by the pullback map $\GW(k) \to \GW(F)$ corresponding to the extension $k \to F$.

The enumerative meaning of the invariants $N_{S,D, \sigma}$ under the hypotheses given in their definition is the following: there is a dense open subset $U$ of $\prod_{i=1}^r \Res_{L_i/k} S $ with the property that for any point $p$ of $U$ we have the equality
\begin{equation}\label{eq:Nenum}
N_{S,D, \sigma} = \sum_{\substack{u \text{ degree } D \\ \text{ rational curve on } S \\\text{ thru the points} \\ p_1, \ldots, p_j}} \Tr_{k(u)/k(p)}\prod_{x \text{ node}\text{ of }u} \mass(x)
\end{equation} in $\GW(k(p))$ where $p_1,\ldots, p_j$ denote the points of $S$ with residue fields $\sigma_{k(p)}$ corresponding to $p$. See \cite[Sections 1.2.4 and 8]{degree}. We will denote $\Tr_{k(u)/k(p)}\prod_{x \text{ node}\text{ of }u} \mass(x)$ by $\wt_{k(p)}(u)$ and consider $N_{S,D, \sigma}$ as an enriched count of rational curves where each rational curve contributes its weight to the count.

We show invariance under basechange
\[
N_{X_F,D_F,\sigma_F} = N_{X,D,\sigma} \otimes_{k} F \in \GW(F).
\] of the quadratic $\A^1$ Gromov--Witten invariants in Proposition~\ref{pr:NA1-stable-basechange-at-point} under suitable hypotheses.

\begin{rem}\label{rem:non perfect field}
  This invariance under base change allows to define quadratic invariants in some cases for
  non perfect fields. 
   Given $k\to F$ is an extension of a perfect field $k$, we define 
   \[
    N_{X_F,D_F,\sigma_F} :=  N_{X,D,\sigma} \otimes_{k} F \in \GW(F).
  \]
\end{rem}

It is often convenient to know that $U$ has a rational point, because then \eqref{eq:Nenum} holds in $\GW(k)$. This is the case when $S$ is $k$-rational and $k$ is infinite. When $k$ is finite and $S$ is $k$-rational, we can only ensure that $U$ has a point whose $p$ such that $[k(p):k]$ is odd. However this will be enough to recover  $N_{X,D,\sigma} \in \GW(k)$ since the map $\GW(k) \to \GW(F)$ is injective for any odd degree extension $k\to F$ by \cite[Satz 10]{Pfister66}. Hence, thanks to Proposition \ref{pr:NA1-stable-basechange-at-point}, the element $N_{X,D,\sigma} \in \GW(k)$ is entirely determined by the invariant $ N_{X,D,\sigma} \otimes_{k} k(p) \in \GW(k(p))$. We include a proof of this well-known fact for completeness.

\begin{prop}\label{prop:rational_scheme_points}
Let $S$ be a $k$-rational scheme and let $k \to \sigma \cong \prod_{i=1}^r L_i$ be a finite $k$-algebra. Then any dense open subset $W$ of $\Res_{\sigma/k} S \cong \prod_{i=1}^r \Res_{L_i/k} S$ contains a point $p$ with $[k(p):k]$ finite and odd. If $k$ is infinite, we may choose $p$ such that $k(p) = k$.
\end{prop}

\begin{proof}
Since $S$ contains a dense open subset isomorphic to a dense open subset of $\A^n$, we may reduce to the case where $S=\A^n$. Then $\Res_{\sigma/k} S \cong \A^m$ for some positive integer $m$. We prove by induction on $m$ that any dense open subset of $\A^m$ contains a point $p$ with $[k(p):k]$ finite and odd and that when $k$ is infinite, we may choose $p$ with $[k(p):k]=1$. For $m=1$, any non-empty open is the complement of finitely many points. Since there are infinitely many points $p$ with $[k(p):k]$ finite and odd, and since when $k$ is infinite, there are infinitely many points $p$ with $[k(p):k]=1$, the case $m=1$ is true. Suppose inductively that the result holds for $m-1$. Since there are infinitely many linear subspaces $\A^{m-1}_L \to \A^m$ of codimension $1$ with $k\subseteq L$ a field extension with $[L:k]$ odd, and there are only finitely many components of $\A^n \setminus W$, we can find such a linearly embedded $\A^{m-1}_L $ with $W \cap \A^{m-1}_L $ nonempty, whence dense. By induction there is a point $p$ of $W \cap \A^{m-1}_L $ with $[k(p): L]$ finite and odd. Thus $[k(p):k] =[k(p): L][L:k] $ is finite and odd. In the case where $k$ is infinite, we may assume that $L=k$ and $k(p) = L$ because of the existence of infinitely many such linearly embedded hyperplanes and rational points.
\end{proof}

Even though quadratic Gromov--Witten invariants are elements of  $\GW(k)$, they specialize to  numerical invariants. 
For example the rank of quadratic Gomov--Witten invariants recovers usual Gromov-Witten invariants, while taking their signature when $k= \R$ recovers Welschinger invariants.  
To our knowledge only a few computations of genus 0 quadratic Gromov--Witten invariants are known up to now, which include the following works:
\begin{itemize}
   \item 
  $D=-K_S$ the anticanonical class of $S$ via $\A^1$-Euler characteristic computations \cite{degree};
 \item  $S$ is a rational del Pezzo surface of degree at least 6 \cite{BruRauWic25}. This computations combines in particular Theorem \ref{thm:N-surgery-intro} with the tropical computations \cite{PauliPuentes23,JPMPR25}.
 \item $S$ is a rational del Pezzo surface of degree 3 and $\sigma=k^n$, by combining  \cite{PauliPuentes23} together with results from this paper and floor decompositions with respect to a conic \cite{BRWunp}.
\end{itemize}

\section{Lefschetz fibrations of del Pezzo surfaces}\label{sec:sugery}

In this section we define Lefschetz fibrations and their surgeries. We relate the Picard groups of their central and generic fibers in the case of del Pezzo surfaces, and identify the vanishing cycle of such  Lefschetz fibration. 

\subsection{Surgeries of Lefschetz fibrations}
\begin{defi}\label{def:1-nodal-Lefschetz-fibration}
Define a {\em $1$-nodal Lefschetz fibration} to be a flat, projective family $\cX \to \Spec k[[t]]$ of surfaces satisfying the following. The total space $\cX$ is regular. The generic fiber is a smooth algebraic surface. The special fiber has a single node $q$ defined over $k$ and no other singularities. The completed local ring $\widehat{\cO_{\cX,q}}$ at $q$ is isomorphic to $k[[t,x,y,z]]/(Q(x,y,z) +  t)$ with $Q(x,y,z)$ a homogeneous degree $2$ element of $k[x,y,z]$ defining a non-singular quadric. 
\end{defi}

It is proved in \cite{KKMS73} that any family with a singular central fiber can be turned into a semi-stable degeneration after finitely blow-ups and base changes. In the case of $1$-nodal Lefschetz fibration over an separably closed field of characteristic different from 2, it is classical that the base change $t\mapsto t^2$ followed by a single blow-up produces such semi-stable reduction. When $k$ contains non-square elements, several base changes are possible. This leads to the following notion of $d$-surgery.

\begin{const}\label{const:d-surgery}
Let $\cX \to \Spec k[[t]]$ be a $1$-nodal Lefschetz fibration  and let $d$ be in $k^*$. We define the {\em $d$-surgery} $\cX(d)$ of $\cX$ as follows. Let $\delta$ be the discriminant of the quadratic form $Q(x,y,z)$ as in Definition~\ref{def:1-nodal-Lefschetz-fibration}. Pullback $\cX$ by the map $ \Spec k[[t]] \to \Spec k[[t]]$ sending $t$ to $\frac{d}{\delta}t^2$ and call the resulting scheme $\widetilde{\cX}(d)$.  The pullback $\widetilde{\cX}(d)$ has a node $q$ at $(x,y,z,t) = (0,0,0,0)$. Then $\cX(d)$ is defined to be $\cX(d) =\Bl_{q} \widetilde{\cX}(d)$.
\end{const}

\begin{rem}
Note that the discriminant of the non-singular quadric given by the zero locus of $Q(x,y,z) +  \frac{d}{\delta}t^2$ in $\P^3_k$ is $d$. 
\end{rem}

We describe the structure of the central fiber  $\cX(d)_0$ of $\cX(d)$ in the next two propositions, which in particular show that $\cX(d)$ is indeed a semi-stable family (i.e. $\cX(d)_0$ is the union of smooth surfaces intersecting transversely).

\begin{prop}\label{prop:classification_quadric_surfaces_in_P3}
Suppose $k$ is not characteristic $2$.  Let $X$ be the quadric in $\mathbb{P}^3_k$ determined by $Q(x,y,z) +  \frac{d}{\delta}t^2 = 0$. Then $X$ is isomorphic to $\Res_{\frac{k[\epsilon]}{(\epsilon^2 - d)}/k} C$ where $C$ is a smooth, proper conic over $k[\epsilon]/(\epsilon^2 - d)$. The Galois action on  $\Pic X_{k^s} \cong \mathbb{Z}^2$ factors through $\Gal(k[\sqrt{d}]/k)$, whose non-trivial element (when $d\notin k^2$) swaps the two factors of $\mathbb{Z}$.
\end{prop}

\begin{proof}
First note that the quadratic form $xy - zt = 0$ determines $\P^1_k \times \P^1_k$ by the Segre embedding. The algebraic identify $4xy = (x+y)^2 - (x-y)^2$ shows that the zero locus of $xy - zt$ to be isomorphic to that of $ x^2 +y^2 - z^2 - t^2$. 

Thus after base change to $k^s$, we have an isomorphism $X_{k^s} \cong \P^1_{k^s} \times \P^1_{k^s}$, coming from a diagonalization of the associated quadratic form and the existence in $k^s$ of the square roots of the diagonal entries. It follows that $X \cong \Res_{L/k} C$ for $L$ a degree $2$ \'etale extension of $k$ and $C$ a smooth, proper, conic curve over $L$. See for example \cite[Proof of Proposition 9.4.12(1)]{Poonen-rational_points_on_varieties}. 

The Picard group $\Pic X_{k^s} \cong \Pic (\P^1_{k^s} \times \P^1_{k^s})$ is isomorphic to $\mathbb{Z}^2$ with the standard symplectic intersection form. As the action of $\Gal(k^s/k)$ preserves this intersection form it must act through the swap map. We have an isomorphism $(\Res_{L/k} C)_L \cong C \times C^{\sigma}$, where $C^{\sigma}$ is the Galois conjugate of $C$ over $k$ (\cite[Section 4.6 Exercise 4.7]{Poonen-rational_points_on_varieties}). Over the separable closure $C_{k^s} \times *$ and $* \times (C^{\sigma})_{k^s}$ represent a basis of the Picard group. Acting the Galois group on these elements, we see that the action factors through $\Gal(L/k)$, whose non-trivial element swaps the two factors of $\mathbb{Z}$. 

It remains to show that $L \cong k[\epsilon]/(\epsilon^2 -d)$. There is a short exact sequence 
\[
1 \to \P\GL_2(k^s) \times \P\GL_2 (k^s) \to \Aut(\P^1 \times \P^1)_{k^s} \stackrel{q_*}{\to} \mathbb{Z}/2 \to 1
\] where the quotient maps takes an automorphism to the associated automorphism of the Picard group. The twist $X$ of $\P^1 \times \P^1$ determines an element $\gamma_X$ of $H^1(k, \Aut((\P^1 \times \P^1)_{k^s}))$. The field extension $L$ is determined by the image $q_*\gamma_X$ of $\gamma_X$ in $H^1(k, \mathbb{Z}/2)$. See \cite[Proof of Proposition 9.4.12(1)]{Poonen-rational_points_on_varieties}. Since we can diagonalize any quadratic form over $k$, we may assume that $Q(x,y,z) +  \frac{d}{\delta}t^2$ is diagonal. The claim then follows from an explicit examination of cocycles. 
\end{proof}

Let $C$ be a smooth, proper, plane conic such that $\Res_{\frac{k[\epsilon]}{(\epsilon^2 - d)}/k} C$ is isomorphic to the zero locus of $Q(x,y,z) +  \frac{d}{\delta}t^2 = 0$ in $\mathbb{P}^3_k$ as in Proposition~\ref{prop:classification_quadric_surfaces_in_P3}.

\begin{prop}\label{d-surgery-central-fiber}
The central fiber $\cX(d)_0$ of $\cX(d) \to \Spec k[[t]]$ is the union of two components, one isomorphic to the blow-up $\widetilde{S}:=\Bl_q (\cX_0)$ at the node of the central fiber $\cX_0$ of $\cX$, and one isomorphic to $\Res_{\frac{k[\epsilon]}{(\epsilon^2 - d)}/k} C$. These two components are glued along the exceptional divisor $E$ of $\Bl_q  (\cX_0)$, which is a  hyperplane section in $\Res_{\frac{k[\epsilon]}{(\epsilon^2 - d)}/k} C$ isomorphic to the zero locus of $Q(x,y,z) = 0$ in $\P^2_k$. Both components are smooth and projective over $k$.
\end{prop}

\begin{proof} 

The exceptional divisor of $\Bl_{(x,y,z,t)} \Spec k[[x,y,z,t]]$ is $\Proj \oplus_n (x,y,z,t)^n/(x,y,z,t)^{n+1} \cong \P^3_k$. One then computes that the exceptional divisor of $\Bl_{(x,y,z,t)} \Spec k[[x,y,z,t]]/(Q(x,y,z) +  \frac{d}{\delta}t^2)$ is the zero locus of $Q(x,y,z) +  \frac{d}{\delta}t^2$ in $\P^3_k$.  Let $U \cong \Spec A$ be an affine neighborhood of $q$ in $\widetilde{\cX}(d)$ and let $\mathfrak{q} \subset A$ denote the prime ideal corresponding to $q$. The exceptional divisor of $\Bl_q \widetilde{\cX}(d)$ is isomorphic to $\Proj \oplus_n \mathfrak{q}^n/\mathfrak{q}^{n+1}$. The map $A \to \hat{A}_{\mathfrak{q}}$ induces an isomorphism of $k$-algebras $\oplus_n \mathfrak{q}^n/\mathfrak{q}^{n+1} \to (\mathfrak{q}\hat{A}_{\mathfrak{q}})^n/ (\mathfrak{q}\hat{A}_{\mathfrak{q}})^{n+1}$. Since $\hat{A}_{\mathfrak{q}} \cong k[[t,x,y,z]]/(Q(x,y,z) +  \frac{d}{\delta}t^2)$, it follows that the exceptional divisor of $\Bl_q \widetilde{\cX}$ is isomorphic to the zero locus of $Q(x,y,z) +  \frac{d}{\delta}t^2$ in $\P^3_k$. By Proposition \ref{prop:classification_quadric_surfaces_in_P3}, we have that this exceptional divisor is  isomorphic to $\Res_{\frac{k[\epsilon]/(\epsilon^2 - d)}{k}} C$. 

The central fiber $\cX(d)_0$ has a component corresponding to the exceptional divisor of $\Bl_q (\widetilde{\cX}(d))$ and a component isomorphic to $\Bl_q ( \cX_0)$ glued along the exceptional divisor $E$ of $\Bl_q  (\cX_0)$.  It follows from the definition that $E$ is the intersection of  $\Bl_q ( \cX_0)$ with the hyperplane $t=0$ of the exceptional divisor of $\Bl_{(x,y,z,t)} \Spec k[[x,y,z,t]]$. In particular $E$ is isomorphic to $\Proj k[x,y,z]/(Q(x,y,z))$ and is non-singular. The surface $\widetilde S$ is smooth and projective over $k$ because $\cX$ is projective with a single node in the special fiber.
\end{proof}

\begin{nota}\label{nt:widetildeSE}
For the $d$-surgery $\cX(d)$ of a $1$-nodal Lefschetz fibration $\cX \to \Spec k[[t]]$, let $\widetilde S$ denote the component of $\cX(d)_0$ isomorphic to $\Bl_q (\cX_0)$ as in Proposition~\ref{d-surgery-central-fiber}. Let $E$ denote the exceptional divisor of $\Bl_q (\cX_0)$. Let $Q(d)$ denote the component of the special fiber $\cX(d)_0$ isomorphic to $\Res_{\frac{k[\epsilon]}{(\epsilon^2 - d)}/k} C$. Finally, let $\Sigma(d)$ denote the generic fiber of $\cX(d)$.
\end{nota}

When there is an ambiguity on the surface where we consider the intersection product, we  denote by $\cdot_{\widetilde S}$ and  $\cdot_{Q(d)}$  the intersection product on $\Pic \widetilde S$ and $\Pic Q(d)$, respectively. We thank Thomas Dedieu for communicating us the proof of the following classical lemma (see also \cite[Proposition 8.1.9]{Dolg12}).

\begin{lemma}\label{lm:EE=-2}
  Let  $\cX(d)$  be the $d$-surgery of a $1$-nodal Lefschetz fibration $\cX \to \Spec k[[t]]$. Then we have 
   $E \cdot_{\widetilde S} E = -2$.
\end{lemma}

\begin{proof}
Since the curve $E$ is a hyperplane section in $Q(d)$, we have $E \cdot_{Q(d)} E = 2$. We show that $E \cdot_{\widetilde S} E = - E \cdot_{Q(d)} E$ which then proves the lemma. The class of $Q(d)+ \widetilde S $ in $\Pic \cX(d)$ is the class of a fiber of the family, so it has $0$ intersection with any class in Chow contained in a fiber. It follows that $(Q(d)+ \widetilde S) \cdot Q(d) \cdot  \widetilde S =0$, that is $Q(d)^2 \cdot \widetilde S + \widetilde S^2 \cdot Q(d)=0$. Since  $E=Q(d)\cap \widetilde S$, we have that $Q(d)^2 \cdot \widetilde S= E \cdot_{Q(d)} E$ and $ \widetilde S^2 \cdot Q(d)=E \cdot_{\widetilde S} E$.
\end{proof}

\subsection{Picard groups of $d$-surgeries and vanishing cycle}
In the rest of the paper, we will only be interested in the following particular cases of $1$-nodal Lefschetz fibrations. A \emph{nodal del Pezzo surface} is a normal algebraic surface with nodes as its only singularities, such that its canonical sheaf $\omega_S$ is invertible and $\omega_S^{-1}$ is ample. Note the difference with \cite[Definition 8.1.2]{Dolg12}, where a del Pezzo surface is allowed to be singular.

\begin{defi}\label{df:uninodal_del_Pezzo}
The minimal resolution of a nodal del Pezzo surface with a single node is called a \emph{uninodal del Pezzo surface}. 
\end{defi}

In particular, the surface $S$ contains a smooth rational curve $E$ with $E^2=-2$ (equivalently $K_{S}\cdot E=0$), and for any other reduced and irreducible curve $C$ in $S$ we have $-K_{S}\cdot C>0$. See also \cite[4.1]{Vakil-curves_rational_surfaces}.
  
\begin{exa}
  By \cite[Theorem 8.1.13 and Corollary 8.1.17]{Dolg12}, examples of uninodal del Pezzo surfaces include  the second Hirzebruch surface $\P_{\P^1}(\cO \oplus \cO(2))$, the blow-up of $\P^2_k$ at two infinitely near points (ie two consecutive blow-ups of $\P^2_k$ at a point, the second point being on the exceptional divisor of the first blow-up), and the blow up of $\P^2_k$ at 6 points on a conic, no 3 of which are on a line. These are the examples we will be concerned with in Section \ref{sec:applications}.
\end{exa}

\begin{defi}\label{df:1-nodal_Lefschetz_fibration}
 A \emph{$1$-nodal Lefschetz fibration of del Pezzo surfaces} is a $1$-nodal Lefschetz fibration  $\cX \to \Spec k[[t]]$ such that the generic fiber is a del Pezzo surface and the surface $\widetilde S$ is a uninodal del Pezzo surface.
\end{defi}

As said above, explicit examples of $1$-nodal Lefschetz fibration of del Pezzo surfaces are given in Section \ref{sec:applications}. Let us now relate the Picard groups of the different fibers of such Lefschetz fibration of del Pezzo surfaces.

\begin{lemma}\label{lm:PicXd=PicXd0}
Let $\cX \to \Spec k[[t]]$ be a $1$-nodal Lefschetz fibration of del Pezzo surfaces. Let $\cX(d)$ denote the corresponding $d$-surgery. The restriction morphism
\[
 \Pic\cX(d) \stackrel{\cong}{\longrightarrow}  \Pic\cX(d)_0
\] is an isomorphism.
\end{lemma}

\begin{proof}
It suffices to show that for each line bundle $\mathcal{L}_0$ on $\cX(d)_0$, there is a unique line bundle $\mathcal{L}$ on $\cX(d)$ whose pullback to the special fiber is $\mathcal{L}_0$. By \cite[Corollary 8.5.6]{IllusieFGA}, it is sufficient to show that $H^2(\cX(d)_0, \mathcal{O}_{\cX(d)_0 } ) = 0$ and $H^1(\cX(d)_0, \mathcal{O}_{\cX(d)_0 } ) = 0$. 

By Proposition~\ref{d-surgery-central-fiber}, we have $\cX(d)_0 = \widetilde{S} \cup_E Q(d)$. We denote by $i_{\widetilde{S}}:  \widetilde{S} \to \cX(d)_0$, by $i_{Q(d)}:Q(d)\to \cX(d)_0$, and by $i_{E}:E\to \cX(d)_0$  the corresponding closed immersions. We obtain the short exact sequence 
\[
0 \to \cO_{\cX(d)_0} \to i_{\widetilde{S},*} \cO_{\widetilde{S}} \oplus i_{Q(d),*} \cO_{Q(d)} \to i_{E,*} \cO_{E}  \to 0
\] of quasi-coherent sheaves on $\cX(d)_0 $. The corresponding long exact sequence in cohomology begins $0 \to k \to k^2 \to k \stackrel{0}{\to} \ldots$. 

By \cite[Lemma 8.1.12]{Dolg12}, we have 
\[
 \forall i=1,2,\qquad  H^i(\cX(d)_0, i_{\widetilde{S},*} \cO_{\widetilde{S}}) \cong H^i(\widetilde{S}, \cO_{\widetilde{S}}) = 0 = H^i(Q(d), \cO_{\Q(d)})\cong H^i(\cX(d)_0, i_{Q(d),*} \cO_{Q(d)}).
\]
Furthermore $h^i(E,\cO_{E} ) = 0$ for $i=1,2$ as $E_{\overline{k}} \cong \P^1$. It follows that $H^2(\cX(d)_0, \mathcal{O}_{\cX(d)_0 } ) = 0$ and $H^1(\cX(d)_0, \mathcal{O}_{\cX(d)_0 } ) = 0$. 
\end{proof}

\begin{lemma}\label{lm:PicXd0=PicStildex_PicEPicQ}
Let $\cX(d) \to \Spec k[[t]]$ denote a $d$-surgery of a $1$-nodal Lefschetz fibration $\cX \to \Spec k[[t]]$ of del Pezzo surfaces. The restriction map \[
\Pic \cX(d)_0 \stackrel{\cong}{\longrightarrow} \Pic \widetilde S \times_{\Pic E} \Pic Q(d)
\]
is an isomorphism. 
\end{lemma}

\begin{proof}
Since $\cX(d)_0 $ is reduced and that $\cX(d)_0 = \widetilde S \cup_{E} Q(d)$, we obtain a short exact sequence of sheaves
\[
1 \to  \cO_{\cX(d)_0}^* \to i_{\widetilde{S},*} \cO_{\widetilde{S}}^* \oplus i_{Q(d),*} \cO^*_{Q(d)} \to i_{E,*} \cO^*_{E}  \to 1.
\] Taking the long exact sequence in cohomology yields
\[
 \cO^*(\widetilde S) \times \cO^* (Q(d))\to \cO^*(E) \to \Pic \cX(d)_0 \stackrel{\cong}{\to} \Pic \widetilde S \times_{\Pic E} \Pic Q(d) \to 0
\] Since $\cO^*(E)\cong k$, the claimed isomorphism follows.
\end{proof}

\begin{rem}\label{Rmk:closure_Cartier_divisor_general_fiber}
  Given $\cX(d)$  a $d$-surgery of a $1$-nodal Lefschetz fibration $\cX \to \Spec k[[t]]$ of del Pezzo surfaces, 
  let $D$ be a Cartier divisor on the generic fiber $\Sigma(d)$. The closure $\overline{D}$ of $D$ in $\cX(d)$ is flat. See e.g. \cite[Lemma 9.5(1)]{KLSW-relor}. It follows that $\overline{D}$ is codimension one.
  Since $\cX(d) \to \Spec k[[t]]$ is projective, the scheme $\cX(d)$ is separated and Noetherian. It is moreover integral. Since  $\cX$ is regular, so is $\cX(d)$, whence all the local rings of $\cX(d)$ are unique factorization domains by the Auslander--Buchsbaum theorem. It follows that $\overline{D}$ determines an effective relative Cartier divisor. See for example \cite[II Proposition 6.11]{Hartshorne}.
\end{rem}

\begin{defi}\label{df:varphi}
Let  $\cX(d)$  a $d$-surgery of a $1$-nodal Lefschetz fibration $\cX \to \Spec k[[t]]$ of del Pezzo surfaces.
Let $\varphi_d: \Pic \widetilde S \times_{\Pic E} \Pic Q(d) \to \Pic \Sigma(d)$ denote the composition
\[
\Pic \widetilde S \times_{\Pic E} \Pic Q(d) \stackrel{\cong}{\to} \Pic \cX(d)_0 \stackrel{\cong}{\to} \Pic \cX(d) \to \Pic \Sigma(d),
\] where the first isomorphism is from Lemma~\ref{lm:PicXd0=PicStildex_PicEPicQ} and the second isomorphism is from Lemma~\ref{lm:PicXd=PicXd0}.
\end{defi}

Over $k^s$, the quadric $Q(d)$ becomes isomorphic to $\P^1_{k^s} \times \P^1_{k^s}$. The basechange map $\P^1_{k^s} \times \P^1_{k^s} \to Q(d)$ produces a map $\Pic Q(d) \to \Pic (\P^1_{k^s} \times \P^1_{k^s})$. There is a  spectral sequence $$H^i(\Gal(k^s/k), H^j_{\textrm{et}} (Q(d)_{k^s}, \mathbb{G}_m)) \Rightarrow H^{i+j}_{\textrm{et}}(Q(d), \mathbb{G}_m).$$ By Hilbert 90, $H^1(\Gal(k^s/k), H^0(Q(d)_{k^s}, \mathbb{G}_m)) \cong H^1(\Gal(k^s/k), k^s)$ is $0$. Thus $$\Pic Q(d) \cong \Pic (\P^1_{k^s} \times \P^1_{k^s})^{\Gal(k^s/k)},$$ and the map induced by basechange produces the injection. Since $ \Pic (\P^1_{k^s} \times \P^1_{k^s}) \cong \Z \times \Z$ and the Galois action exchanges the rulings by Proposition~\ref{prop:classification_quadric_surfaces_in_P3}, we have
\begin{equation}\label{PicQ(d)}
\Pic Q(d) \cong \begin{cases}
 \Z \times \Z, \text{if } d \in (k^*)^2\\
 \{ (i,i): i \in \Z \} \subset  \Z \times \Z, \text{if } d \notin (k^*)^2
\end{cases}
\end{equation} with the isomorphism being given in the first case by pullback to $\Pic (\P^1_{k^s} \times \P^1_{k^s})$ and in the second by pullback to the diagonal subgroup of $\Pic (\P^1_{k^s} \times \P^1_{k^s})$. In particular, the element  $(1,1)$ in $ \Pic (\P^1_{k^s} \times \P^1_{k^s})$ may also be viewed as an element of $\Pic Q(d)$ for any $d\in k^*$.

\begin{prop}\label{pr:varphi_surjective}
The morphism $\varphi_d$ of Definition~\ref{df:varphi} is surjective with kernel generated by $(-E, (1,1))$ inducing an isomorphism
\[
\overline{\varphi}_d: \frac{\Pic \widetilde S \times_{\Pic E} \Pic Q(d)}{ \Z(-E, (1,1))}\stackrel{\cong}{\longrightarrow} \Pic \Sigma(d).
\]
\end{prop}

\begin{proof}
We first show $\varphi_d$ is surjective. Let $\mathcal{L}$ be an invertible sheaf on $\Sigma(d)$. Choose a meromorphic section $s$, and let $\Div (s) = \sum_i n_i D_i$ be the associated divisor. Here $n_i \in \Z$, the $D_i$ are divisors on $\Sigma(d)$ and the sum is finite. By Remark~\ref{Rmk:closure_Cartier_divisor_general_fiber}, $\overline{D_i}$ is an effective relative Cartier divisor on $\cX(d)$. Since the local rings of  $\Sigma(d)$ are unique factorization domains, we have that $\mathcal{L} \cong \cO(\sum_i n_i D_i)$. See for example \cite[II Section 6]{Hartshorne}. It follows that the line bundle $\cO(\sum_i n_i \overline{D_i})$ on $\cX(d)$ restricts to $\mathcal{L}$, showing the surjectivity of $\varphi_d$.  

Suppose that $\mathcal{L}$ is an invertible sheaf on $\cX(d)$, whose restriction to $\Sigma(d)$ is trivial. We may choose a meromorphic section $s$ of $\mathcal{L}$ which is regular and nowhere vanishing on $\Sigma(d)$. Thus the support of $\Div s$ is contained in $\cX(d)_0$, and  $\Div s = m Q(d) + n \widetilde S$ for integers $n$ and $m$. Hence we are left to compute the intersections $\widetilde S^2$, $\widetilde S\cdot Q(d)$, and $Q(d)^2$. Since $\widetilde S\cdot Q(d)=E$ and $(Q(d) + \widetilde S) \cdot Q(d) = (Q(d) + \widetilde S) \cdot \widetilde S =0$  by the proof  of Lemma~\ref{lm:EE=-2}, we obtain
$\widetilde S^2=Q(d)^2=-E$. Hence
\[
 ( m Q(d) + n \widetilde S)\cdot S = -(n-m)E \qquad\mbox{and}\qquad ( m Q(d) + n \widetilde S)\cdot Q(d) = (n-m)E,
\]
which is the desired result since $E$ realizes the class $(1,1)$ in $\Pic Q(d)$.
\end{proof}

\begin{defi}\label{df:vanishing_cycle}
Let $\cX$ be a $1$-nodal Lefschetz fibration of del Pezzo surfaces. Define {\em the vanishing cycle} of $\cX$ to be the element of $\Pic\Sigma(1)$ given by
\[
\gamma=\varphi_1(0,(1,-1)).
\]
\end{defi}

Next lemma says that given $d\in k\setminus k^2$, there is a natural identification of $\Pic\Sigma(d)$ with the orthogonal of the vanishing cycle $\gamma$ in $\Pic\Sigma(1)$.

\begin{cor}\label{cor:mapPicSigma_to_PicS}
There is a well-defined injective group homomorphism
\[
\psi_d: \Pic \Sigma(d) \to  \Pic \Sigma(1)
\] given by the composite
\[
 \Pic \Sigma(d) \xrightarrow{\overline{\varphi}_d^{-1}}  \frac{\Pic \widetilde S \times_{\Pic E} \Pic Q(d)}{ \Z(-E, (1,1))} \to \frac{\Pic \widetilde S \times_{\Pic E}  \Pic (\mathbb{P}^1_{k^s} \times \mathbb{P}^1_{k^s} )}{ \Z(-E, (1,1))}
 \]
 \[
  \cong  \frac{\Pic \widetilde S \times_{\Pic E} \Pic Q(1)}{ \Z(-E, (1,1))} \xrightarrow{\overline{\varphi}_1} \Pic \Sigma(1)
\]
\end{cor}
\begin{proof}
The second arrow and isomorphism are given by basechange, and by Proposition~\ref{pr:varphi_surjective}, $\overline{\varphi}_d$ is an isomorphism.
\end{proof}
  
Next proposition is a key element in relating the enumerative geometry of $\widetilde S$ with the one of $\Sigma(d)$. 
We thank Burt Totaro for helpful suggestions in its proof. 

\begin{prop}\label{lm:KSD=KStildeD}
  Let $\cX \to \Spec k[[t]]$  a $1$-nodal Lefschetz fibration of del Pezzo surfaces and $d\in k^*$. Let $D$ be an element of $\Pic \Sigma(d)$ and let $(D_{\widetilde S}, (i,j))$ in $\varphi_d^{-1}(D)$ correspond to a relative Cartier divisor $\mathcal{D}$ of $\cX(d) \to \Spec k[[t]]$ lifting $D$. Then $\deg(-K_{\widetilde S} \cdot D_{\widetilde S}) = \deg(-K_{\Sigma(d)}\cdot D)$.
\end{prop}

We remind the reader that the characteristic of $k$ is not $2$ in this paper by convention.

\begin{proof}
Let $f$ denote the projection $f:\cX(d) \to \Spec k[[t]]$. Since $f$ is finite type and $\cX(d)$ and $\Spec k[[t]]$ are regular, $f$ is lci \cite[0E9K]{stacks-project}. It follows that $f^! \cO_{\Spec k[[t]]} \simeq \omega_f[2]$ with $\omega_f$ a line bundle, called the dualizing sheaf of $f$. See for example \cite[ArXiv v1 Proposition B.1]{bachmann_wickelgren}. Let $[\omega_f]$ denote the corresponding rational equivalence class of divisors. By local constancy of intersection numbers in flat families \cite[B18]{Kleinman_The_Picard_Scheme}, we have 
\[
[\omega_f]\vert_{\Sigma(d)} \cdot \mathcal{D}\vert_{\Sigma(d)} = [\omega_f]\vert_{\cX(d)_0} \cdot \mathcal{D}\vert_{\cX(d)_0}. 
\] Since $\Sigma(d) \to \Spec k((t))$ is smooth, $[\omega_f]\vert_{\Sigma(d)} = K_{\Sigma(d)}$ is the canonical class of $\Sigma(d)$. Thus 
\[
[\omega_f]\vert_{\Sigma(d)} \cdot \mathcal{D}\vert_{\Sigma(d)} = \deg(K_{\Sigma(d)}\cdot D)
\] It thus remains to show that $[\omega_f]\vert_{\cX(d)_0} \cdot \mathcal{D}\vert_{\cX(d)_0} = \deg(K_{\widetilde S} \cdot D_{\widetilde S})$. We compute 
\begin{align*}
[\omega_f]\vert_{\cX(d)_0} \cdot \mathcal{D}\vert_{\cX(d)_0} &= [\omega_f]\cdot \mathcal{D} \cdot (\widetilde S + Q) \\
 & = [\omega_f]\cdot \mathcal{D} \cdot \widetilde S + [\omega_f]\cdot \mathcal{D} \cdot Q \\
 & = [\omega_f]\vert_{ \widetilde S} \cdot \mathcal{D} \vert_{ \widetilde S} +  [\omega_f]\vert_{ Q} \cdot \mathcal{D} \vert_{ Q} .
\end{align*}
The restriction of $ [\omega_f]$ to the special fiber $\cX(d)_0$ is the dualizing sheaf of $\cX(d)_0$ over $k$ which admits the following description. Let $\nu: \widetilde S \coprod Q \to \cX(d)_0$ denote the normalization map. Let $K_{ \widetilde S \coprod Q}$ denote the canonical sheaf of the normalization and let $K_{ \widetilde S \coprod Q} \otimes \cO(E_{\widetilde S} + E_Q)$ denote the sheaf of $2$-forms possibly having simple poles at the two copies $E_{\widetilde S} \subset \widetilde S$ and $E_Q \subset Q$ of $E$ in the normalization. Then the dualizing sheaf of $\cX(d)_0$ over $k$ is isomorphic to the pushforward under $\nu$ of the subsheaf of $K_{ \widetilde S \coprod Q} \otimes \cO(E_{\widetilde S} + E_Q) $ whose sections over an open subset containing the generic points of $E_1$ and $E_2$ are required to have residues at $E_1$ and $E_2$ summing to $0$, see \cite[Proposition 5.8, $m=1$, $\Delta=0$]{Kollar-Singularities-minimal-model-program} \cite[Remark 2.11]{Friedman83}. In \cite[Proposition 5.8]{Kollar-Singularities-minimal-model-program}, $k$ is assumed to be perfect. However, the proposition holds true for geometrically demi-normal schemes over any field because the dualizing sheaf is compatible with basechange. Since $\cX(d)_0$ is geometrically demi-normal, we do not need to assume that $k$ is perfect.

Thus the restrictions of $ [\omega_f]$ to the components $\widetilde S$ and $Q$ of the special fiber $\cX(d)_0$ are given by $ [\omega_f]\vert_{ \widetilde S} = [K_{ \widetilde S} \otimes \cO(E_{ \widetilde S} )] =  K_{ \widetilde S}+ E_{ \widetilde S} $ and  $ [\omega_f]\vert_{Q} = K_{ Q}+ E_{ Q} $. Thus 
\[
[\omega_f]\vert_{ \widetilde S} \cdot \mathcal{D} \vert_{ \widetilde S}  = (K_{ \widetilde S}+ E_{ \widetilde S}) \cdot D_{ \widetilde S} = K_{ \widetilde S} \cdot D_{ \widetilde S}+ E_{ \widetilde S} \cdot D_{ \widetilde S}
\]
\[
[\omega_f]\vert_{ Q} \cdot \mathcal{D} \vert_{Q}  = (K_{ Q}+ E_{Q}) \cdot \mathcal{D} \vert_{ Q} = K_{Q} \cdot \mathcal{D} \vert_{Q}+ E_{Q} \cdot \mathcal{D} \vert_{ Q}.
\]
Since $\widetilde S \cap Q = E$, we have 
\[
E_{Q} \cdot \mathcal{D} \vert_{ Q} = \widetilde S \cdot Q \cdot \mathcal{D} = E_{ \widetilde S} \cdot \mathcal{D} \vert_{ \widetilde S}.
\] Since $K_Q = -2E_Q$, we have 
\begin{equation*}
\mathcal{D}\vert_Q \cdot K_Q = \mathcal{D}\vert_Q \cdot (-2E_Q) = - 2  \mathcal{D}  \cdot \widetilde S \cdot Q
\end{equation*} Combining with the above, we have
\[
[\omega_f]\vert_{\cX(d)_0} \cdot \mathcal{D}\vert_{\cX(d)_0} = K_{ \widetilde S} \cdot D_{\widetilde S}
\] as desired.
\end{proof}

\begin{nota}\label{nota:defn_n}
Let $D$ be in $\Pic \Sigma(d)$. Define the integer $n$ by $n+1 = \deg(-K_{\widetilde S} \cdot \mathcal{D}) = \deg(-K_{\Sigma(d)}\cdot D)$, for any relative Cartier divisor $\mathcal{D}$ on $\cX(d) \to \Spec k[[t]]$ pulling back to $D$. This is well-defined by Proposition~\ref{lm:KSD=KStildeD}. 
\end{nota}

\begin{rem}\label{rem:spec_of_pi}
({\bf Specialization of points}) Let $\cX(d)$ be a $d$-surgery of a $1$-nodal Lefschetz fibration $\cX \to \Spec k[[t]]$ of del Pezzo surfaces. We can specialize a point on the generic fiber $\Sigma(d)$ to the special fiber $\cX(d)_0$: given a point $p$ of $\Sigma(d)$, the closure $\overline p$ of $p$ in $\cX(d)$ is flat over $k[[t]]$ \cite[Lemma 9.5]{KLSW-relor}; the restriction $p_0$ of $\overline p$ to $\cX(d)_0$ is then a point $p_0$ with $[k(p_{0}): k] = [k(p): k((t))]$.
\end{rem}

\begin{rem} ({\bf Twists})
Let $\cX(d)$ be a $d$-surgery of a $1$-nodal Lefschetz fibration $\cX \to \Spec k[[t]]$ of del Pezzo surfaces.
 Let $\sigma = (L_1,\ldots, L_r)$ denote finite \'etale field extensions of $k$ such that $\sum_{i=1}^r [L_i: k] = n$, where $n$ is as in Notation~\ref{nota:defn_n}. Let $\sigma[[t]] = (L_1[[t]],\ldots, L_r[[t]]) $ and similarly for $\sigma((t)) $. Let $\Res_{\sigma[[t]]/ k[[t]]} \cX(d)$ denote the restriction of scalars of $\cX(d)$ along the \'etale extension $k[[t]] \to \sigma[[t]]$. Inverting $\frac{1}{t}$ in this family produces the general fiber $\Res_{\sigma((t))/ k((t))} \Sigma(d)$. At $t=0$, we have special fiber $\Res_{\sigma/ k} (\cX(d)_0)$. See \cite[Tag 05YC]{stacks-project}.
\end{rem}

\section{Twisted binomial coefficients enriched in quadratic forms}\label{Section:twisted_binomial_coefficients_section}
As we will see in Section \ref{sec:enumerative LF},  given a $1$-nodal Lefschetz fibration of del Pezzo surfaces, relating the enumerative geometry of $\widetilde S$ to the one of $\Sigma(d)$ eventually reduces to a combination of combinatoric and Galois theory. To solve this last step,  we will need the quadratic enrichment of binomial coefficients defined in this section. 

\subsection{Definition}
Let $\sGW$ denote the Grothendieck--Witt sheaf. This is the Nisnevich sheaficiation of the presheaf valued in abelian groups on smooth schemes over a base which sends $U$ to the group completion of the isomorphism classes of unimodular $\cO_U$-valued bilinear forms. The tensor product of bilinear forms gives the structure of a sheaf of rings to $\sGW$. Examples to keep in mind include $\sGW(\P^n_k) \cong \sGW(k) \cong \GW(k)$. See for example, \cite[2.3]{degree} for more information. 

Let $x: \Spec \Omega \to X$ be a geometric point of a scheme $X$. Let 
\[
F_x: \Fet_{/X} \to  \pi_1^{\et}(X,x)-\Set
\] denote the fiber functor of $x$ from finite \'etale schemes over $X$ to sets with an action of the \'etale fundamental group. For a set $S$, let ${ S \choose j}$ denote the set of subsets of $S$ of size $j$. Note that if a group $G$ acts on $S$, the set  ${ S \choose j}$ inherits a natural $G$-action. Let $Y \to X$ be finite \'etale. Binomial coefficients enriched in quadratic forms can be defined as follows.

\begin{defi}\label{df:uGWYchoosej}
Let ${ Y/X \choose j}$ in $\sGW(X)$ denote the class of the trace form of the \'etale $X$-scheme corresponding to the set ${F_x(Y) \choose j }$ with its canonical $\pi_1^{\et}(X,x)$-action.
\end{defi} 

For $X$ a field and $j=2,3$ these binomial coefficients appear in \cite[30.12-30.14]{Garibaldi-Serre-Merkurjev}. If $X$ is understood, it is convenient to denote ${ Y/X \choose j}$ by ${Y \choose j}$. For example, the enriched binomial coefficients ${ \mathbb{F}_{q^n}/\mathbb{F}_q \choose j}$ for $n=0,1,\ldots,6$, $j =0,1,2,\ldots$ are given in Figure \ref{figure_Pascal_Triangle_finite_field}. In the figure, $u \in \mathbb{F}_q^*$ is a non-square.

\begin{figure}[h]\label{figure_Pascal_Triangle_finite_field}
\[
  \xymatrix @!=10pt @dr {
    1&1&1&1&1&1&1 \\
    1&{1+\angles u}&3&{3+\angles u}&5&{5+\angles u} \\
    1&3&6&10&15 \\
    1&{3+\angles u}&10&20 \\
    1&5&15\\
    1&{5+\angles u}\\
    1
  }
\]
\caption{Pascal's triangle ${ \mathbb{F}_{q^n} \choose j }$ over $\mathbb{F}_q$}
\end{figure}

\begin{prop}\label{prop:basic_binomial_identities_scheme} Let $X$ be a scheme equipped with a geometric point and let $Y$ be a finite \'etale $X$-scheme.
\begin{enumerate}
\item\label{nchoosej=nchoosen-j}  ${Y\choose j} = {Y \choose n-j}$ for $n=[Y:X]$. 
\item\label{n+mchoosej} If $Y = Y_2 \times Y_1$ with $Y_1$ and $Y_2$ finite \'etale over $X$, then ${Y \choose j} = \sum_{i = 0}^j {Y_1 \choose i}{Y_2 \choose j-i}$
\end{enumerate}
\end{prop}

\begin{proof}
Immediate from the corresponding isomorphisms of sets with $\pi_1^{\et}$-action.
\end{proof}

We give quadratic twists of the enriched binomial coefficients of Definition~\ref{df:uGWYchoosej}. Fix a finite \'etale degree $2$ Galois cover $Z \to X$ and let $q_Z: \pi_1^{\et}(X,x) \to C_2$ denote the corresponding quotient map, where $C_2$ denotes the cyclic group of order $2$. Let $j$ be a non-negative integer, and let $Y \to X$ a finite \'etale cover of degree $2j$. For a set $S$ of size $2j$, the set ${ S \choose j}$ has an action $C_2$ given by taking a subset $A$ of $S$ of size $j$ to its complement $S \setminus A$. If a group $G$ acts on $S$, the set  ${ S \choose j}$ inherits an action of $G \times C_2$ because taking a set to its complement commutes with any automorphism of $S$.

\begin{defi}\label{df:uGWYchoosejtwisted}
Let ${ Y[Z]/X \choose j}$ in $\sGW(X)$ denote the class of the trace form of the \'etale $X$-scheme corresponding to the set ${F_x(Y) \choose j }$ with $\pi_1^{\et}(X,x)$-action given by the homomorphism $$\pi_1^{\et}(X,x) \stackrel{(1,q_Z)}{\to} \pi_1^{\et}(X,x) \times C_2$$ and the canonical action of $\pi_1^{\et}(X,x) \times C_2$.
\end{defi} 

\begin{exa}
The first few twisted enriched binomial coefficients ${ \mathbb{F}_{q^{2j}}[\mathbb{F}_{q^2}]/\mathbb{F}_q \choose j}$ over the finite field with $q$ elements are $2$, $5 + \langle u \rangle$, $20$, and $69+ \langle u \rangle$ for $j=1,2,3,4$ respectively. Here as above, $u \in \mathbb{F}_q^*$ is a non-square. 
\end{exa}

The complete computation of ${ \mathbb{F}_{q^n}/\mathbb{F}_q \choose j}$ and ${ \mathbb{F}_{q^{2j}}[\mathbb{F}_{q^2}]/\mathbb{F}_q \choose j}$ for all appropriate $j$ and $n$ is available in \cite{Chen-Wickelgren}.

\begin{rem}
For a finite separable field extension $k \subset L$, the trace form $\Tr_{L/k} \langle 1 \rangle$ is the $\mathbb{A}^1$-Euler characteristic $\chi^{\mathbb{A}^1}(\Spec L)$ of $\Spec L$. Thus twisted enriched binomial coefficients over $k$ may be described as $\mathbb{A}^1$-Euler characteristics of associated operations on finite \'etale $k$-schemes. In the untwisted case, this scheme level operation (taking $Y/X$ to the finite \'etale $X$-scheme with fiber ${F_x(Y) \choose j }$) could be called the $j$th alternating power. It is closely related to the $j$th symmetric power whose fiber is unordered lists of $j$ elements of $F_x(Y)$ with repetitions allowed. The $j$th symmetric power determines a power operation on schemes in an appropriate sense, and the study of their $\mathbb{A}^1$-Euler characteristics is undertaken in \cite{PajwaniPal-PowerStructures} \cite{BroringViergeverSymPower} \cite{Bejleri-McKean-symmetric_powers} \cite{BrazeltonSpitz}.

\end{rem}

\subsection{Identities in twisted enriched binomial coefficients}\label{sec:binomial identity}

We have need for the following identities. Let $k$ be a field of characteristic not $2$ and fix an algebraic closure $k \subseteq \overline{k}$. Let $d$ be a non-square in $ k^*$. We let ${\sigma[\sqrt{d}]/k \choose j}$ denote the twisted enriched binomial coefficients of Definition \ref{df:uGWYchoosejtwisted} for the finite \'etale degree $2$ cover corresponding to $k \subset k[\sqrt{d}]$.

\begin{prop}\label{prop:basic_binomial_identities}  
Let $\sigma$ an \'etale $k$-algebra of degree $2j$, and $d$ be a non-square in $ k^*$. If $\sigma = E \times F$ with $[E:k] = 2m$ even, then $${\sigma[\sqrt{d}]/k \choose j}= {E[\sqrt{d}]/k \choose m}{F[\sqrt{d}]/k \choose j-m} +
    (\sum_{i=0}^{m-1} {E/k \choose i}{F/k \choose j-i})(\langle 2 \rangle + \langle 2d \rangle)$$
\end{prop}

\begin{proof}
To ease notation, we suppress the fiber functor associated to a chosen geometric point of $k$ from the notation. Since $[\sigma:k]=2j$, we have $[F:k]=2(j-m)$. The $G_k$-set corresponding to ${\sigma[\sqrt{d}]/k \choose j}$ decomposes as a disjoint union of the $G_k$-sets $$S_i := {E \choose i}{F \choose j-i} \coprod {E \choose 2m-i}{F \choose 2(j-m)-(j-i)} $$ as $i$ ranges over $i=0,1,\ldots,m-1$ and the $G_k$-set corresponding to ${E \choose m}{F \choose j-m}$, as this last set is already stable under the $C_2$ action. This last set contributes ${E[\sqrt{d}]/k \choose m}{F[\sqrt{d}]/k \choose j-m}$ to ${\sigma[\sqrt{d}]/k \choose j}$. The $G_k\times C_2$-set $S_i$ is isomorphic to the product $$S_i \cong {E \choose i}\times{F \choose j-i} \times C_2. $$ The homomorphism $G_k \to G_k \times C_2$ of Definition~\ref{df:uGWYchoosejtwisted} allows one to view the $(G_k \times C_2)$-set $C_2$ as a $G_k$-set. The corresponding \'etale algebra is isomorphic to $k[x]/(x^2 -d)$, which has trace form $(\langle 2 \rangle + \langle 2d \rangle)$. It follows that $S_i$ contributes ${E/k \choose i}{F/k \choose j-i})(\langle 2 \rangle + \langle 2d \rangle)$ to ${\sigma[\sqrt{d}]/k \choose j}$, proving the claim.
\end{proof}

\begin{prop}\label{pr:Binomial_identity}
 If $\sigma$ is an \'etale $k$-algebra with $[\sigma:k] = 2j$,
 then 
   \begin{align*}
    {\sigma[\sqrt{d}]/k \choose j}=
    {\sigma/k \choose j}\langle d^j \rangle
+ (\langle 2 \rangle - \langle 2d \rangle)
\sum_{l=0}^{j-1}(-1)^l{\sigma/k \choose l} .
  \end{align*}
\end{prop}

\begin{proof}
The ranks of both sides equal ${2j \choose j}$. We may therefore check equality in the Witt ring. A Witt invariant is a natural transformation of functors $\Fet_{/(-)}(2j) \to W(-)$ where $W$ represents the functor on field extensions $k \subseteq F$ taking $F$ to the Witt ring $W(F)$, and $\Fet_{/(-)}(2j) $ denotes the functor on field extensions $k \subseteq F$ taking $F$ to the set of finite \'etale $F$-algebras $\sigma$ with $[\sigma:F]=2j$. See  \cite[27.3]{Garibaldi-Serre-Merkurjev}. The natural transformation taking an extension $k\subseteq F$ and $\sigma \in Fet_{/(-)}$ with $[\sigma:F]=2j$ to the element
\[
 {\sigma[\sqrt{d}]/F \choose j}-
    {\sigma/F \choose j}\langle d^j \rangle
- (\langle 2 \rangle - \langle 2d \rangle)
\sum_{l=0}^{j-1}(-1)^l{\sigma/F \choose l} 
\] of the Witt ring $W(F)$ is thus a Witt invariant.
A multiquadratic $k$-algebra is defined to be a $k$-algebra of the form $\prod_{i=1}^N E_i$ where $E_i$ is a $k$-algebra of rank $1$ or $2$. 
By \cite[Theorem 29.1]{Garibaldi-Serre-Merkurjev}, it suffices to show the claim for $\sigma$ a multiquadratic algebra. The case of degree 2 extensions is straightforward, hence the claim then follows by induction from Lemma~\ref{lm:Binomial_identity_ExF}.
\end{proof}

\begin{lemma}\label{lm:Binomial_identity_ExF}
If Proposition~\ref{pr:Binomial_identity} holds for \'etale $k$-algebras $E$ and $F$, with $[F:k]=2$, then it holds for
$\sigma=E\times F$.
\end{lemma}

\begin{proof}
  By Proposition~\ref{prop:basic_binomial_identities}, one has
  \begin{align*}
    {\sigma/k \choose j}={E/k \choose j}{F/k \choose 0} \oplus
    {E/k \choose j-2}{F/k \choose 2}\oplus
    {E/k \choose j-1}{F/k \choose 1}
  \end{align*}
  and
  \begin{align*}
       {\sigma[\sqrt{d}]/k \choose j}=
    {E/k \choose j-2}(\langle 2 \rangle + \langle 2d \rangle)\oplus
    {E[\sqrt{d}]/k \choose j-1}{F[\sqrt{d}]/k \choose 1}.
  \end{align*}
  Hence we get by induction
  \begin{align*}
    {\sigma[\sqrt{d}]/k \choose j}&=
    {E/k \choose j-2}(\langle 2 \rangle + \langle 2d \rangle)\oplus
 \\ &\qquad
    \left(
  {E/k \choose j-1}\langle d^{j-1} \rangle
+ (\langle 2 \rangle - \langle 2d \rangle)
\sum_{l=0}^{j-2}(-1)^l{E/k \choose l} 
\right)
\left(
{F/k \choose 1} \langle d \rangle + \langle 2 \rangle - \langle 2d
\rangle
\right).
  \end{align*}
  Since
  \begin{align*}
    {E/k \choose j-2}(\langle 2 \rangle + \langle 2d \rangle)&=
    {E/k \choose j} \langle d^j \rangle \oplus
     {E/k \choose j-2} \langle d^j \rangle \oplus
   {E/k \choose j-2}(\langle 2 \rangle + \langle 2d \rangle-
    2\langle 2d^j \rangle)
    \\ &=
       {E/k \choose j}
       {F/k \choose 0} \langle d^j \rangle  \oplus
     {E/k \choose j-2} {F/k \choose 2} \langle d^j \rangle \oplus
   {E/k \choose j-2}(\langle 2 \rangle + \langle 2d \rangle-
    2\langle 2d^j \rangle) ,
  \end{align*}
  and
  \begin{align*}
  \langle 2 \rangle + \langle 2d \rangle-
    2\langle 2d^j \rangle = (-1)^{j-1}( \langle 2 \rangle - \langle 2d \rangle),
  \end{align*}
  we obtain
  \begin{align}\label{eq:computation}
    {\sigma[\sqrt{d}]/k \choose j}&=
    {\sigma/k \choose j}\oplus
   (-1)^{j-1} {E/k \choose j-2}(\langle 2 \rangle - \langle 2d \rangle)\oplus
    (-1)^{j-1} {E/k \choose j-1}(\langle 2 \rangle - \langle 2d
     \rangle)
     \oplus
     \\&\qquad
    (\langle 2 \rangle - \langle 2d \rangle)
\sum_{l=0}^{j-2}(-1)^{l-1}{E/k \choose l}  {F/k \choose 1} 
\oplus  (\langle 2 \rangle - \langle 2d \rangle)^2
\sum_{l=0}^{j-2}(-1)^l{E/k \choose l}. \nonumber 
  \end{align}
  Since
  \[
  (\langle 2 \rangle - \langle 2d \rangle)^2=
  2(\langle 1 \rangle - \langle d \rangle)=
  2(\langle 2 \rangle - \langle 2d \rangle),
  \]
  we can rewrite
  \begin{align*}
     (\langle 2 \rangle - \langle 2d \rangle)^2
    \sum_{l=0}^{j-2}(-1)^l{E/k \choose l}&=
    (\langle 2 \rangle - \langle 2d \rangle)
    \left(
    \sum_{l=0}^{j}(-1)^l{E/k \choose l-2}{F/k \choose 2}
\oplus    \sum_{l=0}^{j-2}(-1)^l{E/k \choose l}{E/k \choose 0}
    \right).
  \end{align*}
  Finally we have
  \[
  \sum_{l=0}^{j-2}(-1)^{l-1}{E/k \choose l}  {F/k \choose 1}=
  \sum_{l=0}^{j-1}(-1)^{l}{E/k \choose l-1}  {F/k \choose 1}.
  \]
  Hence by  \eqref{eq:computation}, the
  $(\langle 2 \rangle - \langle 2d \rangle)$-factor of
  ${\sigma[\sqrt{d}]/k \choose j}- {\sigma/k \choose j}$ is
    \begin{align*}\label{eq:computation}
      &  (-1)^{j-1} {E/k \choose j-2}\oplus
       (-1)^{j} {E/k \choose j-2}\oplus
      \sum_{l=0}^{j-1}(-1)^{l-1}\left(
          {E/k \choose l}  {F/k \choose 1}
          \oplus  {E/k \choose l-2}{F/k \choose 2}
          \oplus    {E/k \choose l}{E/k \choose 0}
          \right).
    \end{align*}
    which is exactly
    \[
    \sum_{l=0}^{j-1}(-1)^l{\sigma/k \choose l}.
    \]
    The statement is proved.
\end{proof}

\begin{cor}\label{cr:Useful_Binomial_identity}
 If $\sigma$ is an \'etale $k$-algebra with $[\sigma:k] = 2j$,  
 then
    \begin{align*}
     {\sigma'[\sqrt{d}]/k \choose j}\langle d^j \rangle=
     {\sigma'/k \choose j}
 + (-1)^j (\langle 2 \rangle - \langle 2d \rangle)
 \sum_{l=0}^{j-1}(-1)^l{\sigma'/k \choose l} .
   \end{align*}
 \end{cor}
 \begin{proof}
  This follows from the identity
  \[
    \langle 2 \rangle - \langle 2d \rangle =(-1)^j \langle d^j \rangle (\langle 2 \rangle - \langle 2d \rangle)
  \]
combined with  Proposition \ref{pr:Binomial_identity}, where both sides are multiplied by $\langle d^j \rangle$.
 \end{proof}

\section{Moduli of rational curves on Lefschetz fibrations}\label{sec:enumerative LF}
\subsection{Enumerative properties}
Let  $S\to \Spec k$ be a projective smooth algebraic surface, and $D\in \Pic S$. There is a proper moduli stack $\Mbar_{0,n}(S,D)$ over $k$ parametrizing stable maps $f: C \to S$ of degree $D$ where $C$ is a nodal, genus $0$ curve, together with the data of $n$ marked points $q_1, \ldots, q_n$ in the smooth locus of $C$. See \cite[p.90, Theorem 2.8]{Abramovich--Oort-mixed_char}. Given $k \to \sigma$  a finite \'etale algebra of degree $n$, we denote by $\Mbar_{0,n}(\Sigma,D)_{\sigma}$  the corresponding  twist of $\Mbar_{0,n}(\Sigma,D)$ defined in \cite[Section 8]{KLSW-relor}, which we recall briefly here. Choose an identification of the set $\{1,\ldots,n\}$ with the set of embeddings of $\sigma$ into the separable closure $k^s$ of $k$. The Galois group of $k^s$ acts on the basechange  $\Mbar_{0,n}(S,D)_{k^s}$ by the canonical action on $k^s$ and by permuting the marked points $q_1, \ldots, q_n$ via the corresponding permutation of the embeddings.
 
Define the evaluation map $\ev: \Mbar_{0,n}(S,D) \to S^n $ by sending $(f: C \to \Sigma, q_1,\ldots,q_n)$ to $(f(q_1),\ldots,f(q_n))$. The map $\ev$ twists to define a map
 \[
 \ev_{\sigma}: \Mbar_{0,n}(S,D)_{\sigma} \to \Res_{\sigma/k}S.
 \]
 A point $p$ of $ \Res_{\sigma/k}S$ determines a moduli stack
 \[
 \Mbar_{0,n}(S,D)_{\sigma}(p)=  \ev_{\sigma}^{-1}(p)
 \] parametrising degree $D$ rational curves on $S$ passing through the points $p_1,\ldots, p_r$ of $S$ corresponding to $p$ under the isomorphism $ \Res_{\sigma/k}S(k(p)) \cong S\big(~k(p)\otimes_{k} \sigma~\big) \cong \prod_{i=1}^r S(L_i)$, where $\prod_{i=1}^rL_i \cong k(p)\otimes_{k} \sigma$ is the decomposition of the tensor product into field extensions.

 Consider again a point $\widetilde{p}$ of $ \Res_{\sigma((t))/k((t))}\Sigma$ and the associated moduli stack
 \[
 \Mbar_{0,n}(\Sigma,D)_{\sigma((t))}(\widetilde{p}):=  \ev_{\sigma((t))}^{-1}(\widetilde{p})
 \] parametrising degree $D$ rational curves on $\Sigma$ passing through the points $p_1,\ldots, p_r$ of $\Sigma$ corresponding to $\widetilde{p}$ as above.
 
 \begin{defi}\label{df:general_fiber_enumerative}
Let $S$ be a projective smooth algebraic surface over a field $k$, and $D\in\Pic S$. 
We say that $(S,D)$ is {\em enumerative} if there exists  a finite \'etale extension  $k \to \sigma$ of degree $n=-K_{S} \cdot D -1$
and  a dense open subset $V\subset \Res_{\sigma/k} S$ such that for all $p \in V$,\begin{itemize}
  \item the moduli stack $ \Mbar_{0,n}(S,D)_{\sigma}(p)$ is finite \'etale over the residue field $k(p)$, 
  \item the points of $ \Mbar_{0,n}(S,D)_{\sigma}(p)$ parametrize stable rational degree $D$ maps to $S$ which have irreducible domain curves,
  \item are unramified,
  \item and whose image curves have at worst ordinary double points as singularities. 
  \end{itemize}
\end{defi}

In fact, if $(S,D)$ is enumerative, then for any  finite \'etale extension  $k \to \sigma'$ of degree $n=-K_{S} \cdot D -1$, there exists  a dense open subset $V'\subset \Res_{\sigma'/k} S$ satifying the conditions of Definition \ref{df:general_fiber_enumerative}.
This is because the ordinary double point locus (defined in \cite[Definition 2.9]{KLSW-relor}) of $ \Mbar_{0,n}(S,D)$ is pulled back from $ \Mbar_{0,0}(S,D)$. (The ordinary double point locus is moreover open by the proof of \cite[Lemma 2.14]{KLSW-relor} and the evaluation map is \'etale on this open set \cite[Proof of Lemma 2.27]{KLSW-relor}. The condition that $(S,D)$ is enumerative amounts to the condition that the closed image in $S^n$ of the complement of this open subscheme of $ \Mbar_{0,n}(S,D)$ is not all of $S^n$.)

\begin{exa}\label{exa:dp enum}
  By \cite[Corollary 3.15]{KLSW-relor}, the pair $(S, D)$ is enumerative as soon as the triple $(S,D,k)$ 
  satisfies Hypothesis~\ref{hyp:SDk_with_N}.
\end{exa}

\begin{lemma}\label{lm:enumerative_stable_basechange}
  Let $S \to \Spec k$ be a smooth, projective del Pezzo surface, let $D \in \Pic S$, and $k\to L$ be a field extension. 
  If $(S,D)$ is enumerative, then so is $(S_L,D_L)$.  
\end{lemma}

\begin{proof}
   Let $k \to \sigma$ be a finite \'etale map of rings of degree $n = -K_{S} \cdot D -1$.
   The moduli space $ \Mbar_{0,n}(S_L,D_L)_{\sigma_L}$ is the pullback of $\Mbar_{0,n}(S,D)_{\sigma}$ under $\Spec L \to \Spec k$ and the twisted evaluation maps are stable under basechange.  Since  $(S, D )$ is enumerative, there is a dense open subset of $V_k$ of $\Res_{\sigma /k} S$ satisfying the properties listed in Definition \ref{df:general_fiber_enumerative}. We may let $V_L$ be the inverse image of $V_k$ under the basechange map 
   \[
    \Res_{\sigma_L/L} S_L \to \Res_{\sigma /k} S.
    \] 
 \end{proof}

\begin{defi}\label{def:special_fiber_enumerative}
Let $\widetilde S$ be a uninodal del Pezzo surface with $E$ as the smooth $(-2)$-rational curve, and  $D$ be an element of $\Pic \widetilde S$. We say that $(\widetilde S, D )$ is {\em relatively enumerative} if there exsists a finite \'etale extension $k \to \sigma$ of degree $-K_{\widetilde S} \cdot D -1$ and  an open dense subset $V$ of $\Res_{\sigma/k}(\widetilde S)$ such that for all points $p$ of $V$ and for all $i\in \Z$, there are finitely many rational curves on $\widetilde S$ of degree $D-iE$ with no component mapped to $E$ through the points $p_1,\ldots, p_r$. Moreover, these finitely many curves satisfy the properties that \begin{itemize}
\item they are unramified,
\item they intersect $E$ transversely,
\item the domain curves are irreducible,
\item their image curves have at worst ordinary double points as singularities,  which lie outside $E$.
\end{itemize} Here, $p_1,\ldots, p_r$ denote the points of $\widetilde S$ corresponding to $p$. 
\end{defi}

Analogously to Definition \ref{df:general_fiber_enumerative}, if such $\sigma$  as in Definition \ref{def:special_fiber_enumerative} exists, then any  finite \'etale extension  $k \to \sigma'$ of degree $-K_{\widetilde S} \cdot D -1$ is as in Definition \ref{def:special_fiber_enumerative}. See Remark~\ref{rem:rel_enum_indep_sigma}. 

\begin{rem}\label{rem:family_rel_enum}
  If $(\widetilde S,D)$ is relatively enumerative, so is $(\widetilde S,D-iE)$ for any $i\in \Z$. Let $\cX(d) \to \Spec k[[t]]$ be a $d$-surgery of a $1$-nodal Lefschetz fibration of del Pezzo surfaces. We denote the components of the special fiber by $\widetilde S$ and $Q(d)$ as before. Let the generic fiber be denoted $\Sigma(d)$. For $D$ in $\Pic \Sigma(d)$, we have that $\varphi_d^{-1}(D)$ is of the form $\{ (D_{\widetilde S}, (l,j))+ i (-E, (1,1)): i \in \Z\}$. Thus $(\widetilde S, D_{\widetilde S})$ is relatively enumerative if and only if t$(\widetilde S, D'_{\widetilde S})$ is relatively enumerative for $D'_{\widetilde S}$ the restriction to $\widetilde S$ of any element of  $\varphi_d^{-1}(D)$. 
\end{rem}

\begin{defi}
In the setting of Remark~\ref{rem:family_rel_enum}, we say that $(\widetilde S, D)$ is {\em relatively enumerative} if $(\widetilde S, D_{\widetilde S})$ is relatively enumerative for one (equivalently any) choice of $D_{\widetilde S}$.
\end{defi}

\begin{rem}\label{rem:finitely_many_i}
We have written Definition~\ref{def:special_fiber_enumerative} as infinitely many conditions, one for each $i\in \Z$, which says that there are finitely many rational curves of degree $D-iE$ through the points. However, there are only finitely many $i$ for which there are curves of degree $D-iE$. To see this, note that the arithmetic genus $p_a$ of a curve of degree $D-iE$ is given by \begin{align*}
  2 p_a - 2 &= (D-iE)\cdot ( K_{\widetilde S} + D-iE) \\ 
  &= D \cdot K_{\widetilde S} + D^2 + i (- 2 E \cdot D-E \cdot  K_{\widetilde S}) + i^2 E \cdot E \\
  & =  D\cdot K_{\widetilde S} + D^2 - 2i E \cdot D-2 i^2
  \end{align*} where the last equality follows from Lemma~\ref{lm:EE=-2} and the adjunction formula. Since $0 \leq p_a$, we have 
  \[
  -2 \leq 2 p_a -2  =  D \cdot K_{\widetilde S} + D^2 -2 i E \cdot D -2 i^2.
  \] There are only finitely many $i$ which can satisfy this inequality. Thus Definition~\ref{def:special_fiber_enumerative}  is a condition on curves of finitely many degrees.
\end{rem}

In order to obtain examples of relatively enumerative uninodal del Pezzo surfaces, we give the following (non-surprising) lemma showing that the imposed conditions on the rational curves are open in the moduli stack of rational curves.

\begin{lemma}\label{lm:odp_tE_open}
In the notation of Definition \ref{def:special_fiber_enumerative}, the locus of pointed stable maps in $\Mbar_{0,n}(\widetilde S,D)$ which are
\begin{itemize}
\item unramified,
\item have irreducible domain curves, 
\item at worst double points as singularities in the image which all lie outside $E$,
\item and which intersect $E$ transversly 
\end{itemize}
is represented by an open subscheme $\M_{0,n}(\widetilde S,D)^{\odp,E}$ of $\Mbar_{0,n}(\widetilde S,D)$. 
\end{lemma}

\begin{proof}
The locus $\M_{0,n}(\widetilde S,D)$ of stable maps in $\Mbar_{0,n}(\widetilde S,D)$ with irreducible domain curve determines an open substack. (To check this, it suffices to show that for every map $B \to \Mbar_{0,n}(\widetilde S,D)$ with $B$ a scheme, we have that the projection $i_{\M,\Mbar,B}:B \times_{\Mbar_{0,n}(\widetilde S,D)} \M_{0,n}(\widetilde S,D) \to B$ is an open immersion. The map $B \to \Mbar_{0,n}(\widetilde S,D)$ determines a relative nodal curve $\varpi_B: C \to B$. The nodal locus is closed in $C$, whence has closed image in $B$, and the projection  $i_{\M,\Mbar,B}$ is the inclusion of the open complement.) The sublocus $\M_{0,n}(\widetilde S,D)^{\unr}$ of unramified stable maps is an open subscheme of $\M_{0,n}(\widetilde S,D)$ by \cite[Lemma 2.14]{KLSW-relor}. (Although the hypothesis on \cite[Lemma 2.14]{KLSW-relor} includes the del Pezzo assumption, the proof works for any smooth projective surface, including $\widetilde S$.) The inclusion $\M_{0,n}(\widetilde S,D)^{\odp} \subset \M_{0,n}(\widetilde S,D)^{\unr} $ consisting of the locus of stable maps with at worst ordinary double points as singularities is open by \cite[Lemma 2.14]{KLSW-relor}. Let $\varpi: \sC \to \M_{0,n}(\widetilde S,D)^{\odp}$ denote the universal curve. There is a closed immersion $Z \hookrightarrow \sC \times E$ with $Z$ given by the locus of triples of a pointed stable map $\frak{f} = (f: C \to \widetilde S, p_1,\ldots,p_n)$, one more point $p$ on $C$, and a point $e$ of $E$ satisfying the conditions that $f(p)=e$ and the map induced by $f$ on tangent spaces $T_pf: T_p C \to T_p \widetilde S$ factors through the $1$-dimensional subvector space $T_p E \subset T_p \widetilde S$. \[
Z = \{ (\frak{f},p), e): f(p)=e, \xymatrix{T_p C \ar[dr] \ar[rr]^{T_pf} && T_p \widetilde{S} \\ & T_p E \ar[ur]& } \}
\] Since $\sC \times E \to \Mbar_{0,n}(\widetilde S,D)$ is proper, the image of $Z$ is closed.

To complete the proof, it now suffices to show that the locus $Z_E$ of stable maps in $\M_{0,n}(\widetilde S,D)^{\odp}$ whose image curve has a singularity on $E$ is closed. Let $f: \sC \to \widetilde S$ denote the universal map to $\widetilde S$ sending $((f: C \to \widetilde S, p_1,\ldots, p_n),p)$ to $f(p)$. We then consider the map  
\[
(f \times f): \sC \times_{\M_{0,n}(\widetilde S,D)^{\odp}} \sC \to \widetilde{S} \times \widetilde{S} \times \M_{0,n}(\widetilde S,D)^{\odp}.
\] Let $\D \to \M_{0,n}(\widetilde S,D)^{\odp}$ denote the double point locus, defined to be subscheme of $(f \times f)^{-1}(\Delta_{\widetilde S})$ residual to the diagonal $\Delta_{\sC}$. See \cite[Definition 5.3]{KLSW-relor}. (We will discus the double point locus in more detail below as well. See Proposition~\ref{Double_point_locus_our_moduli_finite_etale}.)  Consider the map
\[
f_{\D}:\D \hookrightarrow \sC \times_{\M_{0,n}(\widetilde S,D)^{\odp}} \sC \stackrel{\pi_i}{\longrightarrow} \sC \stackrel{f \times \varpi }{\longrightarrow} \widetilde S \times \M_{0,n}(\widetilde S,D)^{\odp}.
\] (Here $\pi_i$ can denote either the first or second projection without changing the composition.) Since $\sC \to \M_{0,n}(\widetilde S,D)^{\odp}$ is proper and $\widetilde S$ is separated, the map $f_{\D}$ is proper. $Z_E$ is the projection onto $\M_{0,n}(\widetilde S,D)^{\odp}$ of the fiber product
\[
\D \times_{ \widetilde S \times \M_{0,n}(\widetilde S,D)^{\odp}} (E \times \M_{0,n}(\widetilde S,D)^{\odp})
\] Since $f_{\D}$ is proper, $\widetilde S$ is proper, and $E \to \widetilde S$ is a closed immersion, we have that $Z_E$ is closed in $\M_{0,n}(\widetilde S,D)^{\odp}$ as desired.
\end{proof}

\begin{rem}\label{rem:rel_enum_indep_sigma}
Note that the itemized conditions on the maps in Lemma~\ref{lm:odp_tE_open} are pulled map from $ \Mbar_{0,0}(S,D)$. It follows that if $(\widetilde S,D,\sigma)$ is relatively enumerative (Definition~\ref{def:special_fiber_enumerative}) for one finite \'etale $k$ algebra $k \to \sigma$ of degree $n=-K_{S} \cdot D -1$, then $(\widetilde S,D,\sigma')$ is relatively enumerative for any other finite \'etale $k \to \sigma'$ of the same degree. An appropriate twist of the open dense subset $V$ for $\sigma$ gives the desired open for $\sigma'$.
 \end{rem}

\begin{prop}\label{prop:Stildechar0_enumerative}
Let $k$ be a field of characteristic $0$, and $\widetilde S\to k$ be a 
 uninodal del Pezzo surface of degree $\geq 2$. 
Then $(\widetilde S,D)$ is relatively enumerative for any $D\in\Pic \widetilde{S}$. 
\end{prop}

\begin{proof}
Since it suffices to prove the claim over the algebraic closure $\overline{k}$, hence we may assume that $\sigma = \prod_{i=1}^n k$. 
Since $\widetilde S$ is finite type over the field $k$ and $k$ is characteristic $0$, there exists a finitely generated extension $\mathbb{Q} \subseteq \mathbb{Q}'$ so that $\widetilde S$ and $D$ are pulled back from some $\widetilde S'$ and $D'$ satisfying the hypothesis of the proposition. It follows that we may assume that $k$ is finitely generated over $\mathbb{Q}$. Since every finitely generated extension of $\mathbb{Q}$ embeds into $\mathbb{C}$, we may choose an embedding $e: k \to \mathbb{C}$. Pullback by $e$ defines a commutative diagram
\begin{equation}\label{eq:kCbasechange_hyp_special_fiber}
\xymatrix{\M_{0,n}(\widetilde S,D)^{\odp,E}_{\mathbb{C}} \ar[d] \ar[r]^{\subset} & \Mbar_{0,n}(\widetilde S,D)_{\mathbb{C}} \ar[d] \ar[rr]^{\ev_{\mathbb{C}}}&& \widetilde S_{\mathbb{C}}^n \ar[d] \\
\M_{0,n}(\widetilde S,D)^{\odp,E}_k\ar[r]^{\subset} & \Mbar_{0,n}(\widetilde S,D)_k \ar[rr]^{\ev_k} && \widetilde S^n_k .}
\end{equation}
By Lemma~\ref{lm:odp_tE_open}, we have that $\M_{0,n}(\widetilde S,D)^{\odp,E}_k$ is an open subscheme of $\Mbar_{0,n}(\widetilde S,D)_k$. Since $\ev_k$ is proper, the image of $\Mbar_{0,n}(\widetilde S,D)_k \setminus \M_{0,n}(\widetilde S,D)^{\odp,E}_k$ is closed. To show the proposition is equivalent to showing that this closed subset is not all of $ \widetilde S^n_k$.

For $k = \mathbb{C}$, the claim is shown in \cite[Section 2.3.1]{ItenbergKharlamovShustin-RelativeEnumerative}. It follows that there exists a geometric point of the open complement of $\ev_{\mathbb{C}} (\Mbar_{0,n}(\widetilde S,D)_{\mathbb{C}} \setminus \M_{0,n}(\widetilde S,D)^{\odp,E}_{\mathbb{C}})$. By the commutivity of \eqref{eq:kCbasechange_hyp_special_fiber}, it follows that there is a geometric point of the open complement of $$\ev_k(\Mbar_{0,n}(\widetilde S,D)_k \setminus \M_{0,n}(\widetilde S,D)^{\odp,E}_k),$$ completing the proof.
\end{proof}

\begin{hyp}\label{NA1-enumerative-general-finitel-special}
Let $\cX \to \Spec k[[t]]$ be a $1$-nodal Lefschetz fibration and $\Sigma(d)$ be the generic fiber of the $d$-surgery $\cX(d)$. Let  $\widetilde{S}$ and $E$ be as in Notation~\ref{nt:widetildeSE}. Let $D$ be in $\Pic \Sigma(d)$. We say that
$(\Sigma(d), D)$ 
satisfies Hypothesis~\ref{NA1-enumerative-general-finitel-special} if  
\begin{enumerate}
\item \label{it:NA1-enumerative-general} $(\Sigma(d),D)$ is enumerative, 
\item \label{it:finitel-special} $(\widetilde S, D)$  is relatively enumerative.
\end{enumerate} 
\end{hyp}

Note that if $(\Sigma(d), D)$ satisfies Hypothesis \ref{NA1-enumerative-general-finitel-special}, then so does $(\Sigma(d'), D)$ for any $d'$ because the associated moduli stacks become isomorphic after a finite field extension and one can twist the corresponding open sets.

\begin{exa}
  Let $k$ be of characteristic 0.
Combining Example \ref{exa:dp enum} and Proposition~\ref{prop:Stildechar0_enumerative},
we obtain that if  $\cX \to \Spec k[[t]]$ is a a $1$-nodal Lefshetz fibration of del Pezzo surfaces of degree at least 2 such that the general fiber satisfies Hypothesis~\ref{hyp:SDk_with_N}, then $(\cX, d, D)$ 
satisfies Hypothesis~\ref{NA1-enumerative-general-finitel-special} for any $d\in k^*$ and any $D\in \Pic\Sigma(d)$. 
\end{exa}

\begin{lemma}\label{lm:points_specializing_try2}\label{lm:points_specializing_try3}
For any open dense subset $V$ of $\Res_{\sigma/k}(\widetilde S)$ and open dense subset $\widetilde{V}$ of $\Res_{\sigma((t))/k((t))} \Sigma(d)$, there is a dense open subset $W \subseteq V$ such that for any point $p$ of $W$, there is a point $\widetilde{p}$ of $\widetilde{V}$ specializing to $p$. Moreover, if $p$ is a closed point of $W$ such that $k \subseteq k(p)$ is a separable extension, the closure of $\widetilde{p}$ is finite \'etale over $k[[t]]$ and $\widetilde{p}$ has residue field $k(p)((t))$.
\end{lemma}

\begin{proof}
The generic and special fibers of $\Res_{\sigma[[t]]/k[[t]]} \cX(d)$ are $\Res_{\sigma((t))/k((t))} \Sigma(d)$ and $$\Res_{\sigma/k}(\widetilde S) \cup_{\Res_{\sigma/k} E } \Res_{\sigma/k} Q(d)$$ respectively. Let $A = (\Res_{\sigma((t))/k((t))} \Sigma(d)) \setminus \widetilde{V}$ be the closed complement of $\widetilde{V}$ in $\Res_{\sigma((t))/k((t))} \Sigma(d)$. The closure $\overline{A}$ of $A$ in $\Res_{\sigma[[t]]/k[[t]]} \cX(d)$ is flat over $k[[t]]$ by e.g. \cite[Lemma 9.5(1)]{KLSW-relor}. Since $\widetilde{V}$ is dense in  $\Res_{\sigma((t))/k((t))} \Sigma(d)$, the codimension of $A$ is at least one. Since $\overline{A}$ is flat over $k[[t]]$ and the fibers of $\Res_{\sigma[[t]]/k[[t]]} \cX(d)$ are have all components of dimension $2n$, it follows that the intersection $\overline{A}_k$  of $\overline{A}$ with the special fiber has dense complement. Thus 
\[
W = V \cap (\Res_{\sigma/k}\widetilde S \setminus \Res_{\sigma/k} E) \setminus \overline{A}_k
\] is a dense open subset of $V$.

Take $p$ in $W$. Since $p$ does not lie in $\overline{A}_k$, any point $\widetilde{p}$ of $\Res_{\sigma((t))/k((t))} \Sigma(d)$ specializing to $p$ must lie in $\widetilde{V}$. Thus it remains to show that there is a $\widetilde{p}$ in  $\Res_{\sigma((t))/k((t))} \Sigma(d)$ specializing to $p$. Since $\cX(d) \to \Spec k[[t]]$ is smooth at every point of $\widetilde{S} \setminus E$, we have that $\Res_{\sigma[[t]]/k[[t]]} \cX(d)$ is smooth at $p$. By formal smoothness, any solid diagram \eqref{eq:formal_smoothness} admits a lift $p_{n+1}:\Spec k[[t]]/(t^{n+1}) \to \Res_{\sigma[[t]]/k[[t]]}\cX(d)$
\begin{equation}\label{eq:formal_smoothness}
\begin{tikzcd}
\Spec F[[t]]/(t^{n}) \arrow[r,"p_n"] \arrow[d] &  \Res_{\sigma[[t]]/k[[t]]}\cX(d) \arrow[d] \\
\Spec F[[t]]/(t^{n+1}) \arrow[r] \arrow[ru, dashed] & \Spec k[[t]]
\end{tikzcd}
\end{equation}
Let $p_0$ be the point $p$ with $F \cong k(p)$. Since $\varprojlim F[[t]]/(t^n) \cong F[[t]]$, we have $\Spec F[[t]] \cong \varinjlim \Spec k[[t]]/(t^n) $ so the set of compatible maps $\{ p_n: n =1,2,\}$ determines a map $\Spec F[[t]] \to \Res_{\sigma[[t]]/k[[t]]} \cX(d)$ which defines $\widetilde{p}:\Spec F((t)) \to \Res_{\sigma((t))/k((t))} \Sigma(d)$.

Let $p$ be a closed point of $W$ such that $k \subseteq k(p)$ is a separable extension. The closure $\overline{\widetilde{p}}$ of $\widetilde{p}$ in $\Res_{\sigma[[t]]/k[[t]]} \cX(d)$ is flat over $k[[t]]$ by e.g. \cite[Lemma 9.5(1)]{KLSW-relor} and proper because $\cX(d)$ is proper over $k[[t]]$. The locus where $\overline{\widetilde{p}}$ is ramified is the support of the sheaf of relative differentials, which is closed and does not contain any point of the special fiber. Since the image of this closed set in $\Spec k[[t]]$ is a closed subset not containing the special point it is empty, whence $\overline{\widetilde{p}} \to \Spec k[[t]]$ is flat, unramified, and proper, whence \'etale and finite. By \cite[X Th\'eor\`eme 2.1]{sga1}, finite \'etale $k[[t]]$-algebras are of the form $k[[t]] \to A[[t]]$ where $k \to A$ is a finite \'etale $k$-algebra. It follows that $\overline{\widetilde{p}} \cong \Spec k(p)[[t]]$, whence $\widetilde{p}$ has residue field $k(p)((t))$.
\end{proof}

 \subsection{Moduli spaces of stable maps to Lefschetz fibrations of del Pezzo surfaces:  construction}\label{sec:construction moduli}
 Let $\Sigma(d) \to \Spec k((t))$ be the generic fiber of a $d$-surgery of a $1$-nodal Lefschetz fibration $\cX \to \Spec k[[t]]$. Let $D$ in $\Pic \Sigma(d)$, and  $k \to \sigma$ be a finite \'etale of degree $n = -K_{\Sigma(d)} \cdot D -1$. Suppose  $(\Sigma(d), D)$ satisfies Hypothesis~\ref{NA1-enumerative-general-finitel-special}. Let $W$ denote the dense open subset of Lemma~\ref{lm:points_specializing_try2} applied to open sets $V$ and $\widetilde{V}$ as in Definitions~\ref{def:special_fiber_enumerative} and \ref{df:general_fiber_enumerative}. Choose a closed point $p$ of $W\subset \Res_{\sigma/k} \widetilde S$.  Let $F$ denote the field of definition of $p$. (So $p$ corresponds to a finite list $p_1,\ldots, p_r$ of points of $\widetilde{S}$ with residue fields $\sigma_F$.) Assume that  $k \subseteq F$ is separable. Let $\widetilde{p}$ in $\widetilde{V}$ be a point specializing to $p$ in $W$.  By Lemma~\ref{lm:points_specializing_try3}, the field of definition of $\widetilde{p}$ is $F((t))$. We now construct an Artin stack 
 \[
 \M_0(\cX(d),D)_{\sigma}(\widetilde{p}) \to \Spec F[[t]]
 \] 
 with generic fiber $ \Mbar_{0,n}(\Sigma(d),D)_{\sigma}(\widetilde{p})$. We thank Ravi Vakil for helpful discussions on these moduli spaces.
  
 Since  $(\Sigma(d), D)$ is enumerative, we have that $ \Mbar_{0,n}(\Sigma(d),D)_{\sigma}(\widetilde{p}) \to \Spec F((t))$ is finite \'etale and can thus be identified with a disjoint union $\coprod_{i=1}^m \Spec E_i$ where $F((t)) \to E_i$ are separable field extensions.  
 
Let $F((t)) \subseteq L$ be the finite, separable field extension corresponding to some $E_i$. Let $\mathfrak{f}_i = (f: C \to \Sigma(d)_L, p_1,\ldots,p_n)$ be the corresponding $L$-point of $\Mbar_{0,n}(\Sigma(d),D)_{\sigma}(\widetilde{p})$. The image $f(C)$ of $C$ in $\Sigma(d)_L$ is a Cartier divisor of class $D$. Let $B$ denote the integral closure of $F[[t]]$ in $L$. The extension $F[[t]] \subseteq B$ is a finite extension of discrete valuation rings by \cite[Tag 09E8]{stacks-project}. By Remark~\ref{Rmk:closure_Cartier_divisor_general_fiber}, the closure $\overline{f(C)}$ of $f(C)$ in $\cX(d)_B$ is a relative Cartier divisor, which we will denote by $\D_B$. Let $[\D_B]$ denote the corresponding class in $\Pic \cX(d)_B$. Note that basechange defines a map $\Pic \cX(d) \to \Pic \cX(d)_B$.

\begin{lemma}\label{lm:closure_of_D}
$[\D_B]$ is the image of a unique $\D$ in $\Pic \cX(d)$. 
\end{lemma}

\begin{proof}
By Lemmas \ref{lm:PicXd=PicXd0} and \ref{lm:PicXd0=PicStildex_PicEPicQ}, we have $ \Pic \widetilde S \times_{\Pic E} \Pic Q(d) \cong \Pic \cX(d)$. Hence $$ \Pic \widetilde S_{\ell} \times_{\Pic E_{\ell}} \Pic Q(d)_{\ell} \cong \Pic \cX(d)_B$$ where $\ell$ is the residue field of $B$ by the same argument. We have a commutative diagram
\[
\begin{tikzcd}
\Pic \cX(d) \arrow[r, "\otimes_{k[[t]]} B"] \arrow[d, "\varphi_d"] & \Pic \cX(d)_B \arrow[d, "\varphi_{d,B}"] \\
\Pic \Sigma(d) \arrow[r, "\otimes_{k((t))} L"] &\Pic \Sigma(d)_L
\end{tikzcd}
\]
By the argument of Proposition~\ref{pr:varphi_surjective}, the kernels of $\varphi_d$ and $\varphi_{d,B}$ are both identified with $\mathbb{Z}(-E, (1,1))$ under the above isomorphisms. Thus $\otimes_{k[[t]]}  B$ induces a bijection from $\varphi_d^{-1}(D)$ to $\varphi_{d,B}^{-1}(D_L)$. Since $\varphi_{d,B}[\D_B] = D_L$, the lemma is proven. 
\end{proof}

Applying \cite[p.90, Theorem 2.8]{Abramovich--Oort-mixed_char} again, we have a proper moduli stack $\Mbar_{0,n}(\cX(d),\D)$ over $k[[t]]$, parametrizing stable maps $f: C \to \cX(d)$ of degree $\D$ where $C$ is a nodal, genus $0$ curve, together with the data of $n$ marked points $q_1, \ldots, q_n$ in the smooth locus of $C$. We again form the twist $\Mbar_{0,n}(\cX(d),\D)_{\sigma[[t]]}$ corresponding to the finite etale algebra $k[[t]] \to \sigma[[t]] $ and the twisted evaluation map
\[
\ev_{\cX(d),\D,\sigma}: \Mbar_{0,n}(\cX(d),\D)_{\sigma[[t]]} \to \Res_{\sigma[[t]]/k[[t]]} \cX(d).
\] Let $P$ denote the closure of $\widetilde{p}$ in $\Res_{\sigma[[t]]/k[[t]]} \cX(d)$ as in Remark~\ref{rem:spec_of_pi}. By Lemma~\ref{lm:points_specializing_try3}, $P \cong \Spec F[[t]]$. Let 
\[
\Mbar_{0}(\cX(d),\D)_{\sigma}(P) : = \ev_{\cX(d),\D,\sigma}^{-1}(P) \to \Spec F[[t]]
\] denote the fiber of the evaluation map over $P$. The generic fiber of $\Mbar_{0}(\cX(d),\D)_{\sigma}(P)$ is canonically identified with $\Mbar_{0,n}(\Sigma(d),D)_{\sigma}(\widetilde{p})$. Thus $\mathfrak{f}_i = (f: C \to \Sigma_L, p_1,\ldots,p_n)$ determines a point of $\Mbar_{0}(\cX(d),\D)_{\sigma}(P)$. Let $\overline{\mathfrak{f}_i} \to \Spec F[[t]]$ denote the closure of  $\mathfrak{f}_i$ in $\Mbar_{0}(\cX(d),\D)_{\sigma}(P)$. Our moduli space is now defined as the disjoint union of the $\overline{\mathfrak{f}_i}$. 
 
 \begin{defi}\label{df:OurModuli}
 Let  $\M_0(\cX(d),D)_{\sigma}(\widetilde{p}) \to \Spec F[[t]]$ be given by the disjoint union $\coprod_{i=1}^m \overline{\mathfrak{f}_i}$.
 \end{defi}
 
   \begin{defi}
 Suppose $f_0$ is a point of $\overline{\mathfrak{f}_i} $ lying over the closed point of $\Spec F[[t]]$. The $L$-point $f_i$ of $\Mbar_{0,n}(\Sigma(d),D)_{\sigma}$ is said to {\em specialize} to $f_0$.
 \end{defi}

\subsection{Moduli spaces of stable maps to Lefschetz fibrations of del Pezzo surfaces: properties}

We show in this section that $\M_0(\cX(d),D)_{\sigma}(\widetilde{p}) \to \Spec F[[t]]$ is finite \'etale. We keep the  notation from  Section \ref{sec:construction moduli}.

\begin{lemma}\label{lm:describe_specialized_curves}
Let $f$ be an $L$-point of $\Mbar_{0,n}(\Sigma(d),D)_{\sigma}(\widetilde{p})$.
Then $f$ specializes to a stable rational curve on $\cX(d)_0$ through the points determined by $p$ of the form :
\[
f_0: C_1 \cup C_2 \to \cX(d)_0
\] satisfying the conditions:
\begin{itemize}
\item $C_1$ is an irreducible rational curve.
\item $f_1 = f_0 \vert_{C_1}: C_1 \to \widetilde S$ is a rational curve on $\widetilde S$ which is transverse to $E$. Let $D_0$ in $\Pic \widetilde S$ be the degree of $f_1$.  
\item $(C_2)_{k^s} \cong \coprod^{D_0 \cdot E} \mathbb{P}^1_{k^s}$ and the restriction $f_2 =  f_0 \vert_{C_2}: C_2 \to Q(d)$ is transverse to $E$. 
\item $f^{-1}(E) = C_1 \cap C_2$.
\item  $(f_2)_{k^s}$ maps each $\mathbb{P}^1_{k^s}$ component of $(C_2)_{k^s}$ to one of the rulings of $Q(d)_{k^s} \cong \mathbb{P}^1_{k^s} \times \mathbb{P}^1_{k^s}$.
\end{itemize}
\end{lemma}

\begin{proof}
Let $f_0: C_0= C_1 \cup C_2 \cup C_E \to \cX(d)_0 $ be the specialization of $f$, where $C_E : = f_0^{-1}(E)$ and $C_1$ is the union of components of the domain curve which are mapped to $\widetilde S$ but not mapped entirely to $E$ and similarly for $C_2$. Since $f(C_1)$ is constrained by the point configuration $p$, Hypothesis \ref{NA1-enumerative-general-finitel-special} \eqref{it:finitel-special} implies that $C_1$ is irreducible, and that $f(C_1)$ has at worst nodal singularities, and intersects $E$ transversely. We wish to see that 
\begin{itemize}
  \item $C_E$ is reduced to the nodes of $C_0$;
  \item $f_2$ is transverse to $E$.
\end{itemize}
We claim this implies that  
\begin{itemize}
  \item $(C_2)_{k^s} \cong \coprod^{D_0 \cdot E} \mathbb{P}^1_{k^s}$;
  \item $(f_2)_{k^s}$ maps each $\mathbb{P}^1_{k^s}$ component of $(C_2)_{k^s}$ to one of the rulings of $Q(d)_{k^s} \cong \mathbb{P}^1_{k^s} \times \mathbb{P}^1_{k^s}$.
\end{itemize}
Indeed, suppose that a component $\widetilde C$ of $C_2$ is mapped to a curve of class $(a,b)$ with $a+b\ge 2$. Then $E\cdot f(\widetilde C)\ge 2$, implying that $\widetilde C$ intersects $C_1$ in at least two points. But this contradicts the fact that $C_0$ is rational.

The closure of the divisor $f(C)$ in $\cX(d)$ is the vanishing locus of a section $s\in H^0(\cX(d),\mathcal O(\D))$  by Lemma~\ref{lm:closure_of_D}. The curves $f_1(C_1)$ and $f_2(C_2)$ are the vanishing locus of $s|_{\widetilde S}$ and $s|_{Q(d)}$, respectively. 
By definition, an irreducible component of $C_0$ is  mapped surjectively to $E$ if and only if $s|_{E}$ is the 0-section of $\mathcal O(\D)|_E$.
If this is not the case, then $s|_{E}$ equals the restriction to $E$ of both $s|_{\widetilde S}$ and $s|_{Q(d)}$. In particular, since $f_1(C_1)$ is transverse to $E$, so is $f_2(C_2)$. This also implies that if $f(C_E)$ is $0$-dimensional, then so is $C_E$. Indeed, consider a connected component of $C_E$. It is is either a point, or a tree of rational curves. In this latter case, it must contain an irreducible component $\widetilde C$ that intersects $C_E\setminus \widetilde C$ in at most two points. In particular it is a non-stable component of $C_0$, which cannot then be contracted to a point by $f_0$.

\medskip 
Hence we are left to prove that $C_E$ is $0$-dimensional. 
Suppose to the contrary that $C_E$ contains a curve $\widetilde C_E$ such that $f_0|_{\widetilde C_E}:\widetilde C_E\to E$ is surjective. Denote by $\cX(d)'$ the blow-up of $E$ in $\cX(d)$, by $N$ the exceptional divisor of $\cX(d)'$, and by $ f':\widetilde C_0\to  \cX(d)_0'$  the specialization of $f$ in $\cX(d)_0'$. 
Recall that $N$ is a $\P^1$-bundle over $E$.
Analogously to the proof of Lemma \ref{lm:EE=-2}, we have that
\begin{itemize}
  \item $E_0=N\cap \widetilde S$ is a section of $N$ with self-intersection 2 in $N$;
  \item $E_{\infty}=N\cap Q$ is a section of $N$ with self-intersection $-2$ in $N$. 
\end{itemize} 
Since the curves $E_0$ and $E_\infty$ are disjoint, it follows from the classification of $\P^1_{k}$-bundles over $\P^1_{k}$ that $N=\P(\mathcal O_{E_0}(2)\oplus \mathcal O_{E_0})$. In particular, any reduced and irreducible curve in $N$ which intersect $E_0$ in at most one point (counted with multiplicity) is either $E_\infty$ or a fiber of the projection to $E_0$.
If $\widetilde C_0$ still has some irreducible components that are entirely mapped to $E_0$ or $E_\infty$, then blow-up $E_0$ and $E_\infty$ in $ \cX(d)'$. After finitely many such blow-ups, we obtain a variety $\widehat \cX(d)$, and the specialization of $f$ to $\widehat \cX(d)_0$ is a stable map  $\widehat f:\widehat C_0\to \widehat \cX(d)_0$ such that no irreducible component of $\widehat C$ is mapped to the singular locus of $ \widehat \cX(d)_0$. Furthermore, all exceptional divisors of  $\widehat \cX(d)$ are isomorphic to $N$. 
Since $C_0$ has an irreducible component that surjects to $E$, the curve $\widehat C$ has an irreducible component $\widehat C_E$ mapped to one of this copy of $N$ and which intersects the section $E_0$ in at least two distinct points. Analogously to the above argument for $C_Q$, this contradicts the fact that $\widehat C$ is rational. 
\end{proof}

We thank Ilya Tyomkin for help with the following lemma.

\begin{lemma}\label{lemma:unique_generalization}
  Suppose  $(\Sigma(d) ,D)$ satisfies Hypothesis \ref{NA1-enumerative-general-finitel-special}. Let $\widetilde{p}$ in $\widetilde{V}$ be a point specializing to $p$ in $W$ such that $k \subseteq F = k(p)$ is finite and separable. The moduli stack $\M_0(\cX(d),D)_{\sigma}(\widetilde{p}) \to \Spec F[[t]]$ is formally \'etale at each stable map \[
f_0: C_1 \cup C_2 \to \cX(d)_0
\]  as described in Lemma~\ref{lm:describe_specialized_curves}.
\end{lemma}

\begin{proof}
  Let $C_0 = C_1 \cup C_2$, and let $q_i$ denote a chosen point on $C_1$ mapping to $p_i$. Define $f_1: C_1 \to \widetilde S$ and $f_2: C_2 \to Q(d)$ by restricting $f$. Let $N_{f_i}$ denote the normal sheaf of $f_i$ for $i=1,2$ defined by the short exact sequences
  \begin{equation}\label{eq:Nf1def 2}
  0 \to T_{C_1}(\sum_i q_i) \to f_1^* T_{\widetilde S} \to N_{f_1} \to 0
  \end{equation}
  \begin{equation}\label{eq:Nf2def 2}
  0 \to T_{C_Q(d)} \to f_2^* T_{Q(d)} \to N_{f_2} \to 0
  \end{equation} Let $N_{f_{1,2}}$ denote the normal sheaf of $f_{1,2}: C_1 \cap C_2 \to E$. Let $\Def^1$ denote the tangent space to the deformation space of $f$ and let $\Ob$ denote the obstruction space \cite[Chapter 3]{Illusie-Complexe_cotangent_deformationsI}. Then there is an exact sequence
  \begin{align*}
  0 \to \Def^1 \to \oH^0( N_{f_1}) \oplus \oH^0( N_{f_2}) \to \oH^0( N_{f_{1,2}}) \to
  \Ob \to \oH^1(N_{f_1})  \oplus \oH^1( N_{f_2}) \to 0
  \end{align*}

To show that the moduli stack is formally \'etale at $f_0$, we must show $ \Def^1\cong 0$ and $\Ob \cong 0$. It is thus necessary and sufficient to have
  \begin{equation*}
  \oH^0(N_{f_1})\oplus \oH^0(N_{f_2})\to \oH^0(N_{f_{1,2}})
  \end{equation*} be an isomorphism and
  \begin{equation*}
  \oH^1(N_{f_i})=0 \quad \text{ for } \quad i=1,2.
  \end{equation*}
  It is enough enough to check these statements after base change to the algebraic closure, i.e. for the rest of the proof we may work  over an algebraically closed field $k$. In particular  $Q(d)\cong \P^1_k \times \P^1_k$

  Let $D_0$ be the degree of $f_1$.  By \eqref{eq:Nf1def 2}, $\deg  N_{f_1} +2+n = \deg(-K_{\widetilde S} \cdot (D_0))$. By Proposition~\ref{lm:KSD=KStildeD}, we have that $K_{\widetilde S} \cdot (D_0) = K_{\Sigma(d)} \cdot D$, implying that $\deg  N_{f_1} = -1$. Since $N_{f_1}$ is a line bundle over $\P^1$, we have  $ N_{f_1} \cong \cO(-1)$,
  whence $\oH^1(N_{f_1})=\oH^0(N_{f_1})=0$. 

  We claim $ N_{f_2} \cong \cO$.  An irreducible component $C'$ of $C_2$ is a $\P^1$, and  $f_2(C')$ has degree $(1,0)$ or $(0,1)$. By \eqref{eq:Nf2def 2}, we have  $\deg  N_{f_2} +2 = \deg(-K_{Q} \cdot \deg f_2(C'))=2$. Thus $\deg  N_{f_2} = 0$, and $ N_{f_2} \cong \cO$.  In particular $\dim \oH^0(N_{f_2})=D_0\cdot E $, and $\oH^1(N_{f_2})=0$.
  Since $N_{f_{1,2}}$ is the restriction of the trivial line bundle $N_{f_2}$ to $D_0\cdot E$ points, we have that   $\oH^0(N_{f_2})\to \oH^0(N_{f_{1,2}})$ is an isomorphism as claimed.
  \end{proof}
  
  \begin{thm}\label{thm:moduli-stack-finite-etale}
  Suppose  $(\Sigma(d), D)$ satisfies Hypothesis \ref{NA1-enumerative-general-finitel-special}. Let $\widetilde{p}$ in $\widetilde{V}$ be a point specializing to $p$ in $W$ such that $k \subseteq F = k(p)$ is finite and separable. Then $\M_0(\cX(d),D)_{\sigma}(\widetilde{p}) \to \Spec F[[t]]$ is a finite \'etale morphism of schemes.
  \end{thm}
  
  \begin{proof}
 By \cite[p.90, Theorem 2.8]{Abramovich--Oort-mixed_char}, the moduli stacks  $\Mbar_{0,n}(\cX(d),\D)$ considered in the construction of $\M_0(\cX(d),D)_{\sigma}(\widetilde{p})$ are proper over $\Spec F[[t]]$. It follows that $\Mbar_{0}(\cX(d),\D)_{\sigma}(P) $ and $\overline{\mathfrak{f}_i}$ are also proper over $\Spec F[[t]]$. Thus $\M_0(\cX(d),D)_{\sigma}(\widetilde{p}) \to \Spec F[[t]]$ is proper. 
  
 Since the geometric fibers of all points of the moduli stack $\M_0(\cX(d),D)_{\sigma}(\widetilde{p})$ represent unramified curves, these curves have no automorphisms (\cite[Lemma 2.8]{KLSW-relor}). Thus $\M_0(\cX(d),D)_{\sigma}(\widetilde{p})$ is a scheme. Since  $(\Sigma(d), D)$ is enumerative and $\widetilde{p}$ is in the dense open set $\widetilde{V}$ of Definition~\ref{df:general_fiber_enumerative},  $\M_{0,n}(\Sigma(d),D)_{\sigma}(\widetilde{p}) \to \Spec F((t))$ is finite \'etale. By construction, the generic fiber of  $\M_0(\cX(d),D)_{\sigma}(\widetilde{p})$ is canonically isomorphic to $\M_{0,n}(\Sigma(d),D)_{\sigma}(\widetilde{p}) \to \Spec F((t))$. By Hypothesis \ref{NA1-enumerative-general-finitel-special} \eqref{it:finitel-special} and Lemma~\ref{lm:describe_specialized_curves}, the special fiber of $\M_0(\cX(d),D)_{\sigma}(\widetilde{p}) \to \Spec F[[t]]$ is finite. Since a proper map with finite fibers is finite, $\M_0(\cX(d),D)_{\sigma}(\widetilde{p}) \to \Spec F[[t]]$ is finite.
 
 Moreover, we have seen that $\M_0(\cX(d),D)_{\sigma}(\widetilde{p}) \to \Spec F[[t]]$ is \'etale at all points of the generic fiber. By \cite[Lemma 9.5(1)]{KLSW-relor}, $\M_0(\cX(d),D)_{\sigma}(\widetilde{p}) \to \Spec F[[t]]$ is flat. By Lemma~\ref{lemma:unique_generalization}, the special fiber of $\M_0(\cX(d),D)_{\sigma}(\widetilde{p}) \to \Spec F[[t]]$ is unramified. This implies the that $$\M_0(\cX(d),D)_{\sigma}(\widetilde{p}) \to \Spec F[[t]]$$ is \'etale.
  \end{proof}
  
 \subsection{Double point locus and deformation invariance} \label{subsection:double_point_locus}
  
    Suppose  $(\Sigma(d), D)$ satisfies Hypothesis \ref{NA1-enumerative-general-finitel-special}. As above, let $\widetilde{p}$ in $\widetilde{V}$ be a point specializing to $p$ in $W$ such that $k \subseteq F = k(p)$ is finite and separable.  We define the double point locus $\D \to \M_0(\cX(d),D)_{\sigma}(\widetilde{p})$ as in Fulton \cite[Chapter 9.3]{fulton98}. The construction is as follows.  By Theorem~\ref{thm:moduli-stack-finite-etale}, the moduli stack $\M_0(\cX(d),D)_{\sigma}(\widetilde{p}) \to \Spec F[[t]]$ corresponds to a finite \'etale extension of rings $F[[t]] \subseteq R$. Let $C \to \M_0(\cX(d),D)_{\sigma}(\widetilde{p})$ denote the universal curve and let $\widetilde{f}: C \to \cX(d)$ denote the universal map to $\cX(d)$. Let 
\[
(\widetilde{f}, \widetilde{f}): C \times_{\M_0(\cX(d),D)_{\sigma}(\widetilde{p})} C \to \cX(d) \times \cX(d) \times  \M_0(\cX(d),D)_{\sigma}(\widetilde{p})
\] denote the map determined by $\widetilde{f}$ on the first factor of $C$,  $\widetilde{f}$ on the second factor of $C$ and the projection. The preimage $(\widetilde{f}, \widetilde{f})^{-1}(\Delta_{\cX(d)})$ of the diagonal 
\[
\Delta_{\cX(d)} \hookrightarrow \cX(d) \times \cX(d) \times  \M_0(\cX(d),D)_{\sigma}(\widetilde{p})
\] of $\cX(d) \times  \M_0(\cX(d),D)_{\sigma}(\widetilde{p}) \to \M_0(\cX(d),D)_{\sigma}(\widetilde{p}) $  contains the diagonal
\[
\Delta_C \hookrightarrow C \times_{\M_0(\cX(d),D)_{\sigma}(\widetilde{p})} C 
\] of $C \to \M_0(\cX(d),D)_{\sigma}(\widetilde{p})$. 

\begin{defi}\label{df:Double_point_locus}
The {\em Double Point Locus} $\sB \to \M_0(\cX(d),D)_{\sigma}(\widetilde{p})$ is defined to be the closed subscheme  of $ C \times_{\M_0(\cX(d),D)_{\sigma}(\widetilde{p})} C $ corresponding to the colon sheaf of ideals $( I_{(\widetilde{f}, \widetilde{f})^{-1}(\Delta_{\cX(d)})} : I_{\Delta_C})$.
\end{defi}

The ideal sheaf $( I_{(\widetilde{f}, \widetilde{f})^{-1}(\Delta_{\cX(d)})} : I_{\Delta_C})$ are local sections $s$ of $\cO_{C \times_{\M_0(\cX(d),D)_{\sigma}(\widetilde{p})} C }$ such that $st$ lies in $I_{(\widetilde{f}, \widetilde{f})^{-1}(\Delta_{\cX(d)})} $ for all $t$ in $I_{\Delta_C}$. See also \cite[Section 5]{KLSW-relor}.
 
 \begin{prop}\label{Double_point_locus_our_moduli_finite_etale}
$\sB \to \M_0(\cX(d),D)_{\sigma}(\widetilde{p})$ is finite \'etale.
\end{prop} 

\begin{proof}
Since $C \to \M_0(\cX(d),D)_{\sigma}(\widetilde{p})$ is proper, so is $\sB \to \M_0(\cX(d),D)_{\sigma}(\widetilde{p})$.
Since the stable maps parametrised by $\M_0(\cX(d),D)_{\sigma}(\widetilde{p}) \cong \Spec R$ are unramified, it follows from the proof of \cite[Lemma 2.1]{KLSW-relor}, that $(\widetilde{f}, \widetilde{f})^{-1}(\Delta_{\cX(d)}) \setminus \Delta_C$ is closed in $ C \times_{\M_0(\cX(d),D)_{\sigma}(\widetilde{p})} C $, whence the natural map $(\widetilde{f}, \widetilde{f})^{-1}(\Delta_{\cX(d)}) \setminus \Delta_C \to \sB$ is an isomorphism. Because the stable maps parametrised by $\M_0(\cX(d),D)_{\sigma}(\widetilde{p})$ have geometric fibers with only ordinary double point singularities, the morphism
\[
(\widetilde{f}, \widetilde{f}): (C \times_{\M_0(\cX(d),D)_{\sigma}(\widetilde{p})} C ) \setminus \Delta_C \to  \cX(d) \times \cX(d) \times  \M_0(\cX(d),D)_{\sigma}(\widetilde{p})
\] is transverse to $\Delta_{ \cX(d)}$ over $ \M_0(\cX(d),D)_{\sigma}(\widetilde{p})$. Thus $((\widetilde{f}, \widetilde{f})^{-1}\Delta_{ \cX(d)}) \setminus \Delta_C \to  \M_0(\cX(d),D)_{\sigma}(\widetilde{p})$ is \'etale. The fibers of $\sB \to \M_0(\cX(d),D)_{\sigma}(\widetilde{p})$ are thus proper of dimension $0$, whence finite. A proper map with finite fibers is finite.
\end{proof}
 
 \begin{rem}\label{Rm:disc}({\bf Discriminant of a finite \'etale extension}.)
 Associated to any finite \'etale extension $g: X \to Y$ is a discriminant $\Disc(g) \in \Gamma(\cO^*_Y/ (\cO^*_Y)^2)$. Let us recall a construction. The pushforward $g_* \cO_X$ is locally free because $g$ is finite and flat. Since $g$ is \'etale, there is a trace morphism $\Tr_g: g_* \cO_X \to \cO_Y$ sending a local section $s$ of $g_*\cO_X$ to the trace of a matrix representing multiplication by $s$ as an $\cO_Y$-linear map. Equivalently, the trace morphism $\Tr_g: g_*(\det g^! \cO_Y) \to \cO_Y$ from Serre duality is identified with the previous when $g$ is \'etale because $g^! \cO_Y \simeq 0$, defining a canonical isomorphism $\det g^! \cO_Y \cong \cO_X$. Multiplication on $\cO_X$ composed with the trace defines the non-degenerate, symmetric pairing
 \[
 g_* \cO_X \otimes g_* \cO_X \to g_* \cO_X \xrightarrow{\Tr_g} \cO_Y,
 \] whence an isomorphism
 \[
  g_* \cO_X \to  \Hom( g_* \cO_X, \cO_Y) =: (g_* \cO_X)^{\vee}.
 \] Taking determinants and tensoring with the invertible sheaf $(g_* \cO_X)^{\vee}$, defines a nowhere vanishing section
 \[
 \Disc: \cO_Y \to (g_* \cO_X)^{\otimes -2} 
 \] A section of the square of an invertible sheaf such as $ (g_* \cO_X)^{\otimes -2} $ canonically defines a section $\Disc(g) \in \Gamma(\cO_Y/ (\cO^*_Y)^2)$ by taking local trivializations of the invertible sheaf to identify the section with a section of $\cO_Y$. Since $g$ is \'etale, $\Disc(g)$ is nowhere vanishing. 
 
 For a point $y$ of $Y$, we have $\Disc (X \to Y)(y)$ in $k(y)/(k(y)^*)^2$ and 
 \[
 \Disc (X \to Y)(y) = \Disc (X \times_Y \Spec k(y) \to \Spec k(y)).
 \]
 \end{rem}
 
 By Proposition~\ref{Double_point_locus_our_moduli_finite_etale}, $\sB \to \M_0(\cX(d),D)_{\sigma}(\widetilde{p})$ is finite \'etale. Thus, we have
 \[
 \Disc(\sB \to \M_0(\cX(d),D)_{\sigma}(\widetilde{p})) \in \Gamma(\cO^*/(\cO^*)^2)
 \]
 where $\cO$ abbreviates the sheaf of rings $\cO_{\M_0(\cX(d),D)_{\sigma}(\widetilde{p})}$. 
 
  Let $f$ be a point of the scheme $\M_0(\cX(d),D)_{\sigma}(\widetilde{p})$ with residue field $k(f)$. Let $\pi: \M_0(\cX(d),D)_{\sigma}(\widetilde{p}) \to \Spec F[[t]]$ and let $k(\pi(f))$ be the residue field of the image of $f$ under $\pi$. In particular, $k(\pi(f))$ is either $F$ or $F((t))$. Suppose $\ell$ is an intermediate field extension $k(\pi(f)) \subseteq \ell \subseteq k(f)$. Since $\pi$ is finite and \'etale, it follows that $\ell \subseteq k(f)$ is finite separable.
  
 \begin{defi}
 Define the {\em weight} of $f$ relative to $\ell$ to be 
 \[
 \wt_{\ell} f : = \Tr_{k(f)/\ell} \langle \Disc (\sB \to \M_0(\cX(d),D)_{\sigma}(\widetilde{p}))(f) \rangle.
 \]
 \end{defi}
 
 \begin{rem}
The weight admits the following interpretations. The discriminant has a multiplicative property in towers of field extensions (see e.g. \cite[Lemma 7.5]{degree}) from which one sees (\cite[Proposition 7.4]{degree}) that 
\[
 \wt_{\ell} f := \Tr_{k(f)/\ell} \prod_{x \text{ node}\text{ of }f(C)} \operatorname{mass}(x).
\] Suppose $f$ maps to the generic point of $k[[t]]$ and $\Sigma(d)$ satisfies Hypothesis~\ref{hyp:SDk_with_N}, so we are in the case where $\ev_{\sigma}$ and its pullbacks have well-defined $\mathbb{A}^1$-degrees by the main theorems of \cite{degree}. Then by \cite[Proposition 7.3]{degree}, $\wt_{\ell} f$ is the local $\mathbb{A}^1$-degree of $(\ev \otimes \ell)_{\sigma}$ at $f$. The global degree $N_{\Sigma,D, \sigma}$ is the sum of the local degrees (\cite[Proposition 3.2]{degree}), so we have 
\begin{align*}
N_{\Sigma(d)_\ell,D_\ell, \sigma_\ell}& = \sum_{\substack{f \text{ degree } D \\ \text{ rational curve on } \Sigma(d)_\ell \\\text{ thru } \widetilde{p} }} \Tr_{k(f)/\ell}\prod_{x \text{ node}\text{ of }f(C)} \operatorname{mass}(x)\\
& = \sum_{\substack{f \text{ degree } D \\ \text{ rational curve on } \Sigma(d) \\\text{ thru } \widetilde{p} }} \wt_{\ell} f.
\end{align*}
 \end{rem}
  
  Recall there is a canonical injection $\GW(F) \to \GW(F((t)))$ given by basechange. We use this injection to view an element of $\GW(F)$ as an element of $\GW(F((t)))$. We show that for stable maps associated to points of the generic fiber of $\M_0(\cX(d),D)_{\sigma}(\widetilde{p})$, the weight does not change under specialization.
  
  \begin{prop}\label{prop:invariance_NA1_weight}
Let \[
f_0: C_1 \cup C_2 \to \cX(d)_0
\]  be a stable map as described in Lemma~\ref{lm:describe_specialized_curves} equal to the specialization of a point of $\M_0(\cX(d),D)_{\sigma}(\widetilde{p})$ corresponding to a degree $D$ stable map
\[
f: C \to \Sigma(d)
\] through the points determined by $\widetilde{p}$. Suppose additionally that $\widetilde{p}$ is a rational point in the open dense subset $\widetilde{V}$ of Hypothesis \ref{NA1-enumerative-general-finitel-special} \eqref{it:NA1-enumerative-general}. Then the weight of $f$ equals the weight of $f_0$. 
  \end{prop}
  
  \begin{proof}
  By Theorem~\ref{thm:moduli-stack-finite-etale}, $\M_0(\cX(d),D)_{\sigma}(\widetilde{p}) \to \Spec F[[t]]$ is finite \'etale. By \cite[X Th\'eor\`eme 2.1]{sga1}, finite \'etale $F[[t]]$-algebras are of the form $F[[t]] \to A[[t]]$ where $F \to A$ is a finite \'etale $F$-algebra. Express $A$ as a product of finite separable field extensions $F \subseteq A_i$ of $F$, that is $A \cong \prod_{i=1}^m A_i$. Then $f$ and $f_0$ are the maps corresponding to the generic point and closed point of $A_i[[t]]$ for some $i$. 
  
 Since $F[[t]] \subseteq A_i[[t]]$ is finite flat, we have the transfer \[
 \Tr_{A_i[[t]]/F[[t]]}:\GW(A_i[[t]]) \to \GW(F[[t]])
 \] taking a symmetric unimodular bilinear form $\beta: V \times V \to A_i[[t]]$ to the composite 
 \[
 V \times V \to A_i[[t]] \xrightarrow{\Tr} F[[t]]
 \] where $\Tr: A_i[[t]] \to F[[t]]$ denotes the trace map of the finite \'etale extension $F[[t]] \subseteq A_i[[t]]$. See also \cite[Section 2.1]{degree}. For $\ell$ an intermediate field $k\subseteq \ell \subseteq A_i$, we have 
 \[
  \Tr_{A_i[[t]]/\ell[[t]]}:\GW(A_i[[t]]) \to \GW(\ell[[t]])
 \] in the same way.
 
 Let $\sB_i$ denote the pullback of the double point locus to the component $\Spec A_i[[t]]$ of $\M_0(\cX(d),D)_{\sigma}(\widetilde{p})$
 \[
 \sB_i = \sB \times_{\M_0(\cX(d),D)_{\sigma}(\widetilde{p})} \Spec A_i[[t]].
 \] By Proposition~\ref{Double_point_locus_our_moduli_finite_etale}, $\sB_i \to \Spec A_i[[t]]$ is finite \'etale.
 We thus have $\langle \Disc(\sB_i \to \Spec A_i[[t]]) \rangle$ in $\GW(\Spec A_i[[t]])$ whence
 \[
  \wt_{i,\ell} := \Tr_{A_i[[t]]/\ell[[t]]} \langle \Disc(\sB_i \to \Spec A_i[[t]]) \rangle \in \GW(\ell[[t]]).
 \] The pullbacks of $\wt_{i,\ell}$ to the generic and closed points are $\wt_{\ell} f$ and $\wt_{\ell} f_0$ respectively. Since $\GW(\ell[[t]]) \cong \GW(\ell)$ we have 
 \[
\wt_{\ell} f = \wt_{\ell} f_0
 \] as claimed.
  \end{proof}

  \subsection{Stability under basechange}\label{subsection:stable_basechange}
  
  The quadratic Gromov--Witten invariants are stable under basechange, as we now show. In fact, a more general basechange result holds for enumerative del Pezzo surfaces as follows. By Lemma~\ref{lm:enumerative_stable_basechange}, the enumerative property of Definition~\ref{df:general_fiber_enumerative} is stable under basechange and the corresponding open dense subsets $\widetilde{V} \subset \Res_{\sigma/k} S$ pullback. The pullback $\ev_{\widetilde{V}}$ of the twisted evaluation map 
\[
\ev_{S,\sigma}:  \Mbar_{0,n}(S,D)_{\sigma} \to \Res_{\sigma/k} X
\] to $\ev_{\sigma}^{-1}\widetilde{V}$ is \'etale so its cotangent complex $L$ vanishes and $\omega:=\det L$ has a canonical trivialization $c:\omega \xrightarrow{\cong} \cO_{\Mbar_{0,n}(S,D)_{\sigma}}$. There is a double point locus 
\[
\varpi: \sB \to \ev_{\sigma}^{-1}\widetilde{V}
\] as in Section~\ref{subsection:double_point_locus} defined in \cite[Section 5]{KLSW-relor}. The map $\varpi$ is moreover finite \'etale because $\ev_{\sigma}^{-1}\widetilde{V}$ parametrizes stable maps in the ordinary double point locus (by the same proof as Proposition~\ref{Double_point_locus_our_moduli_finite_etale}), from which it follows that the pullback $\ev_{\widetilde{V}}$ admits a double-point orientation 
\begin{equation}\label{eq:double_point_orientation}
\omega \xrightarrow{c}   \cO_{\Mbar_{0,n}(S,D)_{\sigma}} \xrightarrow{\Disc} (\varpi_*\mathcal{O}_{\sB})^{\otimes 2},
\end{equation} where $\Disc$ is the disriminant recalled in Remark~\ref{Rm:disc}. This orientation was discussed in detail in  \cite{degree} \cite[Section 6]{KLSW-relor} under more restrictive hypotheses, however it is immediate that the map is an isomorphism on $\ev_{\sigma}^{-1}\widetilde{V}$, because $\ev_{\sigma}$ is \'etale on this locus and the double point locus $\varpi$ is finite \'etale above it. In other words, for enumerative del Pezzo surfaces, there is a double point orientation defined away from codimension $1$ on the target.

Quadratic  invariants are defined by taking degrees. Let $p: \Spec k(p) \to \widetilde{V}$ be any point and let $\ev_{\widetilde{V}} \otimes_{\widetilde{V}} k(p)$ denote the pullback of $\ev_{\widetilde{V}}$, which inherits an orientation by \cite[Proposition 2.7]{degree}. Oriented, finite \'etale maps admit an $\A^1$-degree (\cite[Section 2.1]{degree}), defining 
\[
N_{S,D,\sigma}(p):= \deg (\ev_{\widetilde{V}} \otimes_{\widetilde{V}} k(p) ) \in \GW(k(p)), \quad \quad \text{ and }
\]
\[
N_{S,D,\sigma}:= \deg (\ev_{\widetilde{V}}) \in \sGW(\ev^{-1}(\widetilde{V}))
\] where $\sGW(\ev^{-1}(\widetilde{V}))$ denotes the sections of the unramified Grothendieck--Witt sheaf on $\ev^{-1}(\widetilde{V})$. 

\begin{prop}\label{pr:NA1-stable-basechange-at-point}\label{pr:NA1-stable-basechange}
Suppose $S$ is a del Pezzo surface over a field $k$. Let $D$ be in $\Pic S$, and let $k \to \sigma$ be a finite \'etale extension of degree $-K_S \cdot D -1$. Suppose that $(S, D)$ is enumerative with open dense subset $\widetilde{V} \subset \Res_{\sigma/k} S$. Let $k \to E$ be a field extension. Then  $(S_E, D_E)$  is enumerative and we have 
\begin{enumerate}
\item \label{eq:NA-stable-basechange-at-point} For any point $p_E$ of $\ \widetilde{V}_E$ with image $p$ in $\widetilde{V}$  we have
\[
N_{S_E,D_E,\sigma_E}(p_E) = N_{S,D,\sigma}(p) \otimes_{k(p)} k(p_E) \in \GW(k(p_E)).
\] 
\item \label{eq:NA-stable-basechange} If $N_{S,D,\sigma}$ is pulled back from $\GW(k)$, then $N_{S_E,D_E,\sigma_E}$ is pulled back from $\GW(E)$ and 
\[
N_{S_E,D_E,\sigma_E} = N_{S,D,\sigma} \otimes_{k(p)} E \in \GW(E).
\]
\end{enumerate}
\end{prop}

\begin{proof}
By \cite[068E,08QL]{stacks-project}, orientations of local complete intersection morphisms pullback under flat basechange. In particular, the double point orientation \eqref{eq:double_point_orientation} of the restricted evaluation map $\ev_{\widetilde{V}}$ associated to $X$ pulls back to an orientation of $\ev_{\widetilde{V}'}$. We obtain induced orientations of $\ev_{\widetilde{V}} \otimes_{\widetilde{V}} k(p)$ and $\ev_{\widetilde{V}'} \otimes_{\widetilde{V}} k(p_E)$. By \cite[Proposition 2.7]{degree}, it follows that  
\[
 \deg (\ev_{\widetilde{V}'} \otimes_{\widetilde{V}'} k(p_E) ) = \deg (\ev_{\widetilde{V}} \otimes_{\widetilde{V}} k(p) ) \otimes k(p_E) \in \GW(k(p_E))
\] where the degree on the left hand side is taken with respect to the pullback orientation. Because taking discriminants commutes with basechange, the double point orientation \eqref{eq:double_point_orientation} of $\ev_{\widetilde{V}}$ pulls back to the double point orientation of $\ev_{\widetilde{V}'}$, completing the proof of \eqref{eq:NA-stable-basechange-at-point}. 

By the same argument, we have that the image of $N_{S,D,\sigma}$ under the restriction map
\[
\sGW(\ev^{-1}(\widetilde{V})) \to \sGW(\ev^{-1}(\widetilde{V}'))
\] is $N_{S_E,D_E,\sigma_E}$. Then \eqref{eq:NA-stable-basechange} follows from the commutative diagram
\[
\xymatrix{
  \GW(k) \ar[r] \ar[d] & \GW(E) \ar[d] \\
  \sGW(\ev^{-1}(\widetilde{V})) \ar[r] &  \sGW(\ev^{-1}(\widetilde{V}')).
}
\]
\end{proof}

  \subsection{Proof of the wall-crossing formula}

 Choose a field extension $\ell$ of $k$. For any integer $N$, let $\Fet_{/\ell}(N)$ be the set of finite, \'etale $\ell$-algebras of degree $N$. Let $f_0: C_1 \cup C_2 \to \cX(d)_0$ be as in Lemma~\ref{lm:describe_specialized_curves} such that $k(f_0) = \ell$. In particular, we have a stable map $f_1:C_1\to \widetilde S$ transverse to $E$ defined over $\ell$ of degree $D_0$ in $\Pic \widetilde S$. 
 
 \begin{defi}\label{df:IntersectionProfile}
 We define the {\em intersection profile of $f_1$}, denoted $\Int_E(f_1)$, in $\Fet_{/\ell} (D_0 \cdot E)$ to be the (finite, \'etale, degree $D_0 \cdot E$) intersection of $f_1(C_1)$ with $E$.  
\end{defi} 
 
$f_0$ also has an associated stable map $f_2: C_2 \to Q(d) $ defined over $\ell$ which is transverse to $E$. We may also define the intersection profile of $f_2$ to be the finite, \'etale, degree $D_0 \cdot E$, $\ell$-algebra given by the intersection of $f_2(C_2)$ with $E$. Note that the intersection profile of $f_2$ is then canonically isomorphic to the intersection profile of $f_1$ and we may also call both the intersection profile of $f_0$. In symbols, we have  $\Int_E(f_0)=\Int_E(f_1) = \Int_E(f_2)$.

Moreover, we have $(C_2)_{k^s} \cong \coprod^{D_0 \cdot E} \mathbb{P}^1_{k^s}$ and $(f_2)_{k^s}$ maps each $\mathbb{P}^1_{k^s}$ component of $(C_2)_{k^s}$ to one of the rulings of $Q(d)_{k^s} \cong \mathbb{P}^1_{k^s}\times \mathbb{P}^1_{k^s}$. Let $\ell \subseteq \overline{\ell}$ denote an algebraic closure of $\ell$ and let 
\[
F_{\overline{\ell}}: \Fet_{/\ell}(D_0 \cdot E) \to \Gal(\overline{\ell}/\ell)-\Set
\]
 denote the restriction of the fiber functor, as in Section~\ref{Section:twisted_binomial_coefficients_section}. Every element of $F_{\overline{\ell}}(\Int_E(f_Q))$ determines a geometric point of the intersection of $f_2(C_2)$ with $E$. Since each geometric point of $f_2(C_2) \cap E$ either belongs to a component of $(C_2)_{k^s}$ mapped to the $(1,0)$-ruling or to the $(0,1)$-ruling, we have a map 
 \[
 \psi_{f_2}: F_{\overline{\ell}}(\Int_E(f_2)) \to \{(1,0), (0,1) \}
 \] which assigns the corresponding ruling to each element of $F_{\overline{\ell}}(\Int_E(f_Q))$. 
 
\begin{defi}
 Let the {\em ruling profile} of $f_2$, denoted $\Rul(f_2)$, be the subset of $F_{\overline{\ell}}(\Int_E(f_2)) $ given by $\psi_{f_2}^{-1}((1,0))$.
\end{defi} 

\begin{defi}\label{df:Nsigmaell}
For a field extension $\ell$ of $k$ and a finite \'etale $\ell$-algebra $\sigma'$ of degree $D_0 \cdot E$, define $N^{\sigma',\ell}_{\widetilde S,D_0,\sigma}(p)$ to be 
\[
N^{\sigma',\ell}_{\widetilde S,D_0,\sigma}(p) = \sum_{\substack{f_1 \text{ degree } D_0\\  \text{ rational curve on } \widetilde{S}\\k(f_1) = \ell \\ \Int_E(f_1)=\sigma' \\\text{ thru } p }} \wt_{\ell} f_1.
\]
the sum of quadratic weights over all rational stable maps with field of definition $\ell$, degree $D_0$, through $p$, with intersection profile $\sigma'$. 
\end{defi}

Note that $N^{\sigma',\ell}_{\widetilde S,D_0,\sigma}(p)$ may depend on the choice of $p$. 

\begin{thm}\label{thm:N_degeneration2}
Let $\cX \to \Spec k[[t]]$ be a $1$-nodal Lefschetz fibration of del Pezzo surfaces, and $d \in k^*$. Let $D \in \Pic \Sigma(d)$ be the class of a Cartier divisor, and $k \to \sigma$  a finite \'etale extension of degree $-K_{\Sigma(d)} \cdot D -1$.
 Suppose  $(\Sigma(d), D)$ satisfies Hypothesis \ref{NA1-enumerative-general-finitel-special}. Let $p=( p_1,\ldots,p_r)$ be a point in the open dense subset $W$ of Lemma~\ref{lm:points_specializing_try2} applied to open sets $V$ and $\widetilde{V}$ as in Definitions~\ref{def:special_fiber_enumerative} and \ref{df:general_fiber_enumerative}. Assume that $k \subseteq k(p) =:F$ is finite and separable.

  Then for all $\widetilde{p}=(\widetilde p_1,\ldots,\widetilde p_r)$ on $\Sigma(d)$ which specialize to $p$, there is an equality in $\GW(F) \hookrightarrow \GW(F((t)))$  given by
  \begin{align*}
   N_{\Sigma(1),D,\sigma((t))}(\widetilde{p})
   & = \sum_{\ell} Tr_{\ell/F} \big(\sum_{\substack{\sigma' \in \Fet_{/\ell} \\ (D_0,(i,j))\in \varphi_1^{-1}(D)}} {\sigma'/\ell \choose j}N^{\sigma',\ell}_{\widetilde S,D_0,\sigma}(p) \big). 
  \end{align*}
 If $d\notin (k^*)^2$, then 
  \begin{align*}
      N_{\Sigma(d),D,\sigma((t))}(\widetilde{p})
      & =\sum_{\ell} Tr_{\ell/F}\big( \sum_{\substack{\sigma' \in \Fet_{/\ell} \\ (D_0,(j,j))\in \varphi_d^{-1}(D)}} {\sigma'[\sqrt{d}]/k \choose j}\langle d^j \rangle N^{\sigma',\ell}_{\widetilde S,D_0,\sigma}(p)\big). 
     \end{align*}
  \end{thm}
  
 \begin{rem}
Recall that $N_{\Sigma(d),D,\sigma((t))}(\widetilde{p})$  
was defined in Section~\ref{subsection:stable_basechange}.
 \end{rem} 
  
 \begin{proof}
  Let $\widetilde{p}=(\widetilde p_1,\ldots,\widetilde p_r)$ be a collection of closed points on $\Sigma(d)$ which specialize (as in Remark \ref{rem:spec_of_pi}) to $p=( p_1,\ldots,p_r)$, defining the moduli space $\M_0(\cX(d),D)_{\sigma}(\widetilde{p})$.  By Lemma  \ref{lm:describe_specialized_curves}, Lemma \ref{lemma:unique_generalization}, and Proposition \ref{prop:invariance_NA1_weight}, we know that $ N_{\Sigma(d),D,\sigma}$ is the sum of the quadratic weight over all stable maps described in Lemma \ref{lm:describe_specialized_curves}.
  
   Let $f'_1:C'_1\to \widetilde S$ be a stable map passing through $p$ with $C'_1$ an irreducible rational curve, and $f_1'$ transverse to $E$ of degree $D_0$ in $\Pic \widetilde S$. Let $k(f'_1)$ denote the residue field of the point of the moduli space of curves determined by $f'_1$. 
   
    Let $\sigma' = \Int_E(f_1')$ be the intersection profile of $f_1'$. Suppose $f_0:C_1\cup C_2\to \cX(d)_0$ is the specialization of a stable map to $\Sigma(d)$ of genus 0, degree $D$, passing through $\widetilde p$, such that $f_1=(f'_1)_{k(f_1)}.$ By Lemma~\ref{pr:varphi_surjective}, we have $\varphi_d^{-1}(D) = \{ (D_0, (i,j)) + \mathbb{Z}(-E, (1,1)) \}$. Thus the degree of such a map $f_0$ in $\Pic \cX(d)_0$ is $(D_0,(i,j))$, with $(i,j)$ determined by $D_0$ (and the fixed $D$). Moreover $k(f_1') \subseteq \sigma'$ is an extension of degree $i+j$.
   
  Let us suppose first that $d=1$. Then the nodes of $f_2 = f_0\vert_{C_2}$ are all split. Thus the double point locus over $f_0$ is the pullback of the double point locus over $f_1'$ disjoint union a finite \'etale cover with trivial discriminant. Thus \begin{equation}\label{discDd_0d=1}
   \Disc (\D_{f_0} \to \Spec k(f_0)) = \Disc (\D_{f_1'} \to \Spec k(f_1'))
  \end{equation} where the right hand side is implicitly pulled back from $\GW(k(f_1'))$ to $\GW(k(f_0))$.

Since the weight of $f_0$ over a field extension $k'$ of $k$ is $$\wt_{k'}(f_0) =\Tr_{k(f_0)/k'} \Disc (\D_{f_0} \to \Spec k(f_0)),$$ we have
\begin{align*}
\wt_{k(f_1')}(f_0) &= \Tr_{k(f_0)/k(f_1')} \Disc (\D_{f_1'} \to \Spec k(f_1'))\\
& = (\Disc (\D_{f_1'} \to \Spec k(f_1')) \Tr_{k(f_0)/k(f_1')}\langle 1\rangle \\
&=\wt_{k(f_1')}(f_1')  \Tr_{k(f_0)/k(f_1')}\langle 1\rangle
\end{align*} where the first equality follows from \eqref{discDd_0d=1}.

The field of definition $k(f_0)$ of $f_0$ is the extension field of $k(f_1')$ determined by the subgroup of the absolute Galois group of $k(f_1')$ stabilizing the ruling profile $\Rul(f_2)$ of $f_2$. For a subset $R$ of the $\Gal(\overline{k(f_1')}/k(f_1'))$-set $\Int_{E}(f_1)$, let $L(R)$ denote the subfield $k(f_1')\subseteq L(R) \subseteq \overline{k(f_1')}$ corresponding to the stabilizer $\Stab(R)$ of $R$ in $\Gal(\overline{k(f_1')}/k(f_1'))$. In symbols, $$L(R) = \overline{k(f_1')}^{\Stab(R)}.$$ 

Hence, denoting by $\sB(f'_1)$ the set of all stable maps $f_0:C_1\cup C_2\to \cX(d)_0$ as in Lemma  \ref{lm:describe_specialized_curves} such that $f_0 \vert_{C_1} = (f_1')_{k(f_0)}$, we have
 \begin{align*}
 \sum_{f_0\in \sB(f'_1)} \wt_{k(f_1')}(f_0)& = \wt_{k(f_1')}(f_1') \sum_{f_0\in \sB(f'_1)}  \Tr_{k(f_0)/k(f_1')}\langle 1\rangle \\
 & = \wt_{k(f_1')}(f_1') \sum_{\substack{R \subseteq \Int_E(f_1'): \\ \vert R \vert =j}}  \Tr_{L(R)/k(f_1')}\langle 1\rangle \\
 &=  \wt_{k(f_1')}(f_1')  {\sigma'/k(f_1') \choose j}.
 \end{align*} 
 
 Let $\widetilde{\sC}(D_0,\sigma',\ell)$ denote the set of stable rational maps $f'_1:C_1\to \widetilde{S}$ of degree $D_0$, passing through $p$, and such that $k(f'_1)=\ell$ and $\Int_E(f_1' ) =\sigma'$.
 By Definition, we have that
 \[
 N^{\sigma',\ell}_{\widetilde S,D_0,\sigma}(p) = \sum_{f'_1\in \widetilde{\sC}(D_0,\sigma',\ell)} \wt_{\ell}(f_1')
 \]

Given $(D_0,(i,j))$ in $\Pic \cX(d)_0$, let $\sC(D_0,(i,j))$ denote the set of stable rational maps $f_0:C_1\cup C_2\to \cX(d)_0$ of degree $(D_0,(i,j))$ as described in Lemma~\ref{lm:describe_specialized_curves}.

 \begin{align*}
   N_{\Sigma(1),D,\sigma((t))}(\widetilde{p}) & = \sum_{(D_0,(i,j))\in \varphi_1^{-1}(D)}\sum_{\substack{f_0 \text{ in }\\\sC(D_0,(i,j))}} \Tr_{k(f_0)/F} \wt_{k(f_0)} (f_0)\\
   & = \sum_{(D_0,(i,j)\in \varphi_1^{-1}(D)}\sum_{\ell,\sigma'} \sum_{f'_1\in \widetilde{\sC}(D_0,\sigma',\ell)} \sum_{f_0\in\sB(f'_1)} \Tr_{k(f_0)/F}( \wt_{k(f_0)}(f_0))\\
   & = \sum_{(D_0,(i,j)\in \varphi_1^{-1}(D)}\sum_{\ell,\sigma'} \sum_{f'_1\in \widetilde{\sC}(D_0,\sigma',\ell)} \sum_{f_0\in\sB(f'_1)} \Tr_{\ell/F}\circ \Tr_{k(f_0)/\ell}( \wt_{k(f_0)}(f_0) )\\
   & = \sum_{(D_0,(i,j)\in \varphi_1^{-1}(D)}\sum_{\ell,\sigma'}\Tr_{\ell/F}( \sum_{f'_1\in \widetilde{\sC}(D_0,\sigma',\ell)} \sum_{f_0\in\sB(f'_1)} \Tr_{k(f_0)/\ell}( \wt_{k(f_0)}(f_0) ))\\
   & = \sum_{(D_0,(i,j)\in \varphi_1^{-1}(D)}\sum_{\ell,\sigma'}\Tr_{\ell/F}(  \sum_{f'_1\in \widetilde{\sC}(D_0,\sigma',\ell)}\wt_{\ell}(f_1')  {\sigma'/\ell \choose j})\\
   & = \sum_{(D_0,(i,j)\in \varphi_1^{-1}(D)}\sum_{\ell,\sigma'} Tr_{\ell/F} \big({\sigma'/\ell \choose j}N^{\sigma',\ell}_{\widetilde S,D_0,\sigma}(p) \big)\\
   & = \sum_{\ell} Tr_{\ell/F} \big(\sum_{\substack{\sigma' \in \Fet_{/\ell} \\ (D_0,(i,j))\in \varphi_1^{-1}(D)}} {\sigma'/\ell \choose j}N^{\sigma',\ell}_{\widetilde S,D_0,\sigma}(p) \big). 
  \end{align*}

Now suppose that $d \notin (k^*)^2$. Then $i=j$ by \eqref{PicQ(d)}. Furthermore, the set of nodes of $f_2$ consists of $i^2$ geometric points. These nodes form $v_1,\ldots, v_{m_{f_Q}}$ closed points of $Q(d)$ such that $\sum_i [k(v_i): k] = i^2$. For each node $v_i$ the two tangent directions of $v_i$ give the extension $k(v_i) \subseteq k(v_i)[\sqrt{d}]$ because the two rulings $(1,0)$ and $(0,1)$ of $Q$ form these two tangent directions and they are swapped by elements of the Galois group exchanging $\pm \sqrt{d}$. Thus
\[
 \operatorname{mass}(v_i) = \langle d\rangle^{[k(v_i): k(f_0)]} \quad \quad \text{and} \quad \quad \prod_i \operatorname{mass}(v_i) =  \langle d\rangle^{i^2}.
\] This gives
\begin{align*}
\wt_{k(f_1')}(f_0) &= \Tr_{k(f_0)/k(f_1')} \Disc (\D_{f_1'} \to \Spec k(f_1'))\\
& = (\Disc (\D_{f_1'} \to \Spec k(f_1'))  \langle d\rangle^{i^2} \Tr_{k(f_0)/k(f_1')}\langle 1\rangle \\
&=\wt_{k(f_1')}(f_1') \langle d\rangle^{i} \Tr_{k(f_0)/k(f_1')}\langle 1\rangle.
\end{align*} 

The absolute Galois group of $k(f_1')$ acts on the set of subsets of $\Int_E(f_Q)$ through the quotient 
\[
\Gal(\overline{k(f_1')}/ k(f_1')) \to \Gal(\overline{k(f_1')}/ k(f_1')) \times \Gal(k(f_1')[\sqrt{d}]/ k(f_1'))
\] where $\Gal(\overline{k(f_1')}/ k(f_1'))$ acts on $\Int_E(f_Q) = \Int_E(f_1')$ and a non-trivial element of $\Gal(k(f_1')[\sqrt{d}]/ k(f_1'))$ takes a subset to its complement. The field of definition $k(f_0)$ of $f_0$ is the extension field of $k(f_1')$ determined by the subgroup of the absolute Galois group of $k(f_1')$ stabilizing the ruling profile $\Rul(f_2)$ with this twisted action. For a subset $R$ determining an element of the $\Gal(\overline{k(f_1')}/k(f_1'))$-set of subsets of $\Int_{E}(f_1)$, let $L_d(R)$ denote the subfield $k(f_1')\subseteq L_d(R) \subseteq \overline{k(f_1')}$ corresponding to the stabilizer $\Stab_d(R)$ of $R$ in $\Gal(\overline{k(f_1')}/k(f_1'))$ in the twisted action. In symbols, $$L_d(R) = \overline{k(f_1')}^{\Stab_d(R)}.$$ 

Now compute \begin{align*}
 \sum_{f_0\in \sB(f'_1)} \wt_{k(f_1')}(f_0)& = \wt_{k(f_1')}(f_1')  \langle d\rangle^{i} \sum_{f_0\in \sB(f'_1)}  \Tr_{k(f_0)/k(f_1')}\langle 1\rangle \\
 & = \wt_{k(f_1')}(f_1')  \langle d\rangle^{i} \sum_{\substack{R \subseteq \Int_E(f_1'): \\ \vert R \vert =j}}  \Tr_{L_d(R)/k(f_1')}\langle 1\rangle \\
 &=  \wt_{k(f_1')}(f_1')  \langle d\rangle^{i}  {\sigma'[\sqrt{d}] /k(f_1') \choose j}.
 \end{align*} 
The remainder of the proof follows the case of $d=1$.
 \end{proof} 

\begin{thm}\label{thm:N-surgery}
Let $\cX \to \Spec k[[t]]$ be a $1$-nodal Lefschetz fibration of del Pezzo surfaces with vanishing cycle $\gamma$, and  let $d \in k^*$.  
Let $D \in \Pic \Sigma(d)$ be the class of a Cartier divisor, and $k \to \sigma$  a finite \'etale extension of degree $-K_{\Sigma(d)} \cdot D -1$. Suppose $(\Sigma(d), D)$ satisfies Hypothesis \ref{NA1-enumerative-general-finitel-special}. Let $p=( p_1,\ldots,p_r)$ be a point in the open dense subset $W$ of Lemma~\ref{lm:points_specializing_try2} applied to open sets $V$ and $\widetilde{V}$ as in Definitions~\ref{def:special_fiber_enumerative} and \ref{df:general_fiber_enumerative}. Assume that $k \subseteq k(p) =:F$ is finite and separable. Suppose that for all $j\in \Z$, $N_{\Sigma(1),D-j\gamma,\sigma}(\widetilde{p}) $ is independent of the choice of $p$ in $V$ and lift $\widetilde{p}$ in $\widetilde{V}$ of $p$. Then the same is true for $\Sigma(d)$ and we have the equality in $\GW(F)\hookrightarrow \GW(F((t)))$

  \begin{align*}
  N_{\Sigma(d),D,\sigma((t))}
  & = N_{\Sigma(1),D,\sigma((t))}+
 (\langle 2 \rangle - \langle 2d \rangle) \sum_{j\ge 1} (-1)^j N_{\Sigma(1),D-j\gamma,\sigma((t))}. 
\end{align*}
\end{thm}

 \begin{proof}
   By Corollary~\ref{cor:mapPicSigma_to_PicS}, we have a canonical inclusion  $\Pic \Sigma(d) \subset \Pic \Sigma(1)$. This inclusion is given as follows. By Proposition~\ref{pr:varphi_surjective}, we may choose an element of $\varphi_d^{-1}(D)$. By Remark 4.14, this element is of the form $(D', (i,i)))$. We have $D' \cdot E = 2i$ by construction of the fiber product $\Pic \widetilde S \times_{\Pic E} (\mathbb{Z} \times \mathbb{Z})$. Setting $D_0 = D' + iE$, we may choose $(D_0, (0,0)) \in \varphi_d^{-1}(D)$.
  The image of $D$ in $\Pic \Sigma(1)$ under the canonical inclusion $\Pic \Sigma(d) \subset \Pic \Sigma(1)$ is $\varphi_1 (D_0, (0,0))$. Note that $$ \varphi_1^{-1}(D - l \gamma) ) = \{ (D_0 - i E, (i+l , i -l) ) : i \in \mathbb{Z} \}.$$

  Let us write 
  \[
   A= N_{\Sigma(1),D,\sigma((t))}+
    (\langle 2 \rangle - \langle 2d \rangle) \sum_{l\ge 1} (-1)^j N_{\Sigma(1),D-l\gamma,\sigma((t))}.
  \]

   Then by Theorem~\ref{thm:N_degeneration2}, we compute
   \begin{align*}
    N_{\Sigma(1),D,\sigma((t))}(\widetilde{p})
    & = \sum_{\ell} Tr_{\ell/F} \big(\sum_{\substack{\sigma' \in \Fet_{/\ell} \\ (D_0,(i,j))\in \varphi_1^{-1}(D)}} {\sigma'/\ell \choose j}N^{\sigma',\ell}_{\widetilde S,D_0,\sigma}(p) \big). 
   \end{align*}
  \begin{align*}
    A&= \sum_{\ell} Tr_{\ell/F} \left(\sum_{\substack{\sigma' \in \Fet_{/\ell} \\ (D_0-iE,(i,i))\in \varphi_1^{-1}(D)}} {\sigma'/\ell \choose i}N^{\sigma',\ell}_{\widetilde S,D_0-iE,\sigma}(p) \right.
    \\ &\left. \qquad \qquad\qquad + (\langle 2 \rangle - \langle 2d \rangle) \sum_{l\ge 1} (-1)^l \sum_{\substack{\sigma' \in \Fet_{/\ell} \\ (D_0-iE,(i+l,i-l))\in \varphi_1^{-1}(D-l\gamma)}} {\sigma'/l \choose i-l}N^{\sigma',\ell}_{\widetilde S,D_0-(i+l)E,\sigma}(p)\right)
    \\ &=  \sum_{\ell} Tr_{\ell/F} \left( \sum_{\substack{\sigma' \in \Fet_{/\ell} \\ i\in \Z}} {\sigma'/k \choose i}N^{\sigma',\ell}_{\widetilde S,D_0-iE,\sigma}(p)\right.
    \\ &\left.\qquad \qquad\qquad + (\langle 2 \rangle - \langle 2d \rangle) \sum_{l\ge 1} (-1)^l \sum_{\substack{\sigma' \in \Fet_{/\ell} \\ i\in \Z}}{\sigma'/k \choose i-l}N^{\sigma',\ell}_{\widetilde S,D_0-iE,\sigma}(p) \right)
     \\ &= \sum_{\ell} Tr_{\ell/F} \left( \sum_{\substack{\sigma' \in \Fet_{/\ell} \\ i\in \Z}}\left( {\sigma'/k \choose i}+ (\langle 2 \rangle - \langle 2d \rangle) \sum_{l\ge 1} (-1)^l  {\sigma'/k \choose i-l}\right)N^{\sigma',\ell}_{\widetilde S,D_0-iE,\sigma}(p)\right)
   \\ &=  \sum_{\ell} Tr_{\ell/F} \left( \sum_{\substack{\sigma' \in \Fet_{/\ell} \\ i\in \Z}} \left( {\sigma'/k \choose i}+ (-1)^i(\langle 2 \rangle - \langle 2d \rangle) \sum_{l=0}^{i-1} (-1)^l  {\sigma'/k \choose l}\right)N^{\sigma',\ell}_{\widetilde S,D_0-iE,\sigma}(p)\right).
  \end{align*}
        Thanks to Corollary \ref{cr:Useful_Binomial_identity} and Theorem~\ref{thm:N_degeneration2}, we obtain
        \begin{align*}
          A &= \sum_{\ell} Tr_{\ell/F}\big( \sum_{\substack{\sigma' \in \Fet_{/\ell} \\ (D_0,(j,j))\in \varphi_d^{-1}(D)}} {\sigma'[\sqrt{d}]/k \choose j}\langle d^j \rangle N^{\sigma',\ell}_{\widetilde S,D_0,\sigma}(p)\big)
          \\&= N_{\Sigma(d),D,\sigma((t))},
        \end{align*}
        and the theorem is proven.
  \end{proof}
  
  \begin{cor}\label{cor:N-surgery-5}
    Let $k$ be a perfect field and let $d \in k^*$. Let 
    $\cX \to \Spec k[[t]]$ be a $1$-nodal Lefschetz fibration of del Pezzo surfaces of degree at least 3 and let $\Sigma(d)$ denote the generic fiber of the $d$-surgery. Let $D$ be the class of a Cartier divisor on $\Sigma(d)$.
    
    Suppose that the component $\widetilde{S}$ of the special fiber is $k$-rational, that $(\widetilde S,D)$ is relatively enumerative, 
     that $\Sigma(1)$ is $k((t))$-rational, and  
    that    
    $\Sigma(1)$ is the basechange of a surface satisfying Hypothesis~\ref{hyp:SDk_with_N}.

    Then for any  finite \'etale $k \to \sigma$ of degree $-K_{\Sigma(d)} \cdot D -1$, we have that $N_{\Sigma(d),D,\sigma((t))}(\widetilde{p})$ and $ N_{\Sigma(1),D,\sigma((t))}(\widetilde{p})$ are independent of the choice of $\widetilde{p}$ (chosen to be a lift in $\tilde{V}$ of a $p$ in $W$ with $[k(p):k]$ odd and $k \subseteq k(p)$ separable which is possible) and in $\GW(k)\hookrightarrow \GW(k((t)))$ we have the equality:

  \begin{align*}
  N_{\Sigma(d),D,\sigma((t))}
  & = N_{\sigma(1),D,\sigma((t))}+
 (\langle 2 \rangle - \langle 2d \rangle) \sum_{j\ge 1} (-1)^j N_{\Sigma(1),D-j\gamma,\sigma((t))} .
\end{align*}
\end{cor}

\begin{rem}
If $k$ is of characteristic $0$, note that  $(\widetilde S,D)$ is relatively enumerative by Proposition~\ref{prop:Stildechar0_enumerative}.
\end{rem}

\begin{rem}\label{rem:independence_tildep_woassuming_connectivity}
Note that $\Sigma(d)$ is not assumed to be $\A^1$-connected and indeed $\Sigma(d)$ is not necessarily rational (over its field of definition which is $k((t))$). See also Example~\ref{ex:enumerative_Sigma-1_not_connected}. We are obtaining independence of $N_{\Sigma(d),D,\sigma((t))}(\widetilde{p})$ on the choice $\widetilde{p}$ from the wall-crossing formula: for any rational point $p$ of $W \subseteq \widetilde S$ and any lift $\tilde{p}$ in $\tilde V$ of $p$, the value of $N_{\Sigma(d),D,\sigma}(\widetilde{p})$ will be the same. Here $W$ denotes the dense open subset of Lemma~\ref{lm:points_specializing_try2} applied to open sets $V$ and $\widetilde{V}$ as in Definitions~\ref{def:special_fiber_enumerative} and \ref{df:general_fiber_enumerative}. There are many rational points of $W$ because $\widetilde S$ is $k$-rational.
\end{rem}

\begin{proof}
It follows from \cite[Theorems 1 and 2]{degree} and basechange that $(\Sigma(1), D)$  satisfies Hypothesis \ref{NA1-enumerative-general-finitel-special} \eqref{it:NA1-enumerative-general} because $\Sigma(1)$ is the basechange of a surface satisfying Hypothesis~\ref{hyp:SDk_with_N}. It follows from Proposition~\ref{lm:KSD=KStildeD} that $(D - j \gamma)\cdot K_{\Sigma(1)} = D \cdot K_{\Sigma(1)}$. Thus $(\Sigma(1),D - j \gamma)$ is the basechange of a surface satisfying Hypothesis~\ref{hyp:SDk_with_N} for all $j$. By Proposition~\ref{pr:NA1-stable-basechange-at-point}, it follows that for all $j$, $ N_{\Sigma(1),D-j\gamma,\sigma}(\widetilde{p})$ is independent of the choice of $\widetilde{p}$. Since $\Sigma(d)$ and $\Sigma(1)$ are isomorphic over $\overline{k((t))}$ and  $(\Sigma(1),D)$ is enumerative, it follows that $(\Sigma(d), D)$ is enumerative. By hypothesis, $(\widetilde{S},D)$ is relatively enumerative. Thus $(\Sigma(d),D)$ satisfies Hypothesis~\ref{NA1-enumerative-general-finitel-special}.  By Propsition~\ref{prop:rational_scheme_points}, we may choose a point $p$ with $[k(p):k]$ odd of $\Res_{\sigma/k} \widetilde{S}$ in $W$. By choosing $[k(p):k]$ relatively prime to the characteristic, we may assume $k \subseteq k(p)$ to be separable. By Lemma~\ref{lm:points_specializing_try2}, we may lift $p$ to $\Sigma(1)$ and $\Sigma(d)$. The result now follows from Theorem~\ref{thm:N-surgery}.
\end{proof}

We now show that the term $\langle 2 \rangle - \langle 2d \rangle$ in Theorems~\ref{thm:N-surgery-intro}~\ref{thm:N_degeneration2} and Corollary~\ref{cor:N-surgery-5} is the difference of the $\A^1$-Euler characteristic of $\Sigma(1)$ and $\Sigma(d)$ up to a factor of $\langle -1 \rangle$ in characteristic $0$. 

\begin{prop}\label{prop:differenceA1ChiLFdP}
Let $\cX \to \Spec k[[t]]$ denote a Lefschetz fibration of del Pezzo surfaces. Suppose the characteristic of $k$ is $0$. Let $\Sigma(d)$ denote the generic fiber of the $d$-surgery. Then 
$ \chi^{\A^1}(\Sigma(1)) -  \chi^{\A^1}(\Sigma(d)) = \langle -1 \rangle (\langle 2 \rangle - \langle 2d \rangle) $
\end{prop}

\begin{proof}
By \cite[Theorem 11.1]{Levine-EC}, $\chi^{\A^1}( Q(d) ) = h+  \langle 2 \rangle + \langle -2d \rangle $ where $h$ denotes the hyperbolic form. It follows that 
\[
 \chi^{\A^1}(Q(1)) -  \chi^{\A^1}(Q(d)) = \langle -1 \rangle (\langle 2 \rangle - \langle 2d \rangle).
\]
We show that 
\begin{equation}\label{eq:chiA1SigmavQ}
 \chi^{\A^1}(\Sigma(1)) -  \chi^{\A^1}(\Sigma(d)) =  \chi^{\A^1}(Q(1)) -  \chi^{\A^1}(Q(d))
\end{equation} which will complete the proof. Let $X$ be a smooth, projective surface over $K$. By \cite{LevineRaksit_MotivicGaussBonnet} \cite[Corollary 2.13]{bachmann_wickelgren} $\chi^{\A^1}(X)$ is represented by the bilinear form induced by Serre duality on the pushforward to $K$ of the Koszul complex
\[
0 \to \wedge^2 \Omega_{X/K} \to \Omega_{X/K} \to \cO_{X/K} \to 0
\] on $X$. The rank of the pushforward of the Koszul complex may be computed over $K^s$ by \cite[18.2.H]{Vakil-RisingSea}. Since $\Sigma(d)$ and $\Sigma(1)$ are isomorphic over $k((t))^s$, it sufices show \eqref{eq:chiA1SigmavQ} in $W(k((t)))$. Let $K$ abbreviate $k((t))$. In $W(K)$, we have that  $\chi^{\A^1}(X)$ is represented by
\[
\beta_{X,K}: H^1(X, \Omega_{X/K} ) \otimes H^1(X, \Omega_{X/K} ) \to H^2(X,  \Omega_{X/K} \wedge  \Omega_{X/K} ) \to K
\] where the second map is the trace map. This trace form pairing commutes with basechange. So letting $\beta_{X,L}$ denote the analogous pairing over $L$ for $L$ an extension of $K$, we have $\beta_{X,L} = \beta_{X,K} \otimes_K L$. Let 
\[
c_1: \Pic (X) \to H^1(X, \Omega_{X/K} )
\] be the first Chern class in Hodge theory, which commutes with basechange \cite[Tag 0FLE]{stacks-project}. Let $B_{X,L}:\Pic (X_L) \times \Pic (X_L) \to \Z$ denote the intersection pairing of $X_L$. By \cite[V Ex 1.8]{Hartshorne}, we have $B_{X,L}(D_1,D_2) = \beta_{X,L}(c_1(D_1), c_1(D_2))$ for all $D_1$, $D_2$ in $\Pic(X_L)$. 

Let $X$ be a del Pezzo surface over $K$. The induced map 
\[
\Pic (X_{K^s}) \otimes_{\Z} K^s \to H^1(X_{K^s}, \Omega_{X/K^s} )
\] is an isomorphism by Lemma~\ref{lm:PicH11}. 

It follows that the bilinear form $\beta_{X,K}$ is represented by $B_{X,K^s} \otimes_{\Z} K^s$ together with its natural descent data $\Theta_{B,X}$ from $K^s$ to $K$ for $X = \Sigma(d)$ or $Q(d)$. Let $[B_{X,K^s} \otimes_{\Z} K^s, \Theta_{B,X}]$ in $W(K)$ denote the class of the bilinear form $B_{X,K^s} \otimes_{\Z} K^s$ together with the descent data. So on the one hand, we have $\beta_{X,K} = [B_{X,K^s} \otimes_{\Z} K^s, \Theta_{B,X}]$. On the other hand, we have short exact sequences
\[
0 \to \Z(E, (-1,-1)) \to\Pic \widetilde S_L \times_{ \Pic E_L} \Pic Q(d)_L \to  \Pic \Sigma(d)_L \to 0  
\] by Proposition~\ref{pr:varphi_surjective} compatible with the intersection pairings on $\Pic$. By definition, we moreover have short exact sequences
\[
0 \to\Pic \widetilde S_L \times_{ \Pic E_L} \Pic Q(d)_L \to\Pic \widetilde S_L \times\Pic Q(d)_L \to  \Pic E_L \to 0 
\] It follows that 
\begin{align*}
[B_{\Sigma(d),K^s} \otimes_{\Z} K^s, \Theta_{B,\Sigma(d)}] & = [B_{Q(d),K^s} \otimes_{\Z} K^s, \Theta_{B,Q(d)}] + [B_{\widetilde{S},K^s} \otimes_{\Z} K^s, \Theta_{B,\widetilde{S}}]\\
& - [B_{E,K^s} \otimes_{\Z} K^s, \Theta_{B,E}]-[K^2 (E, (-1,-1)] 
\end{align*}

It follows that the difference $\beta_{\Sigma(1),K} - \beta_{\Sigma(d),K}$ in $W(K)$ is given by
\begin{align*}
\beta_{\Sigma(1),K} - \beta_{\Sigma(d),K} &=  [B_{Q(1),K^s} \otimes_{\Z} K^s, \Theta_{B,Q(1)}]  -  [B_{Q(d),K^s} \otimes_{\Z} K^s, \Theta_{B,Q(d)}] \\
& = \beta_{Q(1),K} - \beta_{Q(d),K}, 
\end{align*} which shows Equation~\ref{eq:chiA1SigmavQ} as desired.
 
\end{proof}

\begin{lemma}\label{lm:PicH11}
Let $K$ be a field of characteristic $0$. Let $X$ over $K$ be a blow-up of $\P^2_K$. Then 
\[
\Pic (X_{K}) \otimes_{\Z} K \to H^1(X_{K}, \Omega_{X/K} )
\] is an isomorphism.
\end{lemma}

\begin{proof}
First suppose that $K=\mathbb{C}$. By the Lefchetz (1,1)-theorem 
\[
\Pic (X_{K})\to H^1(X_{K}, \Omega_{X/K} ) \cap H^2(X, \mathbb{Z})
\] is surjective, whence $\Pic (X_{K}) \otimes_{\Z} \C \to H^1(X_{K}, \Omega_{X/K} )$ is surjective. $X$ is a blow-up of $\P^2$ at finitely many points. It follows from direct calculation that $\Pic (X_{K}) \otimes_{\Z} \C$ and $H^1(X_{K}, \Omega_{X/K} )$ are both vector spaces of the same dimension. Thus the lemma follows when $K = \mathbb{C}$.

For $K\subset \mathbb{C}$, the lemma follows because the basechange of $\Pic (X_{K}) \otimes_{\Z} K \to H^1(X_{K}, \Omega_{X/K} )$ to $\C$ is an isomorphism by the previous. 

In general, $X$ is obtained by base change from $X'$ defined over a finitely generated field $K'$ with $X'$ also a blow-up of $\P^2_{K'}$. There is an embedding $K' \to \C$. The result follows for $X'$ by the previous case and for $X$ by basechange.
\end{proof}

\begin{rem}\label{rem:Denef_Loeser}
Proposition~\ref{prop:differenceA1ChiLFdP} is presumably true even in characteristic $p$. It is easily checkable in examples from Section \ref{sec:ex quad} and \ref{sec:ex bup}. This would be the case in general if, analogously to the usual topological Euler characteristic, the $\A^1$-Euler characteristic of the central fiber of a semi-stable degeneration could be computed out of the motivic nearby fiber introduced by Denef and Loeser, see \cite[Section 3]{DenLoe01} and \cite[Section 11]{PetSte08}. See \cite{Azouri-thesis} for related results.

\end{rem}

\section{Applications}\label{sec:applications}

We give four sample applications of Theorem \ref{thm:N-surgery}. These are classical when $k=\R$. When $k=\C$ the application regarding Dehn twists is likewise classical. 

\subsection{Rational quadrics in $\P^3_k$}\label{sec:ex quad}
   Given $d\in k^\times/(k^\times)^2$, we denote by $Q(d)$ the quadric surface in $\P^3_k$ defined by the equation
          \[
          x^2-y^2+z^2 -dw^2=0.
          \]
          This notation is consistent with the notation from Section \ref{sec:sugery}. Indeed, a quadric hypersurface in $\P^n_k$ is the vanishing locus of a quadratic form. Since two quadratic forms define the same quadric if and only if they differ by a non-zero multiplicative scalar, they have the same discriminant when $n$ is odd.
In particular the discriminant of a quadric surface in $\P_k^{3}$
is well defined as an element of $k^\times/(k^\times)^2$.
It follows from Proposition \ref{prop:classification_quadric_surfaces_in_P3} that non-singular quadrics in $\P^3_k$ with a non-empty set of rational points are classified by their discriminant: up to a projective change of coordinates, a  quadric $Q$ in $\P^3_k$ with discriminant $d$ and having rational points is $Q(d)$. Recall that $Q(1)$ is isomorphic to $\P_k^1\times \P_k^1$, so that $\Pic Q(1)=\Z e_1\oplus \Z e_2$, where $e_1$ and $e_2$ are the class of the two rulings of $Q(1)$. The class $H=e_1+e_2$ is the class of the hyperplane section, which is a generator of $\Pic Q(d)$ when $d$ is not a square in $k$.

          \begin{thm}\label{thm:abq quad}
           Choose an integer $a\ge 1$, and  Let $k$ be a perfect field of characteristic either 0 or at least $2a+1$. Choose $k\to \sigma$ a finite étale extension of degree $4a-1$. 
            Then for any $d\in  k^\times/(k^\times)^2$, one has
            \begin{align*}
            N_{Q(d),aH,\sigma}
            & = N_{Q(1),aH,\sigma}+
           (\langle 2 \rangle - \langle 2d \rangle) \sum_{j\ge 1} (-1)^j N_{Q(1),(a+j)e_1+(a-j)e_2,\sigma}.
          \end{align*}
          \end{thm}
          \begin{proof}
            We consider 
            the quadric $\cX\to \Spec k[[t]]$ with equation
            \[
              x^2-y^2+z^2 -tw^2=0.
          \]
          This is a $1$-nodal Lefschetz fibration of del Pezzo surfaces for which $\widetilde S$ is the second Hirzebruch surface $\mathbb F_2$. As usual, we denote by  $\Sigma(d)\to \Spec k((t))$ the generic fiber of a $d$-surgery of $\cX\to \Spec k[[t]]$.
          Since $\Sigma(1)=Q(1)_{k((t))}$,  
           the vanishing cycle is $\gamma=\pm (e_1-e_2)\in \Pic Q(1)=\Pic \Sigma(1)$.
           If $k$ is of characteristic zero, then  $(\widetilde S,D)$ is relatively enumerative for any  $D\in\Pic(\widetilde S)$ by Proposition \ref{prop:Stildechar0_enumerative}.
           Note that $\widetilde S$ is a toric surface with $E$ as a component of the toric boundary. Furthermore, any divisor class $\widetilde D\in \Pic \widetilde S$ appearing in $\varphi^{-1}_1((a+j)e_1+(a-j)e_2)$ correspond to a polarization of $\widetilde S$ having the Newton polygon depicted in Figure \ref{fig:Delta F2}.
           \begin{figure}[h!]
            \begin{center}
            \includegraphics[width=4cm, angle=0]{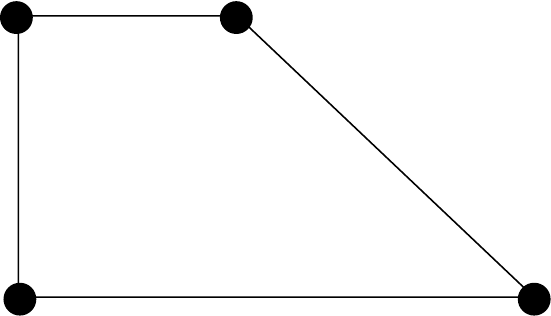}
              \put(-140,0){$(0,0)$}
              \put(-140,60){$(0,b)$}
              \put(-55,60){$(2(a-b),b)$}
              \put(2,0){$(2a,0)$}
            \end{center}
            \caption{}
            \label{fig:Delta F2}
            \end{figure}            
           Hence combining \cite[Propositions 4.12 and 4.13]{Mikhalkin05} with \cite[Theorem 6.2]{Tyomkin12}, we deduce  that $(\widetilde S,\widetilde D)$ is relatively enumerative for all such $\widetilde D\in \Pic \widetilde S$ if $k$ is of characteristic at least  $2a+1$. 
          Hence Corollary~\ref{cor:N-surgery-5} gives
          \begin{align*}
            N_{\Sigma(d),aH,\sigma_{k((t))}}
            & = N_{\Sigma(1),aH,\sigma_{k((t))}}+
           (\langle 2 \rangle - \langle 2d \rangle) \sum_{j\ge 1} (-1)^j N_{\Sigma(1),(a+j)e_1+(a-j)e_2,\sigma_{k((t))}}.
          \end{align*}
          Since $\Sigma(d)=Q(d)_{k((t))}$,  
          by Proposition \ref{pr:NA1-stable-basechange} the invariant $N_{\Sigma(d),D ,\sigma_{k((t))}}$ in $\GW(k((t)))$ is the image of $N_{Q(d),D,\sigma}$ under the natural map $\GW(k)\to \GW(k((t)))$. Since this map is injective by Springer's Theorem, the statement is proved.  
          \end{proof}

\begin{rem}
Remark \ref{rem:non perfect field} allows to partially generalize  Theorem \ref{thm:abq quad} to non perfect fields.
\end{rem}

Recall that a finite étale extension $k\to \sigma$ is \emph{split}, and is denoted by $\sigma_{sp}$, if $\sigma=k^n$. 
The work \cite{PauliPuentes23} of Jaramillo Puentes and Pauli allows one to compute the quadratic Gromov--Witten invariants of toric del Pezzo surfaces when $\sigma$ is split, at least in characteristic $0$ or in sufficiently high characteristic. The answer is of the form $\alpha \langle 1 \rangle + \beta \langle -1 \rangle$ with $\alpha$ and $\beta$ determined by the Gromov--Witten and Welschinger invariants. Let  $GW_{\P^1\times \P^1,D}$ denote the genus $0$ Gromov--Witten invariants of $\P^1\times \P^1$ for the class $D$. Let $W_{Q(d),D,s}$ denote the Weslchinger invariant of $Q(d)$ for the class $D$ and point configurations containing exactly $s$ pairs of complex conjugated points. By \cite[Corollary 1.7]{PauliPuentes23}, one has for any $D\in \Pic(\P^1\times \P^1)$
            \[
              N_{Q(1),D,\sigma_{sp}}= W_{Q(1),D,0}\lra{1} +\frac{GW_{\P^1\times \P^1,D}-W_{Q(1),D,0}}{2}h,
            \]
            where $h=\lra{1}+\lra{-1}\in GW(k)$ is the hyperbolic element, under the above assumption on the characteristic of $k$. 

            Theorem \ref{thm:abq quad} allows to extend this result to the quadric $Q(d)$.
            \begin{cor}\label{cor:quad1-1}
     Let  $D= a(e_1+e_2)\in \subset \Pic Q(d)$, and let $k$ be a perfect field of characteristic either 0 or at least $2a+1$. Then     
     one has
              \[
                N_{Q(d),D,\sigma_{sp}}= W_{Q(-1),D,0} + \frac{W_{Q(1),D,0}-W_{Q(-1),D,0}}{2}(\lra{2}+\lra{2d}) +\frac{GW_{\P^1\times \P^1,D}-W_{Q(1),D,0}}{2} h.
              \]
            \end{cor}
\begin{proof}
  It follows from Theorem \ref{thm:abq quad} combined with \cite[Corollary 1.7]{PauliPuentes23} that, modulo $h$, 
  \begin{align*}
    N_{Q(d),D,\sigma_{sp}}=&\left( W_{Q(1),D,0} + \sum_{j\ge 1} W_{Q(1),D-j(e_1-e_2),0} \right)\lra{1}
   -   (\lra{2}+\lra{2d})\sum_{j\ge 1} W_{Q(1),D-j(e_1-e_2),0}.
  \end{align*}
  By \cite[Theorem 2.1]{Bru18}, i.e. the real version of  Theorem \ref{thm:N-surgery}, we get
  \begin{align*}
    N_{Q(d),D,\sigma_{sp}}= W_{Q(-1),D,0} + \frac{W_{Q(1),D,0}-W_{Q(-1),D,0}}{2}(\lra{2}+\lra{2d}) +\alpha h.
  \end{align*}
The relation
  \[
  \rk (N_{Q(d),D,\sigma_{sp}}) = GW_{\P^1\times \P^1,D}
  \]
  determines the coefficient $\alpha$.
\end{proof}
\begin{rem}
   Corollary \ref{cor:quad1-1} is a particular instance of the more general framework discussed in the forthcoming paper \cite{BruRauWic25}.
\end{rem}
            Combining Corollary \ref{cor:quad1-1} with  computations from \cite{ChenZinger21}, we obtain the first values of $ N_{Q(d),aH,[k,\ldots,k]}$ listed 
            in Table \ref{tab:Qd}.
            \begin{table}[!h]
              \[
              {\setlength{\extrarowheight}{3pt}
              \begin{array}{|c|c|}
              \hline  a & N_{Q(d),aH,\sigma_{sp}}
                \\ \hhline{|=|=|}  1 & \lra{1}
                \\ \hline 2 & 6\lra{1}+ \lra{2}+\lra{2d} +2h=7\lra{2}+\lra{2d}+2h
                \\ \hline 3 & 576\lra{1} +255(\lra{2}+\lra{2d})+1212h= 831\lra{2} + 255\lra{2d}+ 1212h
                \\  \hline  4& 294336 \lra{1} + 262432(\lra{2}+\lra{2d}) + 2844720h=556768 \lra{1} + 262432\lra{d} + 2844720h
                \\ \hline
              \end{array}
              }
              \]
              \label{tab:Qd}
              \caption{Quadratic invariants of $Q(d)$}
            \end{table}
 \begin{rem}
 Note that for $d\neq \pm1 \in  k^\times/(k^\times)^2$, it is impossible to express $N_{Q(d),aH,\sigma_{sp}}$ as an element of the form $\alpha \langle 1 \rangle + \beta \langle -1 \rangle$ with $\alpha,\beta \in \Z$ for $a=2,3$ because the discriminant of $N_{Q(d),aH,\sigma_{sp}}$ is $\pm d\in  k^\times/(k^\times)^2$. 
 \end{rem}           

 \subsection{The blow-up of $\P^2_k$ at two points}\label{sec:ex bup}
            For $d\in k$, define the degree two étale $k$-algebra $\sigma_d=k[X]/(X^2-d)$, i.e. $\sigma_1=k\times k$ and $\sigma_d=k[\sqrt d]$ if $d$ is not a square. Note that $\sigma_d\cong\sigma_{u^2d}$ for any $u\in k^*$. We denote by $\Bl_{d}(\P^2_k)$ the blow-up of $\P^2_k$ along a $0$-dimensional scheme of length 2 and residual étale algebra $\sigma_d$. We denote by $e_0\in \Pic(\Bl_{d}(\P^2_k))$ the pull-back of the class of a line in $\P^2_k$, by $e_1,e_2\in \Pic(\Bl_{1}(\P^2_k))$  the class of the exceptional divisors, and by $f_1\in \Pic(\Bl_{d}(\P^2_k))$ the class of the exceptional divisor when $d$ is not a square. We use the natural identification 
           $ \Pic(\Bl_{d}(\P^2_k))$ as a subgroup of  $ \Pic(\Bl_{1}(\P^2_k))$ via 
         the linear map defined by $e_0\mapsto e_0$ and $f_1\mapsto e_1+e_2$.

          \begin{thm}\label{thm:abq bup}
            Choose two integers $a,b\ge 0$, and  let $k$ be a perfect field of characteristic either 0 or at least $d+1$. Choose $k\to \sigma$ a finite étale extension of degree $3a-2b-1$.
            Then for any $d\in  k^\times/(k^\times)^2$, one has 
                        \begin{align*}
            N_{\Bl_{d}(\P^2_k),ae_0 - bf_1,\sigma}
            & = N_{\Bl_{1}(\P^2)_k,ae_0 -b(e_1+e_2),\sigma}+
           (\langle 2 \rangle - \langle 2d \rangle) \sum_{j\ge 1} (-1)^j N_{\Bl_{1}(\P^2_k),ae_0 -(b+j)e_1 - (b-j)e_2,\sigma}.
          \end{align*}
          \end{thm}
          \begin{proof}
            Let $\cX\to \Spec k[[t]]$ be the blow-up of $\P^2_{k[[t]]}$ along the scheme $(y,x^2-tz^2)$. It is regular with central fiber the blow-up of $\P^2_k$ along the scheme $(y,x^2)$, which is a nodal surface. Hence $\cX$ is a $1$-nodal Lefschetz fibration of del Pezzo surfaces, and $\widetilde S$ is the blow-up of $\P^2_k$ at two infinitely near points.
            If $k$ is of characteristic zero, then  $(\widetilde S,D)$ is relatively enumerative for any  $D\in\Pic(\widetilde S)$ by Proposition \ref{prop:Stildechar0_enumerative}. As in the proof of Theorem \ref{thm:abq bup}, the surface 
           $\widetilde S$ is a toric surface with $E$ as a component of the toric boundary, and $(\widetilde S,\widetilde D)$ is relatively enumerative for classes $\widetilde D\in \Pic \widetilde S$ from $\varphi^{-1}_1(ae_0 -(b+j)e_1 - (b-j)e_2)$ if $k$ is of characteristic at least  $d+1$. 
            Since $\Sigma(1)=(\Bl_{1}(\P^2_{k}))_{k((t))}$,  
            the vanishing cycle is $\gamma=\pm (e_1-e_2)\in \Pic(\Bl_{1}(\P^2_{k}))=\Pic \Sigma(1)$. Hence 
           Corollary \ref{cor:N-surgery-5}   gives in this case
          \begin{align*}
            N_{\Sigma(d),ae_0 - bf_1,\sigma_{k((t))}}
            & = N_{\Sigma(1),ae_0 -b(e_1+e_2),\sigma_{k((t))}}
            \\& \qquad +
           (\langle 2 \rangle - \langle 2d \rangle) \sum_{j\ge 1} (-1)^j N_{\Sigma(1),ae_0 -(b+j)e_1 - (b-j)e_2,\sigma_{k((t))}}.
          \end{align*}
          Since $\Sigma(d)=(\Bl_{d}(\P^2_{k}))_{k((t))}$, the result follows again from  Proposition \ref{pr:NA1-stable-basechange} and the injectivity of the basechange map $\GW(k) \to \GW(k((t)))$.
          \end{proof}
          
          \begin{rem}
            There is an obvious generalization of Theorem \ref{thm:abq bup} and its proof to the blow-up along a scheme of length 2 of a del Pezzo surface of degree at least 4 (so that the blow-up is a del Pezzo surface of degree at least 2). There are nevertheless two issues which a priori   make the corresponding general statement not as pleasant as in the particular case of $\P^2_k$. First, when reducing from $\Sigma(d)$ to $\Bl_d(\P^2_k)$, we use homogeneous coordinates. Next
             there is a positive dimensional moduli space of del Pezzo surfaces of degree at least 4: not all generic configuration of 5 rational points in $\P^2_k$ are projectively equivalent. Even if it sounds highly plausible, it is nevertheless still a conjecture that quadratic invariants of $\P^2$ blown-up along a scheme only depends on the residual étale algebra of this  scheme. To avoid this issue and to keep notations and statements as light as possible, we restrict ourselves in this section to the case of $\P^2$ blown-up along a scheme of length 2.
          \end{rem}

As in Section \ref{sec:ex quad}, we deduce the following, which is again a particular case of \cite{BruRauWic25}. We denote by $GW_{\P^2_\ell,D}$ the genus-0 Gromov-Witten invariant for the class $D$ of the blow-up at $\ell$ points of $\P^2_\C$, and by
            $W_{\P^2_{\ell_1,\ell_2},D,s}$ is the Welschinger invariant  for the class $D$ of the blow-up at $\ell_1$ real points and $\ell_2$ pairs of complex conjugated points of $\P^2_\R$,
            and point configurations containing exactly $s$ pairs of complex conjugated points. (Since the space of such configurations is connected in Euclidean topology and Welschinger and Gromov--Witten invariants do not change under deformation of the surface, these are well defined integers.)
\begin{cor}\label{cor:bup1-1}
  Let  $D=ae_0-bf_1\in \Pic(\Bl_{d}(\P^2_k))$, and let $k$ be a perfect field of characteristic either 0 or at least $a+1$. Then    one has
  \begin{align*}
      N_{\Bl_{d}(\P^2_k),D,\sigma_{sp}}=& W_{\P^2_{0,1},D,0}\lra{1} + \frac{W_{\P^2_{2,0},D,0}-W_{\P^2_{0,1},D,0}}{2}(\lra{2}+\lra{2d}) 
    +\frac{GW_{\P^2_2,D}-W_{\Bl_{2,0}\P^2,D,0}}{2} h.
  \end{align*}
\end{cor}

Combining Corollary \ref{cor:bup1-1} with  computations from \cite{ChenZinger21}, we obtain the values of $ N_{\Bl_{d}(\P^2_k),D,\sigma_{sp}}$ listed 
in Table \ref{tab:Pd}. Note that $N_{\Bl_{d}(\P^2_k),2be_0-bf_1,\sigma_{sp}}=N_{Q(d),bH,\sigma_{sp}}$, so we do not list these values in Table \ref{tab:Pd}.
\begin{table}[!h]
  \[
  {\setlength{\extrarowheight}{3pt}
  \begin{array}{|c|c|}
  \hline  (a,b) & N_{\Bl_{d}(\P^2_k),ae_0-bf_1,\sigma_{sp}}
    \\ \hhline{|=|=|}  (5,2) &  576\lra{1} +255(\lra{2}+\lra{2d})+1212h 
    \\ \hline (6,2) & 88992\lra{1}+ 70080(\lra{2}+\lra{2d}) +664560h 
    \\ \hline (7,2) & 22823424 \lra{1} +27315216(\lra{2}+\lra{2d})+515349192h 
    \\  \hline  (7,3)& 294336 \lra{1} + 262432(\lra{2}+\lra{2d}) + 2844720h
    \\ \hline
  \end{array}
  }
  \]
  \label{tab:Pd}
  \caption{Quadratic invariants of $\Bl_{d}(\P^2_k)$}
\end{table}

\medskip
          In the case when $b=1$, blowing down the exceptional divisors, Theorem \ref{thm:abq bup} gives the following. We denote by $\Bl(\P^2_k)$ the blow-up of $\P^2_k$ at a rational point, and by $e_1$ the class of the exceptional divisor. 
          \begin{cor}\label{cor:quad wel}
            For any $a\ge 1$, and any finite étale extension $k\to\sigma$ of degree $3d-3$,
            one has
            \begin{equation*}
              N_{\P^2_k,ae_0,\sigma_d\times\sigma}
               = N_{\P^2_k,ae_0 ,\sigma_1\times\sigma} -
              (\langle 2 \rangle - \langle 2d \rangle) N_{\Bl(\P^2_k),ae_0 -2e_1,\sigma}.
              \end{equation*}
          \end{cor}
          Corollary \ref{cor:quad wel} is a quadratic analogue of Welschinger formula \cite[Theorem 0.4]{Welschinger-invtsReal4mflds}, which has also been proposed independently   in \cite{JaramilloPuentesWall} based on tropical considerations. Corollary \ref{cor:quad wel} should generalize as follow to any $\A^1$-connected del Pezzo surface $S$ of degree at least 3. Let $D\in \Pic(S)$, and  $k\to\sigma$ a finite étale extension of degree $-K_S\cdot D-3$. Denoting by $\Bl(S)$ the blow-up of $S$ at a rational point, and by $e$ the exceptional divisor class, we expect the following formula to be true:
          \begin{equation*}
            N_{S,D,\sigma_d\times \sigma}
             = N_{S,D ,\sigma_1\times \sigma} -
            (\langle 2 \rangle - \langle 2d \rangle) N_{\Bl(S),D -2e,\sigma}.
            \end{equation*}
            An adaptation of the proof by Welschinger of  \cite[Theorem 0.4]{Welschinger-invtsReal4mflds} may provide a simpler strategy than the one used here to prove this latter formula.

          \subsection{Cubic surfaces}\label{section:enumerative_cubic_surfaces}

          Let $P_2(x,y,z)$ and $P_3(x,y,z)$ be two homogeneous polynomials in $k[x,y,z]$ defining a non-singular conic $C_2$ and cubic $C_3$ in $\P^2_k$, respectively. Assume furthermore that $C_2$ and $C_3$ intersect transversely, and consider the cubic surface $\cX\to \Spec k[[t]]$ in $\P^3_{k[[t]]}$ with equation
          \begin{equation}\label{eq:cubic}
          P_3(x,y,z) + P_2(x,y,z)w + tw^3. 
        \end{equation}
        It is a  $1$-nodal Lefschetz fibration of del Pezzo surfaces, and projecting from the point $[0:0:0:1]$ expresses  $\widetilde S$ as  the blow-up of $\P^2_k$ along the intersection points $q_1,\ldots, q_m$ of $C_2$ and $C_3$. Furthermore, the $(-2)$-curve $E$ is the strict transform in $\widetilde S$ of the conic $C_2$.
          We denote by $k_i$ the residual field of $q_i$,
          and by  $\Sigma(d)\to \Spec k((t))$ the generic fiber of a $d$-surgery of $\cX\to \Spec k[[t]]$. We have
          \[
          \sum_{i=1}^m[k_i:k]=6.
          \]
          Recall that a configuration of 6 distinct points in $\P^2_k$ is \emph{generic} if they are not contained in a conic, and  if no three of them lie on a line.
          
          \begin{lemma}\label{lem:bP26}
            The surface $\Sigma(1)$ is the blow-up of $\P^2_{k((t))}$ at a generic configuration of points with residue fields $k_1\otimes k((t)),\ldots, k_m\otimes k((t))$.
          \end{lemma}
          \begin{proof}
            Let $\widetilde E_i$ be the  exceptional divisor of $\widetilde S$ corresponding to $q_i$.  Choose one of the two rulings of $Q_1$, and denote by $C_i$ the rational curve in $Q_1$ in this ruling and passing through the point $\widetilde E_i\cap E$. 
            By Lemma \ref{lemma:unique_generalization}, each curve $\widetilde E_i\cup C_i$ deforms to a non-singular $(-1)$-rational curve $E_i$ in $\Sigma(1)$, with residue field $k_i\otimes k((t))$. Furthermore since all curves $\widetilde E_i\cup C_i$ are pairwise disjoint, so are the curves $E_i$.
            Contracting the curves $E_1,\ldots, E_m$ exhibit $\Sigma(1)$ as the desired blow-up of  $\P^2_{k((t))}$.  
          \end{proof}
          
          When $d\in  k^\times/(k^\times)^2$ is not a square, the surface $\Sigma(d)$ may or may not be rational or $\A^1$-connected. 
          \begin{exa}
            Suppose that all points of $C_2\cap C_3$ are rational, that is $\Sigma(1)$ is the blow-up of  $\P^2_{k((t))}$ at six $k((t))$-rational points. By Lemma \ref{lemma:unique_generalization},the strict transform in $\widetilde S$ of the 4 lines $(q_1q_2)$, $(q_1,q_3)$, $(q_1q_4)$, and $(q_1,q_5)$ deform to 4 pairwise disjoint non-singular $(-1)$-rational curve $E'_1,\ldots E'_4$ in $\Sigma(d)$, with residue field $k((t))$.
            Let $\widetilde E'_5$ be the union of twice the curve $\widetilde E_5\subset \widetilde S$ together with the curve $C_5'\subset Q_d$ such that $(C'_5)_{k[\sqrt d]}$ is the union of the two rulings of $Q(d)_{k[\sqrt d]}$  through $E_5\cap E$. Still by Lemma \ref{lem:bP26}, the curve  $\widetilde E'_5$ deforms to a non-singular exceptional rational curve $E'_5$ in $\Sigma(d)$ with residual field  $k((t))[\sqrt d]$. By constructions the curves $E'_1,\ldots, E'_5$ are pairwise distinct, and contracting them exhibits $\Sigma(d)$ as the blow-up of  $\P^2_{k((t))}$ at 4 points with residual field $k((t))$ and one point with residual field $k((t))[\sqrt d]$. 
          \end{exa}

          \begin{exa}\label{ex:enumerative_Sigma-1_not_connected}
            Cubic surface that are not $\A^1$-connected are known, for example real cubic surfaces in $\P^3_\R$ whose real part is not connected. When $k=\R$, the specialization of the parameter $t$ to  $a\in \R^*$ of a family $\Sigma_{-1}$ as above may give such real algebraic surfaces $X_a$: if $C_2\cap C_3$ consists of three points with residual field $\C$, then $X_a(\R)$ is not connected for $a$ small enough as explained in {\cite[Example 2.4]{Bru16}}.
            In this case, we suspect the whole family $\Sigma_{-1}$ to be not $\A^1$-connected.
          \end{exa}

          We denote by $e_i$ the class in $\Pic \Sigma(1)$ of the exceptional curve of $\Sigma(1)$ with residue field $k_i\otimes k((t))$, and by $e_0$ the pullback of class of the class of a line in $\P^2_{k((t))}$. It follows from Lemma \ref{lem:bP26} that
          \[
          \Pic \Sigma(1)\cong \Z e_0\oplus \Z e_1\oplus\cdots\oplus \Z e_m,
          \]
          and that the vanishing cycle is $\gamma =2e_0-e_1-\ldots -e_m\in \Pic \Sigma(1)$.
          Next Theorem is now a direct application of Theorem \ref{thm:N-surgery}.

          \begin{thm}
            Let  $D\in\gamma^\perp\subset \Pic \Sigma(1)$, and $k\to\sigma$ a finite étale extension of degree $-K_{\Sigma(1)}\cdot D- 1 $.
             Then for any $d\in  k^\times/(k^\times)^2$  such that $\Sigma(d)$ is $\A^1$-connected, one has
             \begin{align*}
             N_{\Sigma(d),D,\sigma_{k((t))}}
             & = N_{\Sigma(1),D,\sigma_{k((t))}}+
            (\langle 2 \rangle - \langle 2d \rangle) \sum_{j\ge 1} (-1)^j N_{\Sigma(1),D-j\gamma,\sigma_{k((t))}}.
           \end{align*} 
          \end{thm}
          It would be very interesting to understand the relations between the invariants  $N_{\Sigma(d),D,\sigma_{k((t))}}$ and the invariants  $N_{X_a,D,\sigma}$ of the cubic surface $X_a$ with equation \eqref{eq:cubic} when specializing $t$ to $a\in k^*$. 
          We address this question in \cite{BruRauWic25}.

          \subsection{Monodromy of $1$-nodal Lefschetz fibration and quadratic invariants}\label{sec:dehn}
          The following generalizes invariance of Gromov-Witten and Welschinger invariants under (real) Dehn twists, see \cite{Arn95} and \cite[Remark 1.3]{Bru16}. Note that if this invariance is straightforward in the framework of symplectic geometry, this is not the case within algebraic geometry. Our proof of next statement follows the line of the algebro-geometric proof of such invariance for complex and real  del Pezzo surfaces provided in \cite[Theorem 4.1 (1)]{IKS15}.

          \begin{thm}\label{thm:dehn}
            Let $\cX \to \Spec k[[t]]$ be a $1$-nodal Lefschetz fibration of del Pezzo surfaces of degree at least 2.  Let $\Sigma(1)$ be the generic fiber of the $1$-surgery, and  $D \in \Pic \Sigma(1)$ satisfying Hypothesis \ref{hyp:SDk_with_N}. 
            Then for any finite étale extension $k\to\sigma$ of degree $-K_{\Sigma(1)}\cdot D-1$,
             one has
            \[
                N_{\Sigma(1),D,\sigma}=N_{\Sigma(1),D+(D\cdot \gamma)\gamma,\sigma}
            \]
          \end{thm}
          \begin{proof}
            Let $\widetilde{p}=(\widetilde p_1,\ldots,\widetilde p_r)$ be a collection of closed points on $\Sigma(d)$ which specialize (as in Remark \ref{rem:spec_of_pi}) to $p=( p_1,\ldots,p_r)$ in the open subset of Hypothesis.
  Then by Theorem \ref{thm:N_degeneration2}
   we have the equality 
  \begin{align*}
   N_{\Sigma(1),D,\sigma}
   & = \sum_{\ell} Tr_{\ell/k} \big(\sum_{\substack{\sigma' \in \Fet_{/\ell} \\ (D_0,(i,j))\in \varphi_1^{-1}(D)}} {\sigma'/\ell \choose j}N^{\sigma',\ell}_{\widetilde S,D_0,\sigma}(p) \big). 
  \end{align*}

            Let us first prove that
            \[
            (D_0,(i,j))\in \varphi_1^{-1}(D)\Longleftrightarrow   (D_0,(j,i))\in \varphi_1^{-1}(D+(D\cdot \gamma)\gamma).
            \]
            Indeed, denoting by $\delta=(D_0,(i,j))$ and $\gamma_0=(0,(1,-1))$, one has
            \begin{align*}
                (D_0,(j,i))&=(D_0,(i,j))+ (j-i)(0,(1,-1))
                \\ &= \delta+ \left( \delta\cdot \gamma_0)\right)\gamma_0
            \end{align*}
            Since the map $\varphi_1$ preserves intersection products, we get
            \[
            \varphi_1((D_0,(j,i)))=\varphi_1(\delta)+ (\varphi_1(\delta)\cdot \gamma)\gamma
            \]
            as announced.Hence by Theorem \ref{thm:N_degeneration2} again, we have
            \begin{align*}
                N_{\Sigma(1),D+(D\cdot \gamma)\gamma,\sigma}
                & = \sum_{\ell} Tr_{\ell/k} \big(\sum_{\substack{\sigma' \in \Fet_{/\ell} \\ (D_0,(i,j))\in \varphi_1^{-1}(D)}} {\sigma'/\ell \choose i}N^{\sigma',\ell}_{\widetilde S,D_0,\sigma}(p) \big). 
               \end{align*}           
            From $|\sigma'|=i+j$ we get
               \[
                {\sigma'/k \choose i}={\sigma'/k \choose j},
               \]
               and the result follows from Theorem \ref{thm:N_degeneration2}.
          \end{proof}
          
\begin{exa}
In the situation of rational quadrics and $\P^2_k$ blown-up at two points of Sections \ref{sec:ex quad} and \ref{sec:ex bup}, Theorem \ref{thm:dehn} recovers the obvious relations (with the notation used in the corresponding section)
\[
  N_{Q(1),ae_1+be_2,\sigma}=  N_{Q(1),be_1+ae_2,\sigma}\quad\mbox{and}\quad  N_{\Bl_{1}(\P^2_k),ae_0 -ae_1 - be_2,\sigma}= N_{\Bl_{1}(\P^2_k),ae_0 -be_1 - ae_2,\sigma}.
\]
\end{exa}

\begin{exa}
  Theorem \ref{thm:dehn} also provides a fancy computation of $N_{\P^2_k,2e_0,\sigma}= \lra{1}$, where $e_0$ is the class of a line. Let $X$ be the blow-up of $\P^2_k$ at a generic configuration of points with residual fields given by $\sigma$, and at an extra rational point. We have the equality $N_{\P^2_k,2e_0,\sigma}=N_{X,2e_0-e,\emptyset}$, where $e$ is the sum of the classes realized by the exceptional divisors of the blow-ups corresponding to $\sigma$. Denoting by $e_1$ the class of the exceptional divisor from the additional blow-up, and applying Theorem \ref{thm:dehn} to the class $2e_0-e$ and the vanishing cycle $2e_0-e-e_1$ gives $N_{X,2e_0-e,\emptyset}=N_{X,e_1,\emptyset}=\lra{1}$.
\end{exa}

 \bibliographystyle{alpha}

 \bibliography{Bibli}

\newcommand{\etalchar}[1]{$^{#1}$}
\begin{thebibliography}{AMBO{\etalchar{+}}22}

\bibitem[AB01]{AbramovichBertram}
Dan Abramovich and Aaron Bertram.
\newblock The formula {{\(12=10+2\times 1\)}} and its generalizations: Counting
  rational curves on {{\(F_2\)}}.
\newblock In {\em Advances in algebraic geometry motivated by physics.
  Proceedings of the AMS special session on enumerative geometry in physics,
  University of Massachusetts, Lowell, MA, USA, April 1--2, 2000}, pages
  83--88. Providence, RI: American Mathematical Society (AMS), 2001.

\bibitem[AMBO{\etalchar{+}}22]{AMBOWZ22}
Niny Arcila-Maya, Candace Bethea, Morgan Opie, Kirsten Wickelgren, and Inna
  Zakharevich.
\newblock Compactly supported {{\(\mathbb{A}^1\)}}-{Euler} characteristic and
  the {Hochschild} complex.
\newblock {\em Topology Appl.}, 316:24, 2022.
\newblock Id/No 108108.

\bibitem[AO01]{Abramovich--Oort-mixed_char}
Dan Abramovich and Frans Oort.
\newblock Stable maps and {H}urwitz schemes in mixed characteristics.
\newblock In {\em Advances in algebraic geometry motivated by physics
  ({L}owell, {MA}, 2000)}, volume 276 of {\em Contemp. Math.}, pages 89--100.
  Amer. Math. Soc., Providence, RI, 2001.

\bibitem[Arn95]{Arn95}
V.~I. Arnol'd.
\newblock Some remarks on symplectic monodromy of {M}ilnor fibrations.
\newblock In {\em The {F}loer memorial volume}, volume 133 of {\em Progr.
  Math.}, pages 99--103. Birkh\"auser, Basel, 1995.

\bibitem[Azo25]{Azouri-thesis}
Ran Azouri.
\newblock Motivic {E}uler characteristic of nearby cycles and a generalised
  quadratic conductor formula.
\newblock {\em J. Algebra}, 667:109--164, 2025.

\bibitem[BM24]{Bejleri-McKean-symmetric_powers}
Dor Bejleri and Stephen McKean.
\newblock Symmetric powers of null motivic euler characteristic.
\newblock ArXiv 2406.19506, 2024.

\bibitem[BP15]{BruPui15}
Erwan Brugall{\'e} and Nicolas Puignau.
\newblock On {Welschinger} invariants of symplectic 4-manifolds.
\newblock {\em Comment. Math. Helv.}, 90(4):905--938, 2015.

\bibitem[Bru18]{Bru16}
Erwan Brugall\'e.
\newblock Surgery of real symplectic fourfolds and welschinger invariants.
\newblock {\em Journal of Singularities}, 17:267--294, 2018.

\bibitem[Bru20]{Bru18}
Erwan Brugall\'e.
\newblock On the invariance of {W}elschinger invariants.
\newblock {\em Algebra i Analiz}, 32(2):1--20, 2020.

\bibitem[BRW24]{BRWunp}
Erwan Brugallé, Johannes Rau, and Kirsten Wickelgren.
\newblock Unpublished.
\newblock 2024.

\bibitem[BRW25]{BruRauWic25}
Erwan Brugallé, Johannes Rau, and Kirsten Wickelgren.
\newblock Welschinger-{W}itt invariants.
\newblock arXiv:2509.04172, 2025.

\bibitem[BS26]{BrazeltonSpitz}
Thomas Brazelton and Ben Spitz.
\newblock Operations on {G}-sets and the {D}ress map.
\newblock In preparation, 2026.

\bibitem[BV25]{BroringViergeverSymPower}
Lukas~F. Br\"oring and Anna~M. Viergever.
\newblock Quadratic {E}uler characteristic of symmetric powers of curves.
\newblock {\em Manuscripta Math.}, 176(2):Paper No. 26, 24, 2025.

\bibitem[BW23]{bachmann_wickelgren}
Tom Bachmann and Kirsten Wickelgren.
\newblock Euler classes: six-functors formalism, dualities, integrality and
  linear subspaces of complete intersections.
\newblock {\em J. Inst. Math. Jussieu}, 22(2):681--746, 2023.

\bibitem[CW26]{Chen-Wickelgren}
Chongyao Chen and Kirsten Wickelgren.
\newblock Quadratically enriched binomial coefficients over a finite field.
\newblock In {\em Regulators {V}}, volume 842 of {\em Contemp. Math.}, pages
  77--104. Amer. Math. Soc., Providence, RI, 2026.

\bibitem[CZ21]{ChenZinger21}
Xujia Chen and Aleksey Zinger.
\newblock {WDVV}-type relations for {Welschinger} invariants: applications.
\newblock {\em Kyoto J. Math.}, 61(2):339--376, 2021.

\bibitem[DL01]{DenLoe01}
Jan Denef and François Loeser.
\newblock Geometry on arc spaces of algebraic varieties.
\newblock In {\em European {C}ongress of {M}athematics, {V}ol. {I}
  ({B}arcelona, 2000)}, volume 201 of {\em Progr. Math.}, pages 327--348.
  Birkh\"{a}user, Basel, 2001.

\bibitem[Dol12]{Dolg12}
Igor~V. Dolgachev.
\newblock {\em Classical algebraic geometry. {A} modern view}.
\newblock Cambridge: Cambridge University Press, 2012.

\bibitem[Fri83]{Friedman83}
Robert Friedman.
\newblock Global smoothings of varieties with normal crossings.
\newblock {\em Ann. Math. (2)}, 118:75--114, 1983.

\bibitem[Ful98]{fulton98}
William Fulton.
\newblock {\em Intersection theory}, volume~2 of {\em Ergebnisse der Mathematik
  und ihrer Grenzgebiete. 3. Folge. A Series of Modern Surveys in Mathematics
  [Results in Mathematics and Related Areas. 3rd Series. A Series of Modern
  Surveys in Mathematics]}.
\newblock Springer-Verlag, Berlin, second edition, 1998.

\bibitem[GMS03]{Garibaldi-Serre-Merkurjev}
Skip Garibaldi, Alexander Merkurjev, and Jean-Pierre Serre.
\newblock {\em Cohomological invariants in {G}alois cohomology}, volume~28 of
  {\em University Lecture Series}.
\newblock American Mathematical Society, Providence, RI, 2003.

\bibitem[Har77]{Hartshorne}
Robin Hartshorne.
\newblock {\em Algebraic geometry}.
\newblock Springer-Verlag, New York-Heidelberg, 1977.
\newblock Graduate Texts in Mathematics, No. 52.

\bibitem[Hoy14]{Hoyois_lef}
Marc Hoyois.
\newblock A quadratic refinement of the {G}rothendieck-{L}efschetz-{V}erdier
  trace formula.
\newblock {\em Algebr. Geom. Topol.}, 14(6):3603--3658, 2014.

\bibitem[IKS15]{IKS15}
Ilia Itenberg, Viatcheslav Kharlamov, and Eugenii Shustin.
\newblock Welschinger invariants of real del {Pezzo} surfaces of degree
  {{\(\geq 2\)}}.
\newblock {\em Int. J. Math.}, 26(8):63, 2015.
\newblock Id/No 1550060.

\bibitem[IKS17]{IKS17}
Ilia Itenberg, Viatcheslav Kharlamov, and Eugenii Shustin.
\newblock Welschinger invariants revisited.
\newblock In {\em Analysis meets geometry. The Mikael Passare memorial volume},
  pages 239--260. Cham: Birkh{\"a}user/Springer, 2017.

\bibitem[IKS18]{ItenbergKharlamovShustin-RelativeEnumerative}
Ilia Itenberg, Viatcheslav Kharlamov, and Eugenii Shustin.
\newblock Relative enumerative invariants of real nodal del {P}ezzo surfaces.
\newblock {\em Selecta Math. (N.S.)}, 24(4):2927--2990, 2018.

\bibitem[Ill71]{Illusie-Complexe_cotangent_deformationsI}
Luc Illusie.
\newblock {\em Complexe cotangent et d\'{e}formations. {I}}, volume Vol. 239 of
  {\em Lecture Notes in Mathematics}.
\newblock Springer-Verlag, Berlin-New York, 1971.

\bibitem[Ill05]{IllusieFGA}
Luc Illusie.
\newblock Grothendieck's existence theorem in formal geometry.
\newblock In {\em Fundamental algebraic geometry, Grothendieck's FGA
  Explained}, volume 123 of {\em Math. Surveys Monogr.}, pages 179--233. Amer.
  Math. Soc., Providence, RI, 2005.
\newblock With a letter (in French) of Jean-Pierre Serre.

\bibitem[JP24]{JaramilloPuentesWall}
Andrés Jaramillo~Puentes.
\newblock A wall crossing formula for motivic enumerative invariants.
\newblock Preprint, available at \url{https://arxiv.org/abs/2403.17681}, 2024.

\bibitem[JPMPR25]{JPMPR25}
Andrés Jaramillo~Puentes, Hannag Markwig, Sabrina Pauli, and Felix Röhrle.
\newblock Quadratically enriched plane curve counting via tropical geometry.
\newblock {\em Preprint}, available at \url{arXiv:2502.02569}, 2025.

\bibitem[JPP25]{PauliPuentes23}
Anr\'es Jaramillo~Puentes and Sabrina Pauli.
\newblock A quadratically enriched correspondence theorem.
\newblock Journal of the European Mathematical Society, 2025.

\bibitem[KKMSD73]{KKMS73}
G.~Kempf, F.~Knudsen, D.~Mumford, and Bernard Saint-Donat.
\newblock {\em Toroidal embeddings. {I}}, volume 339 of {\em Lect. Notes Math.}
\newblock Springer, Cham, 1973.

\bibitem[Kle05]{Kleinman_The_Picard_Scheme}
Steven~L. Kleiman.
\newblock The {P}icard scheme.
\newblock In {\em Fundamental algebraic geometry}, volume 123 of {\em Math.
  Surveys Monogr.}, pages 235--321. Amer. Math. Soc., Providence, RI, 2005.

\bibitem[KLSW23]{degree}
Jesse~Leo Kass, Marc Levine, Jake~P. Solomon, and Kirsten Wickelgren.
\newblock A quadratically enriched count of rational curves.
\newblock {\em Preprint} ArXiv 2307.01936, 2023.

\bibitem[KLSW25]{KLSW-relor}
Jesse~Leo Kass, Marc Levine, Jake~P. Solomon, and Kirsten Wickelgren.
\newblock A relative orientation for the moduli space of stable maps to a del
  pezzo surface.
\newblock accepted in Algebraic Geometry, 2025.

\bibitem[Kol13]{Kollar-Singularities-minimal-model-program}
J\'anos Koll\'ar.
\newblock {\em Singularities of the minimal model program}, volume 200 of {\em
  Cambridge Tracts in Mathematics}.
\newblock Cambridge University Press, Cambridge, 2013.
\newblock With a collaboration of S\'andor Kov\'acs.

\bibitem[Lev18]{Levine-Welschinger}
Marc Levine.
\newblock Toward an algebraic theory of {W}elschinger invariants.
\newblock {\em Preprint}, available at \url{https://arxiv.org/abs/1808.02238},
  2018.

\bibitem[Lev20]{Levine-EC}
Marc Levine.
\newblock Aspects of enumerative geometry with quadratic forms.
\newblock {\em Doc. Math.}, 25:2179--2239, 2020.

\bibitem[LR20]{LevineRaksit_MotivicGaussBonnet}
Marc Levine and Arpon Raksit.
\newblock Motivic {G}auss-{B}onnet formulas.
\newblock {\em Algebra Number Theory}, 14(7):1801--1851, 2020.

\bibitem[MH73]{milnor73}
John Milnor and Dale Husemoller.
\newblock {\em Symmetric bilinear forms}.
\newblock Springer-Verlag, New York-Heidelberg, 1973.
\newblock Ergebnisse der Mathematik und ihrer Grenzgebiete, Band 73.

\bibitem[Mik05]{Mikhalkin05}
Grigory Mikhalkin.
\newblock Enumerative tropical algebraic geometry in {{\(\mathbb{R}^2\)}}.
\newblock {\em J. Am. Math. Soc.}, 18(2):313--377, 2005.

\bibitem[Pfi66]{Pfister66}
Albrecht Pfister.
\newblock Quadratic forms over arbitrary fields.
\newblock {\em Invent. Math.}, 1:116--132, 1966.

\bibitem[Poo17]{Poonen-rational_points_on_varieties}
Bjorn Poonen.
\newblock {\em Rational points on varieties}, volume 186 of {\em Graduate
  Studies in Mathematics}.
\newblock American Mathematical Society, Providence, RI, 2017.

\bibitem[PP25]{PajwaniPal-PowerStructures}
Jesse Pajwani and Ambrus P\'al.
\newblock Power structures on the {G}rothendieck-{W}itt ring and the motivic
  {E}uler characteristic.
\newblock {\em Ann. K-Theory}, 10(2):123--152, 2025.

\bibitem[PS08]{PetSte08}
Chris A.~M. Peters and Joseph H.~M. Steenbrink.
\newblock {\em Mixed {H}odge structures}, volume~52 of {\em Ergebnisse der
  Mathematik und ihrer Grenzgebiete. 3. Folge. A Series of Modern Surveys in
  Mathematics [Results in Mathematics and Related Areas. 3rd Series. A Series
  of Modern Surveys in Mathematics]}.
\newblock Springer-Verlag, Berlin, 2008.

\bibitem[sga03]{sga1}
{\em Rev\^etements \'etales et groupe fondamental ({SGA} 1)}.
\newblock Documents Math\'ematiques (Paris) [Mathematical Documents (Paris)],
  3. Soci\'et\'e Math\'ematique de France, Paris, 2003.
\newblock S{\'e}minaire de g{\'e}om{\'e}trie alg{\'e}brique du Bois Marie
  1960--61. [Algebraic Geometry Seminar of Bois Marie 1960-61], Directed by A.
  Grothendieck, With two papers by M. Raynaud, Updated and annotated reprint of
  the 1971 original [Lecture Notes in Math., 224, Springer, Berlin; MR0354651
  (50 \#7129)].

\bibitem[{Sta}18]{stacks-project}
The {Stacks Project Authors}.
\newblock \textit{Stacks Project}.
\newblock \url{https://stacks.math.columbia.edu}, 2018.

\bibitem[Tyo12]{Tyomkin12}
Ilya Tyomkin.
\newblock Tropical geometry and correspondence theorems via toric stacks.
\newblock {\em Math. Ann.}, 353(3):945--995, 2012.

\bibitem[Vak00]{Vakil-curves_rational_surfaces}
Ravi Vakil.
\newblock Counting curves on rational surfaces.
\newblock {\em Manuscripta Math.}, 102(1):53--84, 2000.

\bibitem[Vak25]{Vakil-RisingSea}
Ravi Vakil.
\newblock {\em The Rising Sea: Foundations of Algebraic Geometry}.
\newblock Princeton University Press, 2025.

\bibitem[Wel05]{Welschinger-invtsReal4mflds}
Jean-Yves Welschinger.
\newblock Invariants of real symplectic 4-manifolds and lower bounds in real
  enumerative geometry.
\newblock {\em Invent. Math.}, 162(1):195--234, 2005.

\bibitem[Wel15]{Welopendim4}
Jean-Yves Welschinger.
\newblock Open {Gromov}-{Witten} invariants in dimension four.
\newblock {\em J. Symplectic Geom.}, 13(4):1075--1100, 2015.

\end{thebibliography}

\end{document}